\documentclass[11pt,reqno]{amsproc}
\usepackage[semicolon,square,authoryear]{natbib}
\numberwithin{equation}{section}
\usepackage{cite}
\usepackage[usenames,dvipsnames]{xcolor}
\usepackage{graphicx}
\usepackage{wrap fig,lip sum,booktabs}
\usepackage{caption}
\usepackage{fullpage}
\usepackage{epsfig,graphics,subfigure,psfrag,amssymb}
\usepackage{enumerate}
\usepackage{morefloats}
\newtheorem{theorem}{Theorem}

\newtheorem{remark}{Remark}

\newlength{\drop}
\definecolor{burgundy}{rgb}{0.5, 0.0, 0.13}
\title{Mechanics-based solution 
verification for porous media models
}
\author{\textbf{M.~Shabouei} and \textbf{K.~B.~Nakshatrala} \\ 
{\small Department of Civil \& Environmental Engineering, 
University of Houston. \\
Correspondence to:~\textsf{knakshatrala@uh.edu}}}
\begin{document}


\begin{titlepage}
    \drop=0.1\textheight
    \centering
    \vspace*{\baselineskip}
    \rule{\textwidth}{1.6pt}\vspace*{-\baselineskip}\vspace*{2pt}
    \rule{\textwidth}{0.4pt}\\[\baselineskip]
    {\LARGE \textbf{\color{burgundy} Mechanics-based 
    solution verification for porous media models}}\\[0.3\baselineskip]
    \rule{\textwidth}{0.4pt}\vspace*{-\baselineskip}\vspace{3.2pt}
    \rule{\textwidth}{1.6pt}\\[\baselineskip]
    \scshape
    An e-print of the paper is available on arXiv:~1501.02581. \par
    
    \vspace*{1\baselineskip}
    Authored by \\[\baselineskip]
    
    {\Large M.~Shabouei\par}
    {\itshape Graduate Student, University of Houston.}\\[\baselineskip]
    
    {\Large K.~B.~Nakshatrala\par}
    {\itshape Department of Civil \& Environmental Engineering \\
    University of Houston, Houston, Texas 77204--4003. \\ 
    \textbf{phone:} +1-713-743-4418, \textbf{e-mail:} knakshatrala@uh.edu \\
    \textbf{website:} http://www.cive.uh.edu/faculty/nakshatrala\par}
    \vspace*{\baselineskip}
    \begin{figure*}[h]
      \subfigure{
        \includegraphics[scale=.3,clip]{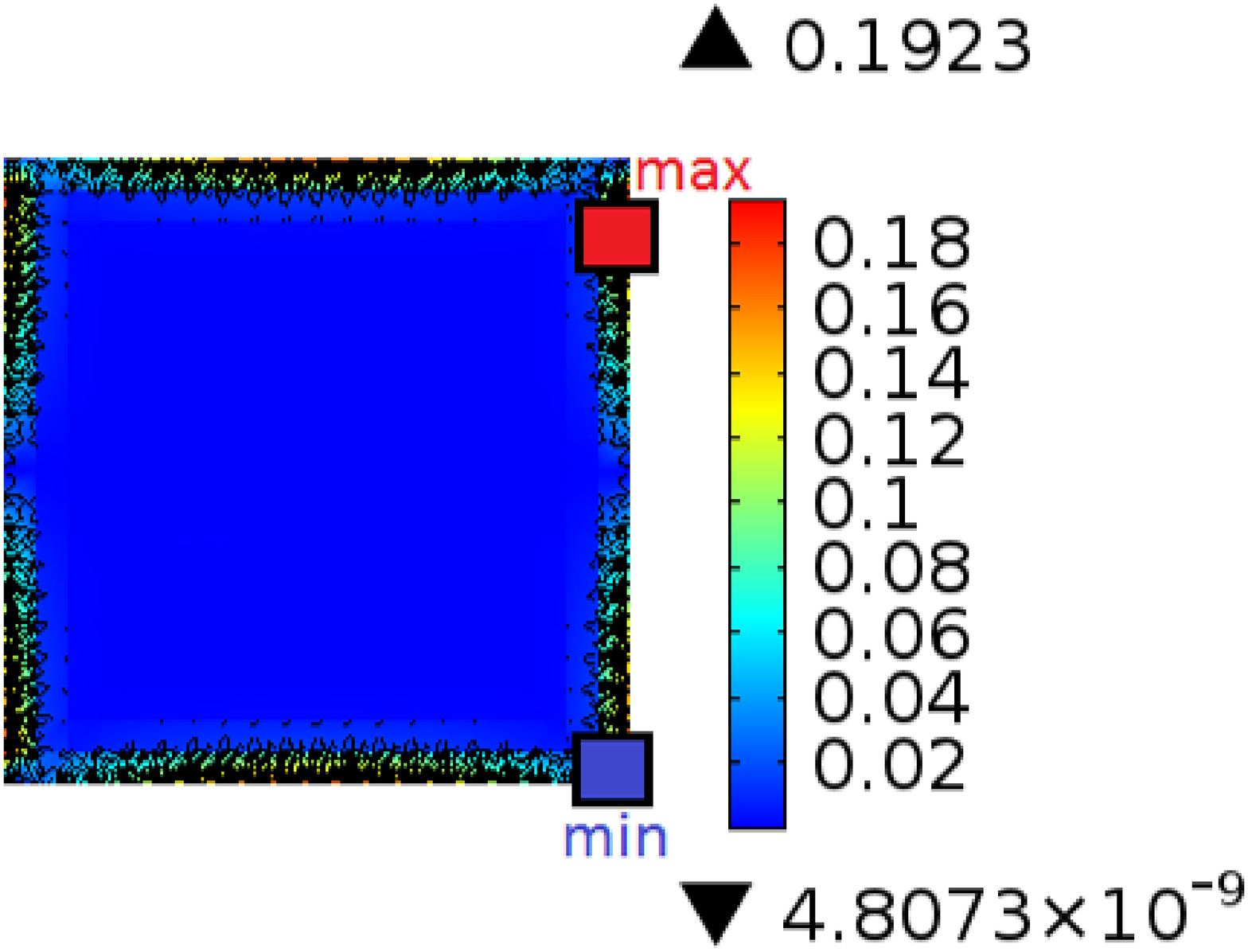}}
      \hspace{1 cm}
      \subfigure{
        \includegraphics[scale=.292,clip]{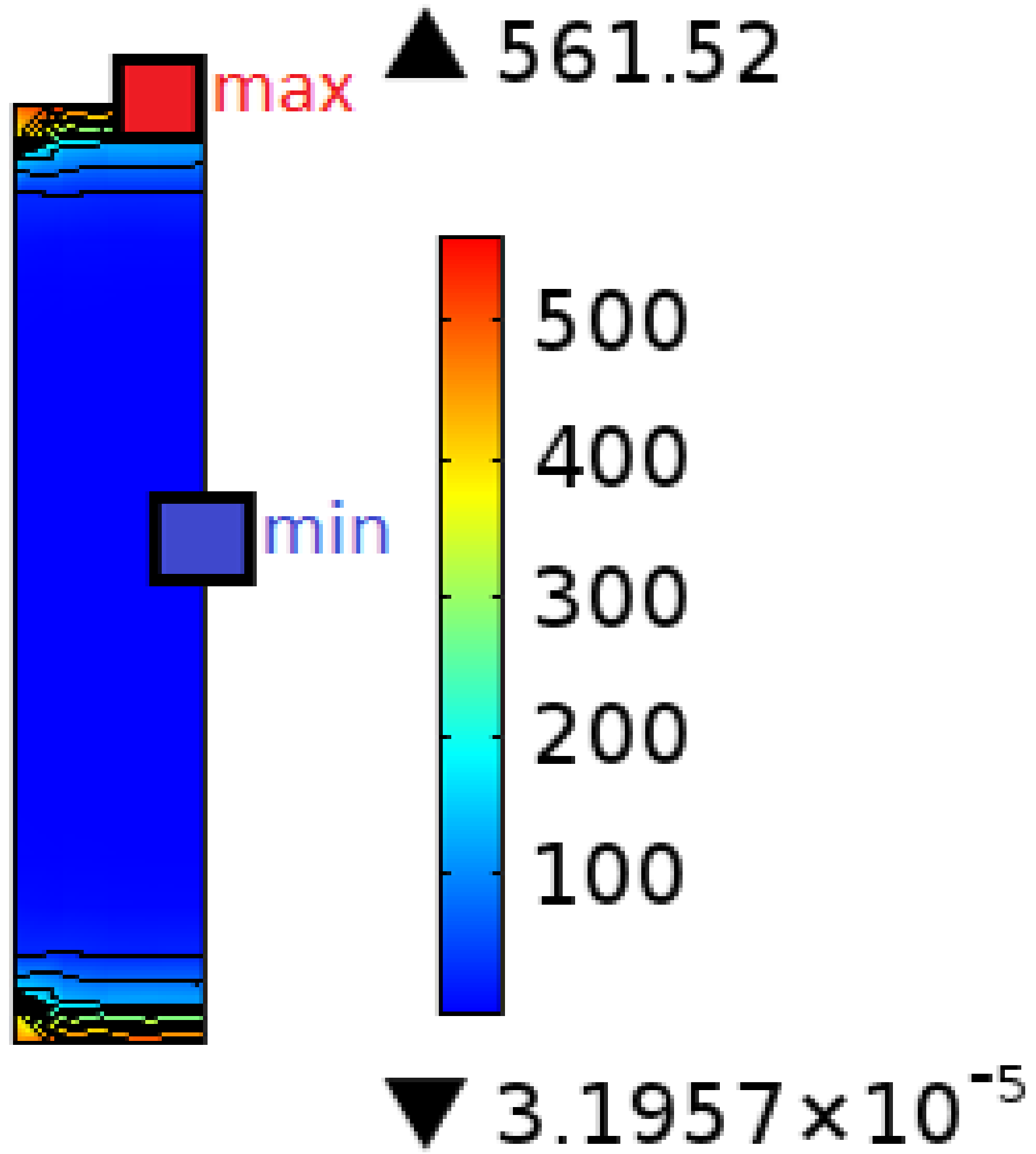}}
    \end{figure*}
        {\small \emph{This figure illustrates that the 
            vorticity under Darcy-Brinkman equations 
	    satisfies the classical maximum principle, 
	    which can serve as a good \emph{a posteriori} 
            criterion.}}
    
    \vfill
    {\scshape 2016} \\
    {\small Computational \& Applied Mechanics Laboratory} \par
  \end{titlepage}

\begin{abstract}
  This paper presents 
  a new approach to verify accuracy of 
  computational simulations. We develop 
  mathematical theorems which can serve 
  as robust \emph{a posteriori} error 
  estimation techniques to identify 
  numerical pollution, check the performance 
  of adaptive meshes, and verify numerical 
  solutions. We demonstrate performance of 
  this methodology on problems from flow 
  thorough porous media. However, one can 
  extend it to other models. We construct 
  mathematical properties such that the 
  solutions to Darcy and Darcy-Brinkman 
  equations satisfy them. 
  The mathematical properties include the 
  total minimum mechanical power, minimum 
  dissipation theorem, reciprocal relation, 
  and maximum principle for the vorticity. 
  All the developed theorems 
  have firm mechanical bases and are 
  independent of numerical methods. So, these 
  can be utilized for solution verification 
  of finite element, finite volume, finite 
  difference, lattice Boltzmann methods and 
  so forth.
  In particular, we show that, for a given set 
  of boundary conditions, Darcy velocity has 
  the minimum total mechanical power of all 
  the kinematically admissible vector fields. 
  We also show that a similar result holds 
  for Darcy-Brinkman velocity. 
  We then show for a conservative body force, 
  the Darcy and Darcy-Brinkman velocities 
  have the minimum total dissipation among their 
  respective kinematically admissible vector fields. 
  Using numerical examples, we show that the minimum 
  dissipation and total mechanical power theorems 
  can be utilized to identify 
  pollution errors in numerical solutions. 
  The solutions to Darcy and Darcy-Brinkman 
  equations are shown to satisfy a reciprocal 
  relation, which has the potential to identify 
  errors in the numerical implementation of 
  boundary conditions. 
  It is also shown that the vorticity under both 
  steady and transient Darcy-Brinkman equations 
  satisfy maximum principles if the body force 
  is conservative and the permeability 
  is homogeneous and isotropic. A discussion on 
  the nature of vorticity under steady and transient 
  Darcy equations is also presented. 
  Using several numerical examples, we will demonstrate 
  the predictive capabilities of the proposed \emph{a 
  posteriori} techniques in assessing the accuracy 
  of numerical solutions for a general class of problems, 
  which could involve complex domains and general 
  computational grids. 
\end{abstract}

\maketitle


\section{INTRODUCTION}
\label{Sec:MDarcy_Introduction}
%
\subsection{Validation and verification (V\&V)}
Errors can arise in both physical modeling and 
numerical simulation. The study of errors due 
to physical modeling is referred to as validation, 
and the study of error in a numerical simulation 
is referred to as verification. As Blottner 
\citep{Blottner_1990_v27_p113_122_J_spc_rock} 
nicely puts it, validation is to solve ``\emph{right 
governing equations}'' and verification is 
to solve ``\emph{governing equation right}''.
Validation errors arise when a model is used 
out of its application range. 
%
The errors in the verification, on the other 
hand, can arise from three broad sources including 
numerical errors, round-off errors (due to the 
finite precision arithmetic), and programming 
mistakes \citep{WL_Oberkampf_VV_v57_p345_2004_Appl_mech_rev}.
Basically, the verification is to ensure that 
the code produces a solution with some degree 
of accuracy, and the numerical solution is 
consistent. Verification itself is conducted 
into two modes: verification of code and 
verification of calculation 
\citep{Roache_Verification_and_Validation, 
rider_2015_Verification_JCP}. 
%
Verification of code addresses the question 
of whether the numerical algorithms have 
been programmed and implemented correctly 
in the code. The two currently popular 
approaches to verify a code are the 
\emph{method of exact solutions} (MES) and 
the \emph{method of manufactured solutions} 
(MMS). 
%
%
More thorough discussions on MES and MMS can be 
found in \citep{Verification_computer_codes_Knupp_2003, 
CJ_Roy_IJNMF_v2004_p599, 
Roache_Verification_and_Validation}.

Verification of calculation (which is also 
referred to as solution verification) estimates 
the overall magnitude (not just the order) 
of the numerical errors in a calculation, 
and the procedure invariably involves 
\emph{a posteriori} error estimation 
\citep{Code_Verification_by_MMS_SANDIA}. 
The numerical errors in the solution 
verification can arise from two different 
sources including discretization errors 
and solution errors. The discretization 
errors refer to all the errors caused by 
conversion of the governing equations 
(PDEs and boundary conditions) into discrete 
algebraic equations whereas the solution 
errors refer to the errors in approximate 
solution of the discrete equations. The 
numerical errors may arise from insufficient 
mesh resolution, improper selection of 
time-step, and incomplete iterative 
convergence. For more details on 
verification of calculation, see 
\citep{CJ_Roy_IJNMF_v2004_p599, 
Code_Verification_by_MMS_SANDIA,
Roache_uncertainty_CFD_AnnRev,
Roache_Verification_and_Validation, 
I_Babuska_VV_v193_p4057_2004_CMAME,
WL_Oberkampf_AIAA_v36_p687, 
WL_Oberkampf_VV_v57_p345_2004_Appl_mech_rev, 
rider_2015_Verification_JCP}.

%
\subsection{\emph{A posteriori} techniques}
The aim of \emph{a posteriori} error estimation 
is to assess the accuracy of the numerical 
approximation in the terms of \emph{known} 
quantities such as geometrical properties 
of computational grid, the input data, and 
the numerical solution. \emph{A posteriori} 
error techniques monitor various forms of 
the error in the numerical solution such 
as velocity, stress, mean fluxes, and drag 
and lift coefficients 
\citep{Becker_e_esti_2001_Acta_Num}. 
Such error estimation differ from \emph{a priori} 
error estimates in that the error controlling 
parameters depend on \emph{unknown} quantities. 
\emph{A priori} error estimation investigates 
the stability and convergence of a solver and 
can give rough information on the asymptotic 
behavior of errors in calculations when grid 
parameters are changed appropriately 
\citep{Ainsworth_e_esti_FEM_1997_CMAME}.

Moreover, discontinuities 
and singularities commonly occur in fluid 
dynamics, solid mechanics, and structural 
dynamics and it is a subject of intense 
research and grave concern in error estimation 
\citep{WL_Oberkampf_VV_v57_p345_2004_Appl_mech_rev}. 
Some representative works on pollution errors 
due to singularities and discontinuities are 
\citep{Babuska_Pollution_problem_v3_p553_1987_CANM,
Babuska_Pollution_adaptive_control_v38_p4207_1995_IJNME,
Babuska_error_estimation_v40_p3883_1997_IJNME,
Oden_Pollution_error_v27_p33_1998_IJNMF,
Roache_Verification_and_Validation,
Botella_singular_solutions_v36_p125_2001_IJNMF}. 
%
%

The question that still remains is how to 
verify the accuracy of numerical solutions, 
especially for realistic problems. 
%
%
Below are some specific challenges  that a 
computational scientist may face in using 
numerical simulators:
\begin{enumerate}[(i)]
\item How much mesh refinement is required 
to obtain solutions with desired degree of 
accuracy for a problem that does not have 
an analytical or reference solution?
\item Will the chosen mesh be able to resolve 
  singularities in the solution and avoid 
  pollution errors? That is, 
  can we identify whether a particular type 
  of mesh suffers from pollution errors for 
  a problem with singularities?
\item Has the computer implementation been 
done properly?
\item Is the chosen numerical formulation 
accurate/appropriate for the chosen problem?
\end{enumerate}
In the literature, one finds the usual 
approach of employing tools from functional 
analysis to obtain \emph{a priori} estimates 
and assess stability. For example, see 
\citep{Brezzi_Fortin,I_Babuska_FEM_2001}. 
On the other hand, this paper aims to address 
the aforementioned challenges by providing 
various \emph{a posteriori} techniques with 
firm mechanics underpinning. 
%
%

In the current study, we develop new methodology 
to verify the accuracy and convergence of 
numerical solutions. In addition, the presented 
\emph{a posteriori} error estimation approach 
is able to identify pollution in computational 
domains and check whether adaptive grids can 
resolve the numerical pollution. The proposed 
criteria is independent of employed numerical 
method. It means, our technique is able to 
verify the solution of all numerical methods 
including finite element, finite volume, finite 
difference and lattice Boltzmann. Hence, we 
named it mechanics-based solution verification. 
Then, we show performance of this powerful tool 
on problems from flow thorough porous solid. 
But, it is not restricted to porous media models 
and it is extendable to other problems. 


%
\subsection{Popular porous media models}
%
%
The simplest and yet the most popular model 
that describes the flow of an incompressible 
fluid in a rigid porous media is the Darcy 
model \citep{Darcy_1856}. 
%
%
The Darcy equations posed several challenges 
to the computational community but played a 
crucial role in the development of mixed and 
stabilized formulations in finite element method 
(FEM) \citep{Masoud_mixed_Darcy_v191_p4341_2002_CMAME,
Nakshatrala_mixed_Darcy_v195_p4036_2006_CMAME}. 
%
%
Brinkman \citep{Brinkman_ASR_1947_vA1_p27, 
Brinkman_ASR_1947_vA1_p81} proposed an 
extension to the Darcy model which is commonly 
referred to as the Darcy-Brinkman model. 
In addition to drag between the fluid and 
porous solid, Darcy-Brinkman model accounts 
for viscous term (friction 
between fluid layers). 
%
It is, however, not possible to obtain 
analytical solutions under these models, and 
one commonly seeks numerical solutions for 
realistic problems. 
%

Since the aim of this paper is \emph{a 
posteriori} error estimation, we assume 
that the code has already been verified 
for the Darcy and Darcy-Brinkman class of 
problems so that programming mistakes are 
not an issue. Likewise, we are not also 
concerned with validation. It means that 
the Darcy and Darcy-Brinkman models are 
physically adequate to model the problems.
Moreover, the solution verification requires 
confirmation of grid convergence which is 
one of the most common and reliable error 
estimation methods 
\citep{Roache_uncertainty_CFD_AnnRev}. 
Similar to grid convergence studies, we 
address only solution verification. 
\subsection{Organization of the paper}
Section \ref{Sec:S2_Brinkman_GE} presents 
the governing equations arising from the 
Darcy and Darcy-Brinkman models. In Section 
\ref{Sec:S3_Brinkman_Theorems}, we propose 
various mathematical properties that the 
solutions to these governing equations 
satisfy. We also discuss how these 
properties can be utilized as robust 
\emph{a posteriori} criteria to assess 
the accuracy of numerical solutions.
Section \ref{Sec:S4_Brinkman_NR} presents 
several  steady-state numerical results 
to illustrate the predictive capabilities 
of the proposed \emph{a posteriori} criteria 
with respect to singularities, pollution 
errors, and discretization errors in the 
implementation of (Neumann) boundary 
conditions.
In Section \ref{Sec:S5_Brinkman_Marmousi}, 
we utilize synthetic reservoir data to 
demonstrate the usefulness of the proposed 
\emph{a posteriori} techniques, especially 
for problems involving spatially heterogeneous 
permeability properties.  
Section \ref{Sec:S6_Brinkman_Transient} 
discusses \emph{a posteriori} criteria 
for transient problems, and presents 
representative numerical results in 
support of the theoretical predictions. 
Finally, conclusions are drawn in Section 
\ref{Sec:S8_Brinkman_Closure}.
\section{DARCY AND DARCY-BRINKMAN MODELS}
\label{Sec:S2_Brinkman_GE}
Let $\Omega \subset \mathbb{R}^{nd}$ be an open and 
bounded domain, where ``$nd$'' denotes the number of 
spatial dimensions. We shall denote the set closure 
of $\Omega$ by $\overline{\Omega}$. Let $\partial 
\Omega := \overline{\Omega} - \Omega$ denote the 
boundary, which is assumed to be piecewise smooth. 
A spatial point in $\overline{\Omega}$ is denoted 
by $\mathbf{x}$. The spatial gradient and divergence 
operators are, respectively, denoted as $\mathrm{grad}
[\cdot]$ and $\mathrm{div}[\cdot]$. Let $\mathbf{v}:~\Omega 
\rightarrow \mathbb{R}^{nd}$ denote the velocity field and 
$p : \Omega \rightarrow \mathbb{R}$ denote the pressure 
field.
The symmetric part of the gradient of velocity is 
denoted by $\mathbf{D}(\mathbf{x})$. That is, 
\begin{align}
  \mathbf{D}(\mathbf{x}) = \frac{1}{2} \left(\mathrm{grad}
         [\mathbf{v}] + \mathrm{grad}[\mathbf{v}]^{\mathrm{T}}
         \right)
\end{align}
The unit outward normal to the boundary is denoted 
as $\widehat{\mathbf{n}}(\mathbf{x})$. The boundary 
is divided into two parts: $\Gamma^{v}$ and $\Gamma^{t}$. 
$\Gamma^{v}$ is the part of the boundary on which the 
velocity is prescribed, and $\Gamma^{t}$ is that part 
of the boundary on which the traction is prescribed. 
For mathematical well-posedness, we have $\Gamma^{v} 
\cap \Gamma^{t} = \emptyset$ and $\Gamma^{v} \cup 
\Gamma^{t} = \partial \Omega$. 

The porous media models that will be considered in 
this paper are the Darcy  and Darcy-Brinkman models. 
Both these models describe the flow of an incompressible 
fluid through \emph{rigid} porous media. We completely 
neglect the motion of the porous solid. The Cauchy stress 
in the Darcy and Darcy-Brinkman models, respectively, take 
the following form:
\begin{subequations}
  \label{Eqn:Cauchy_stresses}
  \begin{align} 
    \label{Eqn:Darcy_Cauchy_stress}
    &\mathbf{T}(\mathbf{x}) = - p(\mathbf{x})\mathbf{I} \\
    \label{Eqn:Brinkman_Cauchy_stress}
    &\mathbf{T}(\mathbf{x}) = - p(\mathbf{x}) 
    \mathbf{I} + 2 \mu \mathbf{D}(\mathbf{x})
  \end{align}
\end{subequations}
where $\mathbf{I}$ denotes the second-order 
identity tensor, and $\mu$ is the dynamic 
coefficient of viscosity. The steady-state 
governing equations based on the Darcy model 
can be written as follows:
\begin{subequations}
  \label{Eqn:Darcy_GE}
  \begin{alignat}{2}
    \label{Eqn:Darcy_LM}
    &\alpha (\mathbf{x}) \mathbf{v}(\mathbf{x}) 
    + \mathrm{grad}[p(\mathbf{x})]  
    = \rho \mathbf{b}(\mathbf{x}) \quad &&\mbox{in} 
    \; \Omega \\    
    \label{Eqn:Darcy_Continuity}
    &\mathrm{div}[\mathbf{v}(\mathbf{x})] = 0 \quad 
    &&\mbox{in} \; \Omega \\
    \label{Eqn:Darcy_velocity_BC}
    &\mathbf{v}(\mathbf{x}) \cdot 
    \widehat{\mathbf{n}}(\mathbf{x}) 
    = v_n(\mathbf{x}) 
    \quad &&\mbox{on} \; \Gamma^{v} \\
    \label{Eqn:Darcy_traction_BC}
    &p(\mathbf{x}) = p_0(\mathbf{x}) 
    \quad &&\mbox{on} \; \Gamma^{t} 
  \end{alignat}
\end{subequations}
where $\alpha(\mathbf{x})$ is the drag coefficient, $\rho$ 
is the density of the fluid, $\mathbf{b}(\mathbf{x})$ is 
the specific body force, $v_n(\mathbf{x})$ is the prescribed 
normal component of the velocity, and $p_0(\mathbf{x})$ is 
the prescribed pressure.
The steady-state governing equations based on the 
Darcy-Brinkman model take the following form: 
\begin{subequations}
  \label{Eqn:Brinkman_GE}
  \begin{alignat}{2}
    \label{Eqn:Brinkman_LM}
    &\alpha (\mathbf{x}) \mathbf{v}(\mathbf{x}) 
    + \mathrm{grad}[p(\mathbf{x})] - \mathrm{div}
    \left[2 \mu \mathbf{D}\right] 
    = \rho \mathbf{b}(\mathbf{x}) \quad &&\mbox{in} 
    \; \Omega \\
    \label{Eqn:Brinkman_Continuity}
    &\mathrm{div}[\mathbf{v}(\mathbf{x})] = 0 \quad 
    &&\mbox{in} \; \Omega \\
    \label{Eqn:Brinkman_velocity_BC}
    &\mathbf{v}(\mathbf{x}) = \mathbf{v}^{\mathrm{p}}(\mathbf{x}) 
    \quad &&\mbox{on} \; \Gamma^{v} \\
    \label{Eqn:Brinkman_traction_BC}
    &\mathbf{T} 
    \mathbf{\widehat{n}}(\mathbf{x}) = \mathbf{t}^{\mathrm{p}}
    (\mathbf{x}) \quad &&\mbox{on} \; \Gamma^{t} 
  \end{alignat}
\end{subequations}
where $\mathbf{v}^{\mathrm{p}}(\mathbf{x})$ is the prescribed 
velocity vector, and $\mathbf{t}^{\mathrm{p}}(\mathbf{x})$ 
is the prescribed traction. We shall call a vector-field 
to be \emph{Darcy velocity} if satisfies equations 
\eqref{Eqn:Darcy_LM}--\eqref{Eqn:Darcy_traction_BC}. 
We shall call a vector field to be \emph{Darcy-Brinkman 
velocity} if it satisfies equations 
\eqref{Eqn:Brinkman_LM}--\eqref{Eqn:Brinkman_traction_BC}.

Equations \eqref{Eqn:Darcy_LM} and \eqref{Eqn:Brinkman_LM} 
can be obtained from the balance of linear momentum under 
the mathematical framework offered by the theory of 
interacting continua 
\citep{Nakshatrala_Rajagopal_IJNMF_2011_v67_p342}. 
The drag term $\alpha(\mathbf{x})\mathbf{v}(\mathbf{x})$ 
models the frictional force between the fluid and the 
porous solid. The term $\mathrm{div}[2 \mu \mathbf{D}]$ 
in the Darcy-Brinkman model arises due to the internal 
friction between the layers of the fluid. The pressure 
$p(\mathbf{x})$ is an undetermined multiplier that 
arises due to the enforcement of the incompressibility 
constraint given by equations \eqref{Eqn:Darcy_Continuity} 
and \eqref{Eqn:Brinkman_Continuity}. The drag coefficient 
is related to the coefficient of viscosity of the fluid 
and the permeability $k(\mathbf{x})$ as follows: 
\begin{align}
  \label{Eqn:Brinkman_alpha_mu}
  \alpha(\mathbf{x}) = \frac{\mu}{k(\mathbf{x})} 
\end{align}

In general, it is not possible to obtain analytical 
solutions to the systems of equations given by either 
\eqref{Eqn:Darcy_LM}--\eqref{Eqn:Darcy_traction_BC} or 
\eqref{Eqn:Brinkman_LM}--\eqref{Eqn:Brinkman_traction_BC}. 
Hence, one needs to resort to numerical solutions. 
\emph{This paper does not concern with developing new 
  numerical formulations to solve the aforementioned 
  mathematical models. The paper instead focuses on 
  deriving mathematical properties with firm mechanics 
  underpinning that the solutions to these mathematical 
  models satisfy. We shall then illustrate how these 
  mathematical properties can serve as robust ``a 
  posteriori'' criteria to assess the accuracy of 
  numerical solutions.}

\section{MATHEMATICAL PROPERTIES:~STATEMENTS AND DERIVATIONS}
\label{Sec:S3_Brinkman_Theorems}
In the remainder of this paper, we shall refer to 
a vector field $\widetilde{\mathbf{v}}:~\Omega 
\rightarrow \mathbb{R}^{nd}$ as \emph{kinematically 
admissible} if it satisfies the following conditions:
\begin{enumerate}[(i)]
\item $\widetilde{\mathbf{v}}(\mathbf{x})$ is solenoidal  
  (i.e., $\mathrm{div}[\widetilde{\mathbf{v}}(\mathbf{x})] 
  = 0 \; \mathrm{in} \; \Omega$), and 
\item $\widetilde{\mathbf{v}}(\mathbf{x})$ 
  satisfies the boundary conditions.
\end{enumerate}
It needs to emphasized that a kinematically 
admissible vector field need not satisfy the 
balance of linear momentum given by equation 
\eqref{Eqn:Darcy_LM} or \eqref{Eqn:Brinkman_LM}. 
Clearly, the Darcy velocity and the Darcy-Brinkman 
velocity are kinematically admissible vector fields. 
For some of the results presented in this paper, we 
will need the body force to be conservative, which is a 
common terminology in potential theory \citep{Kellogg}. 
The body force $\rho \mathbf{b}(\mathbf{x})$ is said 
to be conservative if there exists a scalar potential 
$\psi(\mathbf{x})$ such that $\rho \mathbf{b}(\mathbf{x}) 
= -\mathrm{grad}[\psi]$. We shall define the dissipation 
functional as follows:
\begin{align}
  \Phi[\mathbf{v}] := \left\{\begin{array}{ll} 
  \int_{\Omega} \alpha(\mathbf{x}) \mathbf{v}(\mathbf{x}) 
  \cdot \mathbf{v}(\mathbf{x}) \; \mathrm{d} \Omega & 
  \mbox{Darcy model} \\
  \int_{\Omega} \alpha(\mathbf{x}) 
  \mathbf{v}(\mathbf{x}) \cdot \mathbf{v}(\mathbf{x}) 
  \; \mathrm{d} \Omega + \int_{\Omega} 2 \mu 
  \mathbf{D}(\mathbf{x}) \cdot \mathbf{D}(\mathbf{x}) 
  \; \mathrm{d} \Omega & \mbox{Darcy-Brinkman model} 
  \end{array} \right.
\end{align}
%
  Since $\alpha > 0$ and $\mu > 0$, it is straightforward 
  to check that $\Phi[\mathbf{v}]$ is a norm. In fact, it 
  can be shown that $\Phi[\mathbf{v}]$ under the Darcy 
  model is \emph{equivalent} to the natural norm in 
  $(L^{2}({\Omega}))^{nd}$, where $(L^{2}(\Omega))^{nd}$ 
  is a space of square integrable vector fields defined 
  from $\Omega$ to $\mathbb{R}^{nd}$. Similarly, it 
  can be shown that $\Phi[\mathbf{v}]$ under the 
  Darcy-Brinkman model is equivalent to the natural 
  norm in $(H^{1}(\Omega))^{nd}$, which is a Sobolev 
  space. For further details on function spaces and 
  norms, refer to \citep{Evans_PDE}.

In this section, we shall present four important 
mathematical properties that the solutions to Darcy 
equations and Darcy-Brinkman equations satisfy. These 
properties will be referred to as (i) the minimum total 
mechanical power theorem, (ii) the minimum dissipation 
theorem, (iii) reciprocal relation, and (iv) maximum 
principle for vorticity. As a passing comment, we will 
employ the minimum dissipation theorem to show the 
uniqueness of solution for Darcy equations and 
Darcy-Brinkman equations. 
In the porous media literature, these results have neither 
been discussed nor utilized to solve problems. More 
importantly, these results have not been used to assess 
the accuracy and convergence of numerical solutions of 
porous media models. For example, it will be shown that 
the minimum total mechanical power theorem can be utilized 
to assess the accuracy of the implementation of both 
Dirichlet and Neumann boundary conditions. On the 
other hand, the minimum dissipation theorem can be 
utilized to identify pollution errors in the numerical 
solution. Herein, we give detailed mathematical 
proofs for Darcy-Brinkman equations. We however 
provide comments on the corresponding proofs for 
Darcy equations. 

\begin{theorem}[\textsf{Minimum total mechanical power theorem}]
  \label{Theorem:Minimum_total_mechanical_power}
  Let $\mathbf{v}(\mathbf{x})$ be the Darcy-Brinkman 
  velocity vector field. Then, any kinematically admissible 
  vector field $\widetilde{\mathbf{v}}(\mathbf{x})$ 
  satisfies the following inequality:
  \begin{align}
    \varepsilon_{\mathrm{TMP}}[\mathbf{v}] \leq 
    \varepsilon_{\mathrm{TMP}}[\widetilde{\mathbf{v}}] 
  \end{align}
  where 
  \begin{align}
    \label{Eqn:Brinkman_TMP}
    \varepsilon_{\mathrm{TMP}}[\mathbf{z}] 
    := \frac{1}{2} \Phi[\mathbf{z}]
    - \int_{\Omega} \rho \mathbf{b}(\mathbf{x}) \cdot 
    \mathbf{z}(\mathbf{x}) \; \mathrm{d} \Omega
    - \int_{\Gamma^{t}} \mathbf{t}^{\mathrm{p}}(\mathbf{x}) 
    \cdot  \mathbf{z}(\mathbf{x}) \; \mathrm{d} \Gamma   
  \end{align}
  That is, for given boundary conditions, body force 
  and tractions; the Darcy-Brinkman velocity will 
  have the minimum total mechanical power among all 
  the possible kinematically admissible vector fields.
\end{theorem}
%
\begin{proof}
	Let 
\begin{subequations}
	\label{Eqn:delta_v_and_D_definition}
	\begin{align}
    &\delta \mathbf{v}(\mathbf{x}) := \widetilde{\mathbf{v}}
    (\mathbf{x}) - \mathbf{v}(\mathbf{x}) \\
	&\delta \mathbf{D}(\mathbf{x}) := \widetilde{\mathbf{D}}
    (\mathbf{x}) - \mathbf{D}(\mathbf{x})
	\end{align}
\end{subequations} 
From the hypothesis of the theorem, $\delta \mathbf{v}
(\mathbf{x})$ satisfies the following relations:
\begin{subequations}    
	\label{Eqn:delta_v_hypothesis}	
	\begin{align}
 	&\delta \mathbf{v}(\mathbf{x}) = \mathbf{0} 
      \quad \forall \mathbf{x} \in \partial \Omega \\
      &\mathrm{div}[\delta \mathbf{v}] = 0 
      \quad \forall \mathbf{x} \in \Omega 
	\end{align}
\end{subequations} 
Let us start with the dissipation due to the vector 
field $\widetilde{\mathbf{v}}(\mathbf{x})$:
  \begin{align} 
	\label{Eqn:total_mechanical_power} 
    \Phi[\widetilde{\mathbf{v}}(\mathbf{x})]
    &:=\int_{\Omega} \alpha(\mathbf{x}) \widetilde{\mathbf{v}} 
    (\mathbf{x}) \cdot \widetilde{\mathbf{v}}(\mathbf{x}) \; 
    \mathrm{d} \Omega 
    +\int_{\Omega} 2 \mu \widetilde{\mathbf{D}} 
    (\mathbf{x}) \cdot \widetilde{\mathbf{D}}(\mathbf{x}) \; 
    \mathrm{d} \Omega \nonumber\\
    &= \int_{\Omega} \alpha(\mathbf{x}) \left( \delta \mathbf{v}
    (\mathbf{x}) + \mathbf{v}(\mathbf{x}) \right)
    \cdot \left( \delta \mathbf{v}(\mathbf{x}) + \mathbf{v}
    (\mathbf{x}) \right) \; \mathrm{d} \Omega 
    + \int_{\Omega} 2 \mu \left( \delta \mathbf{D}(\mathbf{x}) 
    + \mathbf{D}(\mathbf{x}) \right)
    \cdot \left( \delta \mathbf{D}(\mathbf{x}) + \mathbf{D}
    \mathbf{x}) \right) \; \mathrm{d} \Omega \nonumber \\
    &\geq 2 \int_{\Omega} \alpha(\mathbf{x}) \delta \mathbf{v}
    (\mathbf{x}) \cdot \mathbf{v}(\mathbf{x}) \; \mathrm{d} \Omega 
    + 2 \int_{\Omega} 2 \mu \delta \mathbf{D}(\mathbf{x}) 
    \cdot \mathbf{D}(\mathbf{x}) \; \mathrm{d} \Omega 
	+ \Phi{[\mathbf{v}(\mathbf{x})]}
	\end{align}
Using equations \eqref{Eqn:Brinkman_LM} and 
\eqref{Eqn:Brinkman_Cauchy_stress} the first 
integral in the above equation can be written 
as follows: 
\begin{align}
	\label{Eqn:total_dissipation_due_to_drag}
    \int_{\Omega} \alpha(\mathbf{x}) \delta \mathbf{v}(\mathbf{x}) 
    \cdot \mathbf{v}(\mathbf{x}) \; \mathrm{d} \Omega
	=\int_{\Omega} \delta \mathbf{v}
    (\mathbf{x}) \cdot \left( \rho \mathbf{b}(\mathbf{x}) +
	 \mathrm{div}\left[\mathbf{T}(\mathbf{x})\right] \right) \; \mathrm{d} \Omega
\end{align}
The symmetry of $\mathbf{D}(\mathbf{x})$ allows 
the second integral to be written as follows:
\begin{align*}
  \int_{\Omega} 2 \mu \delta \mathbf{D}(\mathbf{x}) 
  \cdot \mathbf{D}(\mathbf{x}) \; \mathrm{d} \Omega
  &=\int_{\Omega} 2 \mu \; \mathrm{grad}[\delta \mathbf{v}(\mathbf{x})] 
  \cdot \mathbf{D}(\mathbf{x}) \; \mathrm{d} \Omega \\
  &=\int_{\Omega} \; \mathrm{grad}[\delta \mathbf{v}(\mathbf{x})] 
  \cdot  \left( \mathbf{T}(\mathbf{x}) + p(\mathbf{x}) 
  \mathbf{I} \right) \; \mathrm{d} \Omega \\
  &=\int_{\Omega} \; \mathrm{grad}[\delta \mathbf{v}(\mathbf{x})] 
  \cdot  \mathbf{T}(\mathbf{x}) \; \mathrm{d} \Omega
  + \int_{\Omega} \; \mathrm{div}[\delta \mathbf{v}(\mathbf{x})] 
  \cdot  p(\mathbf{x}) \; \mathrm{d} \Omega
\end{align*}
Noting that $\mathrm{div}[\delta \mathbf{v}]=0\; 
\mathrm{in} \;\Omega$ and by employing Green's 
identity, we have the following:
\begin{align}
  \label{Eqn:total_dissipation_due_to_internal_friction}
  \int_{\Omega} 2 \mu \delta \mathbf{D}(\mathbf{x}) 
  \cdot \mathbf{D}(\mathbf{x}) \; \mathrm{d} \Omega
  &=  \int_{\Omega} \mathrm{div}[\mathbf{T}
    ^{\mathrm{T}}(\mathbf{x}) \;\delta \mathbf{v}
    (\mathbf{x})] \; \mathrm{d} \Omega
  -  \int_{\Omega} \delta \mathbf{v}(\mathbf{x}) 
  \cdot \mathrm{div}[\mathbf{T}(\mathbf{x})] \; \mathrm{d} \Omega
  \nonumber \\
  &=  \int_{\Gamma^{t}} \delta \mathbf{v}(\mathbf{x}) \cdot 
  \mathbf{t}^{\mathrm{p}}(\mathbf{x}) \; \mathrm{d} \Gamma
  -  \int_{\Omega} \delta \mathbf{v}(\mathbf{x}) 
    \cdot \mathrm{div}[\mathbf{T}(\mathbf{x})] \; \mathrm{d} \Omega
\end{align}
From equations \eqref{Eqn:total_dissipation_due_to_drag} 
and \eqref{Eqn:total_dissipation_due_to_internal_friction}, 
inequality \eqref{Eqn:total_mechanical_power} can be 
written as follows:
\begin{align}
  \Phi[\widetilde{\mathbf{v}}(\mathbf{x})]        
  &\geq \Phi[\mathbf{v}(\mathbf{x})] +  2 
  \int_{\Omega} \delta \mathbf{v}(\mathbf{x}) 
  \cdot \rho \mathbf{b}(\mathbf{x})  \; 
  \mathrm{d} \Omega
  + 2 \int_{\Gamma^{t}} \delta \mathbf{v}(\mathbf{x}) 
  \cdot \mathbf{t}^{\mathrm{p}}(\mathbf{x}) \; \mathrm{d} 
  \Gamma
\end{align}
This completes the proof.
\end{proof}
\begin{remark}
  For Darcy equations 
  \eqref{Eqn:Darcy_LM}--\eqref{Eqn:Darcy_traction_BC}, 
  one can state the minimum total mechanical power 
  theorem as follows: For given boundary conditions, 
  and body force; the Darcy velocity, $\mathbf{v}
  (\mathbf{x})$, has the minimum total mechanical 
  power among all the kinematically admissible vector 
  fields. That is,
  \begin{align}
    \frac{1}{2} \Phi[\mathbf{v}]
    - \int_{\Omega} \mathbf{v}(\mathbf{x}) \cdot 
    \rho \mathbf{b}(\mathbf{x}) \; \mathrm{d} 
    \Omega
    + \int_{\Gamma^{t}} p_0(\mathbf{x}) \mathbf{v}(\mathbf{x}) \cdot  
    \widehat{\mathbf{n}}(\mathbf{x}) \; \mathrm{d} \Gamma 
    & \leq \frac{1}{2} \Phi[\widetilde{\mathbf{v}}]
    - \int_{\Omega} \widetilde{\mathbf{v}}(\mathbf{x}) 
    \cdot \rho \mathbf{b}(\mathbf{x}) \; \mathrm{d} \Omega \nonumber \\
    & + \int_{\Gamma^{t}} p_0(\mathbf{x}) \widetilde{\mathbf{v}}(\mathbf{x}) \cdot  
    \widehat{\mathbf{n}}(\mathbf{x}) \; \mathrm{d} \Gamma 
    \quad \forall \widetilde{\mathbf{v}}(\mathbf{x})
  \end{align}
\end{remark}

\begin{theorem}[\textsf{Minimum dissipation inequality}]
  \label{Theorem:Minimum_dissipation}
Let $\mathbf{v}(\mathbf{x})$ be the Darcy-Brinkman 
velocity vector field, and let $\Gamma^{v} = \partial 
\Omega$. Then, any kinematically admissible vector 
field $\widetilde{\mathbf{v}}(\mathbf{x})$ satisfies 
the following inequality:
\begin{align}
  \Phi[\mathbf{v}] \leq \Phi[\widetilde{\mathbf{v}}]
\end{align}
That is, for given velocity boundary conditions 
and conservative body force, the Darcy-Brinkman 
velocity has the minimum total dissipation due 
to drag and internal friction of all the possible 
kinematically admissible vector fields. 
\end{theorem}
\begin{proof}
We shall employ the notation introduced in 
equation \eqref{Eqn:delta_v_and_D_definition}. 
Recall that the mechanical dissipation under the 
Darcy-Brinkman model is 
\begin{align}
\Phi[\mathbf{v}] = \int_{\Omega} \alpha \delta \mathbf{v}(\mathbf{x}) 
    \cdot \mathbf{v}(\mathbf{x}) \; \mathrm{d} \Omega 
    +  \int_{\Omega} 2 \mu \delta \mathbf{D}(\mathbf{x}) 
    \cdot \mathbf{D}(\mathbf{x}) \; \mathrm{d} \Omega 
\end{align}
Let us start with the difference in total dissipation, 
and from inequality \eqref{Eqn:total_mechanical_power} 
we have:
\begin{align}
    \Phi[\widetilde{\mathbf{v}}(\mathbf{x})] - 
    \Phi[\mathbf{v}(\mathbf{x})]
    \geq 2 \int_{\Omega} \alpha \delta \mathbf{v}(\mathbf{x}) 
    \cdot \mathbf{v}(\mathbf{x}) \; \mathrm{d} \Omega 
    + 2 \int_{\Omega} 2 \mu \delta \mathbf{D}(\mathbf{x}) 
    \cdot \mathbf{D}(\mathbf{x}) \; \mathrm{d} \Omega 
\end{align}
Using Green's identity, the first integral 
can be simplified as follows:
\begin{align}	
    \int_{\Omega} \alpha \delta \mathbf{v}(\mathbf{x}) 
    \cdot \mathbf{v}(\mathbf{x}) \; \mathrm{d} \Omega &=
	 \int_{\Omega} \delta 	\mathbf{v}
    (\mathbf{x}) \cdot \mathrm{grad}[\psi(\mathbf{x}) - 
      p(\mathbf{x})] \; \mathrm{d} \Omega
   +  \int_{\Omega} \delta \mathbf{v}(\mathbf{x}) 
    \cdot \mathrm{div}[2 \mu \mathbf{D}(\mathbf{x})] \; \mathrm{d} \Omega 
    \nonumber \\
  &=  \int_{\partial \Omega} \delta \mathbf{v}
    (\mathbf{x}) \cdot \widehat{\mathbf{n}}(\mathbf{x}) \left(\psi
    (\mathbf{x}) - p(\mathbf{x})\right) \; \mathrm{d} \Gamma
    - \int_{\Omega} \mathrm{div}[\delta \mathbf{v}
      (\mathbf{x})] \left(\psi(\mathbf{x}) - 
      p(\mathbf{x})\right) \; \mathrm{d} \Omega \nonumber \\
	&\quad +  \int_{\Omega} \delta \mathbf{v}(\mathbf{x}) 
    \cdot \mathrm{div}[2 \mu \mathbf{D}(\mathbf{x})] \; \mathrm{d} \Omega 
\end{align}
Noting the symmetry of $\mathbf{D}(\mathbf{x})$ and 
using Green's identity, the total dissipation due 
to internal friction can be simplified as follows:
\begin{align}	
	\int_{\Omega} 2 \mu \delta \mathbf{D}(\mathbf{x}) 
    \cdot \mathbf{D}(\mathbf{x}) \; \mathrm{d} \Omega
	&=\int_{\Omega} 2 \mu \; \mathrm{grad}[\delta \mathbf{v}(\mathbf{x})] 
    \cdot \mathbf{D}(\mathbf{x}) \; \mathrm{d} \Omega \nonumber \\
	&=  \int_{\partial \Omega} 2 \mu \delta \mathbf{v}
    (\mathbf{x}) \cdot \mathbf{D}(\mathbf{x}) 
	\widehat{\mathbf{n}}(\mathbf{x}) \; \mathrm{d} \Gamma
    -  \int_{\Omega} \delta \mathbf{v}(\mathbf{x}) 
    \cdot \mathrm{div}[2 \mu \mathbf{D}(\mathbf{x})] \; \mathrm{d} \Omega 
\end{align}
From equations \eqref{Eqn:total_dissipation_due_to_drag}--\eqref{Eqn:total_dissipation_due_to_internal_friction}, 
the total dissipation due to drag and internal friction satisfies:
\begin{align}
  \Phi[\widetilde{\mathbf{v}}(\mathbf{x})] - 
  \Phi[\mathbf{v}(\mathbf{x})]   
  &\geq  2 \int_{\partial \Omega} \delta \mathbf{v}
  (\mathbf{x}) \cdot \widehat{\mathbf{n}}(\mathbf{x}) \left(\psi
  (\mathbf{x}) - p(\mathbf{x})\right) \; \mathrm{d} \Gamma
  - 2 \int_{\Omega} \mathrm{div}[\delta \mathbf{v}
    (\mathbf{x})] \left(\psi(\mathbf{x}) - 
  p(\mathbf{x})\right) \; \mathrm{d} \Omega \nonumber \\
  & +2 \int_{\partial \Omega} 2 \mu \delta \mathbf{v}
  (\mathbf{x}) \cdot \mathbf{D}(\mathbf{x}) 
  \widehat{\mathbf{n}}(\mathbf{x})\; \mathrm{d} \Gamma
  = 0	
\end{align}
This completes the proof.
\end{proof}

The solutions to Darcy and Darcy-Brinkman equations 
posses reciprocal relations similar to the famous 
Betti's reciprocal relation in the theory of linear 
elasticity \citep{Truesdell_Noll, Sadd} and to a 
classical reciprocal relation in the area of creeping 
flows \citep{Guazzelli_Morris}. The Betti's reciprocal 
relation is often employed to solve a class of problems 
in linear elasticity, which otherwise may be difficult 
to solve. 
Mathematically, the Betti's reciprocal relation 
is equivalent to the existence and symmetry of 
Green's function. We now precisely state a 
reciprocal relation that the solutions of 
Darcy-Brinkman equations satisfy, and then 
provide a mathematical proof. 

\begin{theorem}[\textsf{Reciprocal relation}]
  \label{Theorem:Betti_relation}
  Assume that $\mathbf{v}^{\mathrm{p}}(\mathbf{x}) = \mathbf{0}$ 
  on $\Gamma^{v}$. Let $\left\{\mathbf{v}_1(\mathbf{x}), 
  p_1(\mathbf{x})\right\}$ and $\left\{\mathbf{v}_2(\mathbf{x}), 
  p_2(\mathbf{x})\right\}$ be the solutions of equations 
  \eqref{Eqn:Brinkman_LM}--\eqref{Eqn:Brinkman_traction_BC} 
  for the prescribed data $\left\{\mathbf{b}_1(\mathbf{x}), 
  \mathbf{t}^{\mathrm{p}}_1(\mathbf{x})\right\}$ and 
  $\left\{\mathbf{b}_2(\mathbf{x}), \mathbf{t}^{\mathrm{p}}_2
  (\mathbf{x})\right\}$, respectively. Then, these fields 
  satisfy the following relation:
  \begin{align}
    \label{Eqn:Brinkman_reciprocal}
    \int_{\Omega} \rho \mathbf{b}_1(\mathbf{x}) 
    \cdot \mathbf{v}_2(\mathbf{x}) \; \mathrm{d} 
    \Omega
    + \int_{\Gamma^{t}} \mathbf{t}_1^{\mathrm{p}}(\mathbf{x}) 
    \cdot \mathbf{v}_2(\mathbf{x}) \; \mathrm{d} \Gamma
    = \int_{\Omega} \rho \mathbf{b}_2(\mathbf{x}) 
    \cdot \mathbf{v}_1(\mathbf{x}) \; \mathrm{d} 
    \Omega
    + \int_{\Gamma^{t}} \mathbf{t}_2^{\mathrm{p}}(\mathbf{x}) 
    \cdot \mathbf{v}_1(\mathbf{x}) \; \mathrm{d} \Gamma
  \end{align}
\end{theorem}
\begin{proof}
  Let us start with the left side of equation 
  \eqref{Eqn:Brinkman_reciprocal}. Noting that 
  $\mathbf{v}^{\mathrm{p}}(\mathbf{x}) = \mathbf{0}$ 
  on $\Gamma^v$ and $\Gamma^{v} \cup \Gamma^{t} = 
  \partial \Omega$, one can proceed as follows: 
  \begin{align*}
   \int_{\Omega} \rho \mathbf{b}_1(\mathbf{x}) \cdot 
   \mathbf{v}_2(\mathbf{x}) \; \mathrm{d} \Omega
   &+ \int_{\Gamma^{t}} \mathbf{t}_1^{\mathrm{p}}(\mathbf{x}) 
   \cdot \mathbf{v}_2(\mathbf{x}) \; \mathrm{d} \Gamma
   = \int_{\Omega} \rho \mathbf{b}_1(\mathbf{x}) \cdot 
   \mathbf{v}_2(\mathbf{x}) \; \mathrm{d} \Omega 
   + \int_{\Gamma^{t}} \left(\mathbf{T}_1 \widehat{\mathbf{n}}
   (\mathbf{x})\right) \cdot \mathbf{v}_2(\mathbf{x}) \; 
   \mathrm{d} \Gamma \\
   &= \int_{\Omega} \rho \mathbf{b}_1(\mathbf{x}) \cdot 
   \mathbf{v}_2(\mathbf{x}) \; \mathrm{d} \Omega 
   + \int_{\partial \Omega} \left(\mathbf{T}_1 \widehat{\mathbf{n}}
   (\mathbf{x})\right) \cdot \mathbf{v}_2(\mathbf{x}) \; 
   \mathrm{d} \Gamma \\
   &= \int_{\Omega} \rho \mathbf{b}_1(\mathbf{x}) \cdot 
   \mathbf{v}_2(\mathbf{x}) \; \mathrm{d} \Omega 
   + \int_{\Omega} \mathrm{div} \left[\mathbf{T}_1^{\mathrm{T}}
     (\mathbf{x}) \mathbf{v}_2(\mathbf{x})\right] \; \mathrm{d} \Omega \\
   &= \int_{\Omega} \left(\rho \mathbf{b}_1(\mathbf{x}) 
   + \mathrm{div}[\mathbf{T}_1] \right) \cdot 
   \mathbf{v}_2(\mathbf{x}) \; \mathrm{d} \Omega 
   + \int_{\Omega} \mathbf{T}_1(\mathbf{x}) \cdot \mathrm{grad}
   \left[\mathbf{v}_2\right]\; \mathrm{d} \Omega \\
   &= \int_{\Omega} \alpha(\mathbf{x}) \mathbf{v}_1(\mathbf{x}) 
   \cdot \mathbf{v}_2(\mathbf{x}) \; \mathrm{d} \Omega
   + \int_{\Omega} \left(-p_1(\mathbf{x}) \mathbf{I} + 2 \mu 
   \mathbf{D}_1(\mathbf{x})
   \right) \cdot \mathrm{grad}\left[\mathbf{v}_2 \right]\; 
   \mathrm{d} \Omega \\
   &= \int_{\Omega} \alpha(\mathbf{x}) \mathbf{v}_1(\mathbf{x}) 
   \cdot \mathbf{v}_2(\mathbf{x}) \; \mathrm{d} \Omega - 
   \int_{\Omega} p_1(\mathbf{x})\mathrm{div}\left[\mathbf{v}_2 \right] 
   \; \mathrm{d} \Omega + \int_{\Omega} 2 \mu \mathbf{D}_1(\mathbf{x})
   \cdot \mathbf{D}_2(\mathbf{x}) \; \mathrm{d} \Omega
\end{align*}
  In the above step, we have used the fact that $\mathbf{D}_1
  (\mathbf{x})$ is a symmetric second-order tensor. Since 
  ${\mathrm{div}\left[\mathbf{v}_2\right]=0}$ we have the 
  following: 
  \begin{align*}
    \int_{\Omega} \rho \mathbf{b}_1(\mathbf{x}) \cdot 
    \mathbf{v}_2(\mathbf{x}) \; \mathrm{d} \Omega
    + \int_{\Gamma^{t}} \mathbf{t}_1^{\mathrm{p}}(\mathbf{x}) 
    \cdot \mathbf{v}_{2}(\mathbf{x}) \; \mathrm{d} \Gamma
    &= \int_{\Omega} \alpha(\mathbf{x}) \mathbf{v}_1
    (\mathbf{x}) \cdot \mathbf{v}_2(\mathbf{x}) 
    \; \mathrm{d} \Omega + \int_{\Omega} 2 \mu 
    \mathbf{D}_1(\mathbf{x}) \cdot \mathbf{D}_2
    (\mathbf{x}) \; \mathrm{d} \Omega
  \end{align*}
  Similarly, it can be shown that the right side 
  of equation \eqref{Eqn:Brinkman_reciprocal} is 
  also equal to 
  \begin{align*}
  \int_{\Omega} \alpha(\mathbf{x}) \mathbf{v}_1(\mathbf{x}) 
  \cdot \mathbf{v}_2(\mathbf{x}) \; \mathrm{d} \Omega 
  + \int_{\Omega} 2 \mu \mathbf{D}_1(\mathbf{x}) \cdot 
  \mathbf{D}_2(\mathbf{x}) \; \mathrm{d} \Omega
  \end{align*}  
  This completes the proof. 
\end{proof}
The following notation will be used later 
to verify the reciprocal relation:
\begin{align}
  \label{Eqn:Brinkman_epsilon_reciprocal}
  \varepsilon_{\mathrm{reciprocal}} := \frac{\mbox{left integral} 
    - \mbox{right integral}}{\mbox{left integral}}
\end{align}
where the left and right integrals are, 
respectively, defined as follows:
\begin{subequations}
  \begin{align}
    \mbox{left integral} := \int_{\Omega} \rho 
    \mathbf{b}_1(\mathbf{x}) \cdot \mathbf{v}_2
    (\mathbf{x}) \; \mathrm{d} \Omega + 
    \int_{\Gamma^{t}} \mathbf{t}_1^{\mathrm{p}}(\mathbf{x}) 
    \cdot \mathbf{v}_2(\mathbf{x}) \; \mathrm{d} \Gamma \\
    \mbox{right integral} := \int_{\Omega} \rho 
    \mathbf{b}_2(\mathbf{x}) \cdot \mathbf{v}_1
    (\mathbf{x}) \; \mathrm{d} \Omega + 
    \int_{\Gamma^{t}} \mathbf{t}_2^{\mathrm{p}}(\mathbf{x}) 
    \cdot \mathbf{v}_1(\mathbf{x}) \; \mathrm{d} \Gamma
  \end{align}
\end{subequations}

\begin{remark}
  \label{Remark:Betti_relation_Darcy}
  The corresponding reciprocal relation for Darcy 
  equations can be written as follows: Assume that 
  $v_{n}(\mathbf{x}) = 0$ on $\Gamma^{v}$. Let $\left
  \{\mathbf{v}_1(\mathbf{x}), p_1(\mathbf{x})\right\}$ 
  and $\left\{\mathbf{v}_2(\mathbf{x}), p_2(\mathbf{x})
  \right\}$ be the solutions of equations 
  \eqref{Eqn:Brinkman_LM}--\eqref{Eqn:Brinkman_traction_BC} 
  for the prescribed data $\left\{\mathbf{b}_1(\mathbf{x}), 
  {p_{0}}_1(\mathbf{x})\right\}$ and $\left\{\mathbf{b}_2
  (\mathbf{x}), {p_{0}}_2(\mathbf{x})\right\}$, respectively. 
  Then, these fields satisfy the following relation:
  {\small
    \begin{align}
      \int_{\Omega} \rho \mathbf{b}_1(\mathbf{x}) 
      \cdot \mathbf{v}_2(\mathbf{x}) \; \mathrm{d} 
      \Omega
      &- \int_{\Gamma^{t}} {p_0}_1(\mathbf{x}) \widehat{\mathbf{n}}
      (\mathbf{x}) \cdot \mathbf{v}_2(\mathbf{x}) \; \mathrm{d} 
      \Gamma \nonumber \\
      &= \int_{\Omega} \rho \mathbf{b}_2(\mathbf{x}) 
      \cdot \mathbf{v}_1(\mathbf{x}) \; \mathrm{d} 
      \Omega
      - \int_{\Gamma^{t}} {p_0}_2(\mathbf{x}) \widehat{\mathbf{n}}
      (\mathbf{x})\cdot \mathbf{v}_1(\mathbf{x}) \; \mathrm{d} 
      \Gamma
    \end{align}
  }
\end{remark}

\begin{remark}
  It needs to be emphasized that the reciprocal relation will 
  not directly be able to assess the accuracy of the pressure 
  field in the computational domain. The reciprocal relation 
  is ideal for assessing the accuracy of the velocity vector 
  field in the domain and the accuracy of the implementation 
  of prescribed traction boundary conditions. This reciprocal 
  relation will not be able to provide information about the 
  accuracy of the implementation of non-zero velocity boundary 
  conditions.
\end{remark}

Next, we discuss in the form of a theorem on the nature 
of the vorticity under the Darcy and Darcy-Brinkman 
models. To this end,  
\begin{align}
  \boldsymbol{\omega}(\mathbf{x}) \equiv  
  \mathrm{curl}[\mathbf{v}(\mathbf{x})]
\end{align}
In a Cartesian coordinate system, the components 
of vorticity take the following form:
\begin{align}
  \label{Eqn:vorticity_xyz_coord}
  \omega_x = \frac{\partial v_z}{\partial y} - 
  \frac{\partial v_y}{\partial z}, \; 
  \omega_y = \frac{\partial v_x}{\partial z} - 
  \frac{\partial v_z}{\partial x}, \;
  \omega_z = \frac{\partial v_y}{\partial x} - 
  \frac{\partial v_x}{\partial y}
\end{align}
It should be emphasized that $\mathrm{curl}[\cdot]$ 
operator is defined only in $\mathbb{R}^{3}$ (i.e., 
the three-dimensional Euclidean space).
However, for two-dimensional problems, one 
can consider the vorticity as follows:
\begin{align}
  \boldsymbol{\omega}(\mathrm{x},\mathrm{y}) = \omega_{z}
  (\mathrm{x},\mathrm{y}) \widehat{\mathbf{e}}_z
\end{align}
where $z$ denotes the axis perpendicular to the 
two-dimensional plane in which the problem is 
defined, and $\widehat{\mathrm{e}}_z$ is the 
unit vector along the $z$-direction. 

\begin{theorem}[\textsf{On the nature of vorticity 
      under Darcy and Darcy-Brinkman models}]
  \label{Theorem:Vorticity}
  Assume that the medium is isotropic and homogeneous 
  (i.e., $\alpha(\mathbf{x})$ is a constant scalar), 
  the body force is a conservative vector field (i.e., 
  $\rho \mathbf{b}(\mathbf{x}) = -\mathrm{grad}[\psi
    (\mathbf{x})]$), and the response is steady-state. 
  Then the vorticity vanishes under Darcy equations. 
  Under Darcy-Brinkman equations, the vorticity is 
  an eigenvector of the Laplacian with $1/k$ 
  as the eigenvalue. Moreover, the 
  vorticity satisfies a maximum principle 
  in which the non-negative maximum and 
  the non-positive minimum occur on the 
  boundary.   
\end{theorem}
\begin{proof}
  By taking the curl on both sides of the balance 
  of linear momentum under Darcy equations 
  \eqref{Eqn:Darcy_LM} we obtain the following:
  \begin{align}
    \mathrm{curl}[\alpha \mathbf{v}] = -\mathrm{curl}[
      \mathrm{grad}[\psi + p]] = \mathbf{0} 
  \end{align}
  Since $\alpha$ is spatially homogeneous scalar, 
  one can conclude that the vorticity vanishes 
  under the Darcy model.

  The incompressibility constraint implies 
  that the balance of linear momentum under 
  the Darcy-Brinkman model can be written 
  as follows:
  \begin{align}
    \alpha \mathbf{v}(\mathbf{x}) + 
    \mathrm{grad}[p(\mathbf{x})] - \mu \mathrm{div}
           [\mathrm{grad}[\mathbf{v}]]
           = -\mathrm{grad}[\psi] 
  \end{align}
  By taking curl on both sides of the above 
  equation, we get the following:
  \begin{align}
    \Delta \boldsymbol{\omega}(\mathbf{x}) = 
    \frac{1}{k} \boldsymbol{\omega}(\mathbf{x}) 
  \end{align}
  where $\Delta$ denotes the Laplacian operator and 
  $k$ is the permeability. The above equation is a 
  \emph{vector} eigenvalue problem in which the 
  vorticity vector is the eigenvector, and $1/k$ 
  is the corresponding eigenvalue. 
  For two-dimensional problems, we have the following: 
  \begin{align}
    \Delta_{2\mathrm{D}} \omega_{z}(\mathrm{x},\mathrm{y}) 
    = \frac{1}{k} \omega_{z}(\mathrm{x},\mathrm{y})
  \end{align}
  where $\Delta_{2\mathrm{D}}$ denotes the two-dimensional 
  Laplacian operator. The above equation is a \emph{scalar} 
  eigenvalue problem in which $\omega_{z}(\mathrm{x},\mathrm{y})$ 
  is the eigenvector and $1/k$ is the corresponding eigenvalue. 
  
  Since $1/k > 0$, this equation is commonly referred to 
  as diffusion with decay, which is a linear self-adjoint 
  elliptic partial differential equation. It is well-known 
  that such a partial differential equation satisfies a 
  maximum principle \citep{Gilbarg_Trudinger}. 
  Mathematically, the maximum principle for the vorticity 
  under the Darcy-Brinkman model takes the following form: 
  If $w_{z}(\mathbf{x}) \in C^{2}(\Omega) \cap C^{0}
  (\overline{\Omega})$, then the non-negative maximum 
  and the non-positive minimum occur on the boundary.
  That is, 
  \begin{subequations}
    \begin{align}
      &\max_{\mathbf{x} \in \overline{\Omega}} \left[\omega_{z}(\mathbf{x})\right] \leq \max\left[0,\max_{\mathbf{x} \in \partial \Omega} \omega_{z}(\mathbf{x}) \right] \\
      &\min_{\mathbf{x} \in \overline{\Omega}} \left[\omega_{z}(\mathbf{x}) \right] \geq \min\left[0,\min_{\mathbf{x} \in \partial \Omega} \omega_{z} (\mathbf{x}) \right]
    \end{align}
  \end{subequations}
  where $C^{2}(\Omega)$ denotes the set of twice 
  differentiable functions defined on $\Omega$, 
  and $C^{0}(\overline{\Omega})$ is the set of 
  functions that are continuous to the boundary. 
\end{proof}

The maximum principle is used to obtain 
physical meaningful numerical solutions 
(see \citep{nakshatrala_2015_NN_CiCP, 
Maruti_2015_NN_ADR_JCP} and References therein).
Herein, it can be utilized to verify 
the accuracy of numerical solutions by plotting 
vorticity and checking whether the non-negative 
maximum and non-positive minimum of the 
vorticity occur on the boundary. 
For the Darcy model with isotropic and homogeneous 
medium properties and conservative body force, 
it can be shown that the vorticity vanishes (i.e., 
$\boldsymbol{\omega}(\mathbf{x}) = \boldsymbol{0}$). 
However, it should be noted that heterogeneity, 
pressure-dependent viscosity, or non-conservative 
body force can introduce vorticity under the Darcy 
model.
All the above results can serve as invaluable tools 
to assess the performance of a numerical formulation 
to verify a computer implementation, and to provide 
metrics for numerical convergence.

\section{STEADY-STATE NUMERICAL RESULTS}
\label{Sec:S4_Brinkman_NR}
We shall first non-dimensionalize the governing 
equations by choosing primary variables that seem 
appropriate for problems arising in modeling of flows 
through porous media. This non-dimensional procedure 
is different from the standard non-dimensionalization 
procedure for incompressible Navier-Stokes in the 
choice of primary variables. In the standard 
non-dimensionalization of Navier-Stokes equations, 
one employs characteristic velocity $v$, characteristic 
length $L$  and density of the fluid $\rho$ as primary 
variables. 
We shall choose $L$ (reference length in the problem), 
$g$ (acceleration due to gravity) and $p_{\mathrm{atm}}$ 
(atmospheric pressure) as the reference quantities. 
Using these reference quantities, we define the 
following non-dimensional quantities: 
\begin{align}
  &\overline{\mathbf{x}} = \frac{\mathbf{x}}{L}, \; 
  \overline{\mathbf{v}} = \frac{\mathbf{v}}{\sqrt{gL}}, \; 
  \overline{\mathbf{v}}^{\mathrm{p}} = \frac{\mathbf{v}^{\mathrm{p}}}{\sqrt{gL}}, \;
  \overline{\mathbf{T}} = \frac{\mathbf{T}}{p_{\mathrm{atm}}}, \;
  \overline{\mathbf{b}} = \frac{\mathbf{b}}{g}, \;
  \overline{p} = \frac{p}{p_{\mathrm{atm}}},  \nonumber \\ 
  &\overline{p}_0 = \frac{p_0}{p_{\mathrm{atm}}}, \; 
  \overline{\rho} = \frac{\rho g L}{p_{\mathrm{atm}}}, \; 
  \overline{\alpha} = \frac{\alpha \sqrt{g L^3}}{p_{\mathrm{atm}}}, \;
  \overline{\mu}= \frac{\mu \sqrt{g/L}}{p_{\mathrm{atm}}}, \;
  \overline{k}= \frac{k}{L^2}
\end{align}
where the non-dimensional quantities are denoted by superposed bars.
The gradient and divergence operators with respect to 
$\overline{\mathbf{x}}$ are denoted as $\overline{\mathrm{grad}}
[\cdot]$ and $\overline{\mathrm{div}}[\cdot]$, respectively. The 
scaled domain $\Omega_{\mathrm{scaled}}$ is defined as follows: A 
point in space with position vector $\overline{\mathbf{x}} \in 
\Omega_{\mathrm{scaled}}$ corresponds to the same point with position 
vector given by $\mathbf{x} = \overline{\mathbf{x}} L \in \Omega$. 
Similarly, one can define the scaled boundaries: $\partial 
\Omega_{\mathrm{scaled}}$, $\Gamma^{v}_{\mathrm{scaled}}$, and 
$\Gamma^{p}_{\mathrm{scaled}}$. 
Using the above non-dimensionalization procedure, 
Darcy-Brinkman equations can be written as follows:
\begin{subequations}
  \begin{align}
    &\overline{\alpha}(\overline{\mathbf{x}}) \overline{\mathbf{v}}(\overline{\mathbf{x}}) + \overline{\mathrm{grad}} 
      [\overline{p}(\overline{\mathbf{x}})] - \overline{\mathrm{div}} [ 2 \overline{\mu} \overline{\mathbf{D}}(\overline{\mathbf{x}})] 
    = \overline{\rho} \overline{\mathbf{b}}(\overline{\mathbf{x}}) \quad 
    \; \mathrm{in} \; {\Omega}_{\mathrm{scaled}} \\
    &\overline{\mathrm{div}}[\overline{\mathbf{v}}(\overline{\mathbf{x}})] = 0 \quad \; 
    \mathrm{in} \; {\Omega}_{\mathrm{scaled}} \\
    &\overline{\mathbf{v}}(\overline{\mathbf{x}}) = 
    \overline{\mathbf{v}}^{\mathrm{p}}(\overline{\mathbf{x}}) \quad \; 
        \mathrm{on} \; {\Gamma}^{v}_{\mathrm{scaled}} \\
    &\overline{\mathbf{t}}(\overline{\mathbf{x}}) = 
        \overline{\mathbf{t}}^{\mathrm{p}} (\overline{\mathbf{x}}) 
    \quad \; \mathrm{on} \; {\Gamma}^{t}_{\mathrm{scaled}}
  \end{align}
\end{subequations}
Similarly, one could write the corresponding non-dimensional 
form of Darcy equations. For simplicity, the ``over-lines" 
will be dropped in the remainder of the paper. 

We shall use several flow through porous media problems with 
different boundary conditions to illustrate that the proposed 
\emph{a posteriori} techniques can be used as good measures 
for the accuracy and convergence of numerical results. 
We have employed P2P1 (which is based on six-node triangle 
element T6) and Q2P1 (which is based on nine-node quadrilateral 
element Q9) interpolations available in COMSOL \citep{COMSOL_v4.3b}. 
Consistent SUPG stabilization has been employed if the interpolations 
(i.e., Q2P1 and P2P1) violate the LBB \emph{inf-sup} stability condition \citep{Hughes_Franca_Balestra_CMAME_1986_v59_p85,COMSOL_v4.3b}. 
Typical structured finite elements utilized in this paper are 
shown in Figure \ref{Fig:quadrilateral_triangular_ele}. Unless 
mentioned otherwise, all the elements in a quadrilateral mesh 
are squares and all the elements in a triangular mesh are 
right-angled isosceles triangles. In this numerical solution 
study, we have taken $h$ (the maximum element size) to be equal 
to the length of the side for square elements, and to the length 
of the base (or height) for right-angled isosceles triangles.

\subsection{Body force problem}
The test problem is pictorially described in Figure 
\ref{Fig:Test_prob}. The non-dimensional parameters 
used in the numerical simulation are provided in 
Table \ref{Tab:Input_to_COMSOL_for_body_force}. 
The conservative body force is taken as $\rho 
\mathbf{b}(\mathbf{x}) = 10[\mathrm{sin}(\pi x), 
\mathrm{cos}(\pi y)]$ (i.e., $a=10$). 
Velocity boundary condition is prescribed on the 
entire boundary (i.e., $\Gamma^{v} = \partial 
\Omega$). For the Darcy-Brinkman model, we 
assume $\mathbf{v}^{\mathrm{p}}(\mathbf{x}) = 
\mathbf{0}$; and for the Darcy model, we 
assume $v_n(\mathbf{x}) = 0$.  
Figure \ref{Fig:test_prob_diss} shows the minimum 
total mechanical power and the minimum dissipation 
with mesh refinement. The numerical results for the 
reciprocal relation with mesh refinement are shown in 
Figure \ref{Fig:test_prob_Betti}. All the numerical 
results are in accordance with the theoretical 
predictions for this test problem. 

\begin{table}
  \begin{center}
    \caption{\textsf{Body force problem:}~Non-dimensional 
      parameters used in the problem.
      \label{Tab:Input_to_COMSOL_for_body_force}}
    \begin{tabular}{c c}
      \hline
      Parameter                 & Value \\ \hline\hline
      $\alpha$		        & 1 \\
      $\mu$		        & 1 and 0.001 \\
      $\rho$		        & 1 \\
      $L$		        & 1 \\
      $\mathbf{b}(\mathbf{x})$	& $a[\mathrm{sin}(\pi x), \mathrm{cos}(\pi y)]$ \\      
      \hline
    \end{tabular}
  \end{center} 
\end{table}

\subsection{Lid-driven cavity problem}
The two-dimensional lid-driven cavity problem is 
a benchmark study widely used to investigate the accuracy 
of numerical formulations for various fluid models
\citep{Burggraf_1966_v24_p113_151,E_Ertuk_2009_v60_p275_294_IJNMF}.
Figure \ref{Fig:lid_driven_cavity_problem} provides 
a pictorial description of the problem. The domain 
of the problem is a bi-unit square. Velocity boundary 
condition is prescribed on the entire boundary (i.e., 
$\Gamma^{v}=\partial \Omega$) which implies that the 
minimum dissipation theorem is also applicable to this 
problem. 
It should be noted that lid-driven cavity problem is 
not compatible with Darcy equations which need only 
the normal component of the velocity to be prescribed 
on the boundary. However, the lid-driven cavity problem 
demands the prescription of the entire velocity vector 
field on the boundary. We shall therefore employ 
Darcy-Brinkman model in our numerical simulations. 

It is crucial to note that the solution to the 
lid-driven cavity problem has singularities at 
the top corners which arise due to velocity 
discontinuity on the boundary 
\citep{Botella_Peyret_1998_v27_p421_433,
  Batchelor_Fluid_Dynamics}. 
%

Herein, we use an adaptive mesh to resolve the 
singularities  in the solution. The non-dimensional 
parameters used in the lid-driven cavity problem are 
presented in Table \ref{Tab:Input_to_COMSOL_for_lid_cavity}. 
Figures 
\ref{Fig:lid_cavity_uni_Q9_4} and \ref{Fig:lid_cavityadapt_Q9_4_top} 
show the uniform structured mesh and the adaptive mesh, respectively. 
For the adaptive mesh, we generate fine grid in the region close to 
top lid (see Figure \ref{Fig:lid_cavityadapt_Q9_4_top}). 
Figures \ref{Fig:Br_cavity_min_diss} and 
\ref{Fig:Br_fine_top_cavity_min_diss} show 
variation of the minimum dissipation with 
mesh refinement for respective grids. 
Under the adaptive mesh, the dissipation 
decreased uniformly and converged to a 
constant value with mesh refinement. On 
the other hand, the dissipation increased 
monotonically with mesh refinement under 
the structured mesh. More importantly, due 
to the presence of singularities and pollution 
error, the dissipation did not hit a plateau 
even for very fine meshes (i.e., $1/h \geq 220$). 
The dissipation reached a plateau relatively 
quickly under the adaptive mesh (say $1/h = 10$). 
\emph{This problem clearly illustrates that the 
minimum dissipation theorem can be used to 
identify pollution errors and to 
check performance of adaptive mesh in numerical 
solutions which is one of the main findings 
of this paper.}

\begin{table}[t]
  \caption{\textsf{Lid-driven cavity problem:}~Non-dimensional 
    parameters used in the problem.
    \label{Tab:Input_to_COMSOL_for_lid_cavity}}
  \centering
  \begin{tabular}{c c}
    \hline
    Parameter              & Value \\
    \hline\hline
    $\alpha$		& 1 \\
    $\mu$		& 1 \\
    $\rho$		& 1 \\
    $L$		& 1 \\
    \hline
  \end{tabular}
\end{table}

\subsection{Pipe bend problem}
\label{Pipe_bend_problem}
As another application of the proposed techniques, 
we consider an engineering problem commonly 
found in the fluid mechanics literature, which 
is called the pipe bend problem (for example, see 
References \citep{T_Borrvall_Topology_pipe_bend_v41_p77_2003_IJNMF, AG_Hansen_topology_pipe_v30_p_181_2005_Struct_Multidisc_Optim, N_Aage_topology_pipe_v35_p175_2008_Struc_Multi_Optim, 
G_Pingen_optim_pipe_v38_p910_2009_Compt_Fluids, 
VJ_Challis_Topology_pipe_bend_v79_p1284_2009_IJNME, 
M_Hassine_Topology_pipe_bend_2012_chapte10}).
In the pipe bend problem, the computational 
domain $\Omega$ is a square with $L = 1$.

\subsubsection{Velocity boundary condition}
The problem is pictorially described in Figure \ref{Fig:pipe_bend}. 
An inflow parabolic velocity is enforced on a part of 
the left boundary and an outflow parabolic 
velocity on a segment of the bottom. 
Each parabolic velocity profile 
has a unit maximum value (i.e., ${v}_{{x}_{max}}=1$ 
or ${v}_{{y}_{max}}=1$).
Elsewhere, the homogeneous 
velocity is prescribed  (i.e., $\mathbf{v}^{\mathrm{p}}
(\mathbf{x}) = \mathbf{0}$ for the Darcy-Brinkman model 
and $v_{n}(\mathbf{x}) = 0$ for the Darcy model).
The velocity boundary condition makes 
the problem compatible with the total mechanical 
power and the dissipation theorems. 
The non-dimensional parameters used in the problem 
are presented in Table 
\ref{Tab:Input_to_COMSOL_for_2D_pipe_bend}.
Figure \ref{Fig:pipe_bend_diss} shows the 
variation of minimum total mechanical power and minimum 
dissipation with mesh refinement for the 
quadrilateral and triangular elements. 
The result of the numerical solutions verification 
are presented in Figure \ref{Fig:Dr_vs_Br_pipe_min_diss} 
and \ref{Fig:Dr_vs_Br_pipe_min_mech_pw} for the 
minimum dissipation and the total mechanical power, 
respectively. The numerical error decreases and 
converges uniformly.
\subsubsection{Parabolic velocity-pressure boundary condition}
A pictorial description of the problem is given by Figure 
\ref{Fig:2D_press_pipe_bend}. 
The traction is prescribed on a part of the bottom 
boundary (i.e., $\mathbf{t}^\mathrm{p}(\mathbf{x})
=-p_{\mathrm{atm}}\widehat{\mathbf{n}}(\mathbf{x})$ 
on $\Gamma^t$). 
The velocity has parabolic 
profile with unit maximum value (i.e., ${v}_{{x}_{max}}=1$) 
prescribed on a segment of the left boundary. 
Elsewhere, the homogeneous velocity is enforced. 
On account of the traction and parabolic velocity boundary 
conditions, current problem is not compatible with 
the minimum dissipation and reciprocal theorems. It is 
only compatible with the total mechanical power 
theorem.
It should be noted that the solution to the 
problem has the singularity at near corners of 
the outlet (i.e., $\Gamma^t$). 
Hence, we again use an adaptive mesh to resolve the 
pollution in the solution.
The non-dimensional parameters using in the 
problem are provided in 
Table \ref{Tab:Input_to_COMSOL_for_2D_pipe_bend}.
Figure \ref{Fig:press_pipe_min_mech_pw} depicts 
the variation of the minimum total mechanical 
power with mesh refinement for the quadrilateral 
and triangular elements.
The top figures show the uniform structured 
and adaptive meshes.
Under the structured mesh, the total mechanical 
power increased uniformly with mesh refinement 
due to the polluted area in the computational domain.
The error decreased uniformly and converged 
to a constant value with mesh refinement using 
the adaptive mesh.
So, in addition to error estimation capability, 
the total mechanical power theorem can be used 
to identify the pollution errors in the numerical 
solutions and check the performance 
of adaptive mesh.
%
\begin{table}
\begin{center}
\caption {\textsf{Pipe bend problem:}~Non-dimensional 
parameters of the parabolic velocity boundary condition problem.
\label{Tab:Input_to_COMSOL_for_2D_pipe_bend}}
\begin{tabular}{c c}
\hline
Parameter              & Value \\
\hline\hline
$\alpha$		& 1 and 10 \\
$\mu$		& 1 and 0.001 \\
$\rho$		& 1 \\
$\mathbf{b}(\mathbf{x})$		& [1, 1] \\
$L$		& 1 \\
\hline
\end{tabular}
\end{center} 
\end{table}
\subsection{Pressure slab problem}
Figure \ref{Fig:press_slab} provides a pictorial 
description of the problem. The non-dimensional 
parameters used in the numerical simulation are 
provided in Table \ref{Tab:Input_to_COMSOL_for_slab}.
The domain is a $W \times L$ rectangle. 
The homogeneous velocity 
boundary condition is enforced on the top and 
bottom sides of the boundary. 
The traction is prescribed on the left side (i.e., 
$\mathbf{t}^\mathrm{p}(\mathbf{x})=-p_{\mathrm{inj}} 
\widehat{\mathbf{n}}(\mathbf{x})$ on $\Gamma^t_1$) 
and on the right side (i.e., 
$\mathbf{t}^\mathrm{p}(\mathbf{x})=-p_{\mathrm{atm}} 
\widehat{\mathbf{n}}(\mathbf{x})$ on $\Gamma^t_2$). 
We shall use this test problem to assess the 
accuracy of numerical solutions with respect 
to the reciprocal relation.  
Figure \ref{Fig:Br_Betti_slab_no_bodyf} depicts 
the variation of the error in the reciprocal 
relation with mesh refinement for the triangular 
and quadrilateral grids. 
The error in the reciprocal relations under Darcy 
equations is very close to zero for all the meshes. 
The error under Darcy-Brinkman equations decrease 
uniformly with mesh refinement. All the numerical 
results are in accordance with the theoretical 
predictions. 
\begin{table}
\begin{center}
\caption{\textsf{Pressure slab problem:}~Non-dimensional 
  parameters used in the problem.
\label{Tab:Input_to_COMSOL_for_slab}}
\begin{tabular}{c c}
\hline
Parameter              & Value \\
\hline\hline
$\alpha$		& 1 \\
$\mu$		& 0.001 \\
$\rho$		& 1 \\
$p_{\mathrm{inj}}$	& 5 and 7.5 \\
$p_{\mathrm{atm}}$	& 1 \\
$L$		& 1 \\
$W$		& 0.2 \\
\hline
\end{tabular}
\end{center}
\end{table}

\subsection{Pressure driven problem}
The domain is a square with $L = 1$. 
A pictorial description of the problem 
is given in Figure \ref{Fig:press_driven}. 
The traction is prescribed on the left 
boundary (i.e., $\mathbf{t}^\mathrm{p}
(\mathbf{x})=-p_{\mathrm{inj}} \widehat
{\mathbf{n}}(\mathbf{x})$ on $\Gamma^t_1$) 
and on the middle of the right boundary 
(i.e., $\mathbf{t}^\mathrm{p}(\mathbf{x})
=-p_{\mathrm{atm}} \widehat{\mathbf{n}}
(\mathbf{x})$ on $\Gamma^t_2$). 
Elsewhere, homogeneous velocity boundary 
condition is enforced. 
Due to the prescription of the traction 
on a part of the boundary, this problem 
is incompatible with the minimum dissipation 
theorem but is compatible with the reciprocal 
relation and total mechanical power theorems. 
Herein, we present the results for the 
reciprocal relation. The non-dimensional 
parameters used in the numerical simulation 
are provided in Table 
\ref{Tab:Input_to_COMSOL_for_press_driven}. 
Figure \ref{Fig:Betti_press_driven} shows 
the variation of the error in the reciprocal 
relation with mesh refinement for triangular 
and quadrilateral meshes. 
%
The convergence under uniform structured meshes 
is uniformly to zero which is expected by theorem.
%
\begin{table}
\begin{center}
\caption{\textsf{Pressure driven problem:}~Non-dimensional 
  parameters used in the problem.
\label{Tab:Input_to_COMSOL_for_press_driven}}
\begin{tabular}{c c}
\hline
Parameter             & Value \\ \hline\hline
$\alpha$              & 1 \\
$\mu$		      & 1 and 0.001 \\
$\rho$	              & 1 \\
$p_{\mathrm{inj}}$        & 5 and 7.5 \\
$p_{\mathrm{atm}}$	      & 1 \\
$L$		      & 1 \\
\hline
\end{tabular}
\end{center}
\end{table}

\subsection{Vorticity results}
The maximum principle given by Theorem \ref{Theorem:Vorticity} 
can be used to assess the accuracy of numerical solutions to 
the Darcy-Brinkman equations by checking whether the non-negative 
maximum and non-positive minimum of the vorticity occur 
on the boundary. Figure \ref{Fig:Metrics_vorticity_Brinkman} 
shows that the maximum principle is satisfied for various 
problems under the steady-state Darcy-Brinkman equations.

\section{SYNTHETIC RESERVOIR DATA:~MARMOUSI DATASET}
\label{Sec:S5_Brinkman_Marmousi}
We will now solve an idealized reservoir problem using 
a popular synthetic dataset -- the so-called (smooth) 
Marmousi dataset \citep{Versteeg_marmousi_1990,
Versteeg_marmousi_1991,Versteeg_1993,Marmousi_Online,
Marmousi_Rice}. The dataset provides spatially varying 
speed of sound on a $384 \times 122$ grid. We have 
assumed that the permeability scales linearly with 
the values provided by the dataset. This is just an 
arbitrary choice to generate a heterogeneous dataset 
for permeability. However, it should be noted that the 
conclusions that will be drawn here will be valid even 
if one uses another dataset for the permeability. Figure 
\ref{Fig:Brinkman_Marmousi} shows the contours of Marmousi 
dataset. 

The boundary value problem of the reservoir is pictorially 
described in Figure \ref{Fig:reservoir}. This computational 
domain has been employed in some recent works (e.g., 
\citep{Nakshatrala_Rajagopal_IJNMF_2011_v67_p342}). 
All these works have assumed homogeneous medium properties, 
and did not use a reservoir data like the Marmousi dataset. 
Moreover, these studies did not address the use of 
\emph{a posteriori} techniques to access numerical 
accuracy which is the main focus of the current 
paper. The parameters used in this problem are 
provided in Table \ref{Tab:Input_to_COMSOL_for_Marmousi}. 
Figure \ref{Fig:Betti_Marm_smooth_reservoir} shows that 
the errors in the reciprocal relation are larger under 
uniform structured meshes than under adaptive meshes. This 
is due to the fact that uniform structured meshes suffer 
from pollution errors due to the singularity near the 
production well. The variation of the error in the reciprocal 
relation with mesh refinement provides guidelines on how much 
refinement is required to obtain solution of some desired 
accuracy especially in those cases where there are no analytical 
solutions. Figure \ref{Fig:Vort_Marm_smooth_reservoir} shows the 
magnitude of the vorticity, and one can see from the figure that 
the maximum principle for the vorticity is satisfied. 
%
\begin{table}[h]
  \begin{center}
    \caption{\textsf{Synthetic reservoir problem:}~Non-dimensional 
      parameters used in the problem.
      \label{Tab:Input_to_COMSOL_for_Marmousi}}
    \begin{tabular}{c c}
      \hline
      Parameter              & Value \\
      \hline\hline
      $\mu$		& 0.001 \\      
      $\rho$		& 1 \\
      $p_{\mathrm{inj}}$	& 5 and 7.5 \\
      $p_{\mathrm{atm}}$	& 1    \\
      $L$		& 384    \\
      $H$		& 384/2    \\
      $W$		& 384$\times$0.1  \\
      $k$		& Marmousi dataset/$L^2$ \\
      \hline
    \end{tabular}
  \end{center}
\end{table}

\section{TRANSIENT CASE}
\label{Sec:S6_Brinkman_Transient}

\subsection{Governing equations}
Let us denote the time interval of interest by 
$\mathcal{I}$, and the time by $t \in \mathcal{I}$. 
The unsteady governing equations under the Darcy 
model take the following form:
\begin{subequations}
  \begin{alignat}{2}
    &\rho \frac{\partial \mathbf{v}}{\partial t} + 
    \alpha \mathbf{v} + \mathrm{grad}[p] = \rho \mathbf{b}(\mathbf{x},t) && 
    \quad \mathrm{in} \; \Omega \times \mathcal{I} \\
    &\mathrm{div}[\mathbf{v}] = 0 && \quad \mathrm{in} \; \Omega \times \mathcal{I} \\
    &\mathbf{v}(\mathbf{x},t)\cdot \widehat{\mathbf{n}}(\mathbf{x}) = v_n(\mathbf{x},t) 
    && \quad \mathrm{on} \; \Gamma^v \times \mathcal{I} \\
    &p(\mathbf{x},t) = p_0(\mathbf{x},t) && \quad \mathrm{on} \; \Gamma^t 
    \times \mathcal{I} \\
    &\mathbf{v}(\mathbf{x},t=0) = \mathbf{v}_0(\mathbf{x}) 
    && \quad \mathrm{in} \; \Omega 
  \end{alignat}
\end{subequations}
Of course, the convective term $\mathrm{grad}
[\mathbf{v}]\mathbf{v}$ is neglected.
We assume that the coefficient of viscosity of the fluid 
and the permeability of the porous solid to be constants, 
and hence the drag coefficient is constant. We further 
assume that the density is homogeneous, and the body 
force is assumed to be conservative.
The unsteady Darcy-Brinkman equations can be written as follows:
\begin{subequations}
  \begin{alignat}{2}
    &\rho \frac{\partial \mathbf{v}}{\partial t} + 
    \alpha \mathbf{v} + \mathrm{grad}[p] - \mathrm{div}
           [2 \mu \mathbf{D}] = \rho \mathbf{b}(\mathbf{x},t) && 
           \quad \mathrm{in} \; \Omega \times \mathcal{I} \\
           &\mathrm{div}[\mathbf{v}] = 0 && \quad \mathrm{in} \; \Omega 
           \times \mathcal{I} \\
           &\mathbf{v}(\mathbf{x},t) = \mathbf{v}^{\mathrm{p}}(\mathbf{x},t) 
           && \quad \mathrm{on} \; \Gamma^v \times \mathcal{I} \\
           &\mathbf{T} 
           \widehat{\mathbf{n}}(\mathbf{x}) = \mathbf{t}^{\mathrm{p}}
           (\mathbf{x},t) && \quad \mathrm{on} \; \Gamma^t \times 
           \mathcal{I} \\
           &\mathbf{v}(\mathbf{x},t=0) = \mathbf{v}_0(\mathbf{x}) 
           && \quad \mathrm{in} \; \Omega 
  \end{alignat}
\end{subequations}
where $\mathbf{v}^{\mathrm{p}}(\mathbf{x},t)$ is the 
prescribed velocity vector, and $\mathbf{t}^{\mathrm{p}}
(\mathbf{x},t)$ is the prescribed traction.
%
\subsection{Mathematical properties}
Under unsteady Darcy equations, the vorticity 
satisfies the following equation:
\begin{align}
  \rho \frac{\partial \boldsymbol{\omega}}{\partial t} 
  + \alpha \boldsymbol{\omega} = \boldsymbol{0}
\end{align}
which is an ordinary differential equation 
at each spatial point. The solution takes 
the following form: 
\begin{align}
  \boldsymbol{\omega}(\mathbf{x},t) = \boldsymbol{\omega}_0(\mathbf{x}) 
  \exp\left[-\frac{\alpha t}{\rho}\right]
\end{align}
This means that the decay of vorticity should be 
exponential. One can check whether the numerical 
solutions exhibit this trend by plotting in log 
scale the each component of the vorticity with 
respect to time should, and this should be a 
straight line with slope $-\alpha/\rho$. If 
this trend is not satisfied for a given mesh 
and given time-step, does refining the grid 
spacing and the time-step improve the trend? 
Of course, one can check whether the vorticity 
goes to zero for large times.  
The vorticity under unsteady Darcy-Brinkman 
equations satisfies the following equation:
\begin{align}
  \rho \frac{\partial \boldsymbol{\omega}}{\partial t} + 
  \alpha \boldsymbol{\omega} - \mu \Delta \boldsymbol{\omega} 
  = \boldsymbol{0} \quad \mathrm{in} \; \Omega \times \mathcal{I}
\end{align}
where $\Delta$ is the Laplacian operator. The above 
equation is a homogeneous linear parabolic partial 
differential equation, which is known to satisfy a 
maximum principle \citep{Pao}. This implies that both the maximum 
and the minimum will occur either in the initial 
condition or on the boundary. One can check whether 
numerical solutions satisfy the aforementioned 
maximum principle. Also, whether refining the grid 
spacing and time-steps affect the performance of 
numerical solutions with respect to this metric. 
One can devise any test problem as long as the 
following assumptions are met: (i) $\mu$ is 
constant, (ii) permeability is homogeneous, 
and (iii) the body force is conservative.
%
\subsection{Representative numerical results}
The theoretical results presented in this section 
are corroborated numerically in Figures 
\ref{Fig:Vorticity_trans_Darcy_slab}--\ref{Fig:Vorticity_trans_Brinkman}.
The vorticity results for the pressure slab problem  
under the transient Darcy model are provided in 
Figures \ref{Fig:Vorticity_trans_Darcy_slab}.
(Recall that Figure \ref{Fig:press_slab} provided a 
pictorial description of the pressure slab problem.) 
The non-dimensional parameters used in the 
numerical simulation are presented in Table 
\ref{Tab:Input_to_COMSOL_for_slab}, and the 
initial condition is provided in Table 
\ref{Tab:IC_for_slab}.
Figure \ref{Fig:Vorticity_trans_Darcy_slab} shows 
that $\mathrm{log}(\frac{\omega}{\omega_0})$ is a 
straight line with slope equal to $\frac{\alpha}{\rho}$ 
under the mesh and time refinements for various 
spatial points in the solution of the transient 
Darcy equations (for the current problem 
$\frac{\alpha}{\rho}=-1$).
Figure \ref{Fig:Vorticity_trans_Brinkman} verifies 
the maximum principle for vorticity for various 
test problems under transient Darcy-Brinkman 
equations. In all the cases, the non-negative 
maximum and non-positive minimum of the 
vorticity occur on the boundary, which 
agree with the mathematical theory presented 
above. 

\begin{table}
\begin{center}
\caption{\textsf{Pressure slab problem:}~Initial condition used in the problem.
\label{Tab:IC_for_slab}}
\begin{tabular}{c c}
\hline
Parameter              & Value \\
\hline\hline
$p$		& $p_{\mathrm{atm}}$ \\
$v_x=v_y$	&	$\sin(\pi x/W) \sin(\pi y/L)$ \\
\hline
\end{tabular}
\end{center}
\end{table}

%

\section{CONCLUDING REMARKS}
\label{Sec:S8_Brinkman_Closure}
We presented a new methodology 
for solution verification which is independent 
of utilized numerical formulation. We called 
it mechanics-based solution verification.
To employ this approach 
we developed four important properties 
(minimum dissipation theorem, minimum 
total mechanical power, reciprocal relation, 
and maximum principle for vorticity) that 
the solutions under the Darcy and the 
Darcy-Brinkman models satisfy.
However, the proposed technique is 
not merely restricted to porous media type 
problems and one can easily extend it to 
other models. We also presented various 
test problems can serve as benchmark 
problems to verify numerical implementation 
of solvers for the Darcy and Darcy-Brinkman 
models. Results showed that the proposed 
technique can be effectively used to 
assess the accuracy of numerical solutions, 
identify numerical pollution, 
and check the performance of adaptive mesh. 
An attractive feature is that these 
properties can be verified for any given 
problem (i.e., needed not be one of the 
benchmark problems), and for any computational 
domain. For example, if the problem involves 
prescribing velocity boundary condition on 
the entire boundary, then one plots the 
dissipation with respect to mesh refinement. 
Some of the main conclusions are: 
\begin{enumerate}[(a)]
\item If the numerical formulation is converging, the 
  dissipation, total mechanical power, and reciprocal 
  relation should decrease with mesh refinement. If 
  this does not occur, one needs to suspect that 
  there are singularities in the solutions or that 
  the numerical formulation does not perform well 
  with respect to the local mass balance property.
\item It has been shown that the minimum dissipation, 
  minimum total mechanical power, 
  and reciprocal relation theorems can be 
  utilized to identify pollution errors in 
  numerical solutions. The theorems can also be 
  used to assess whether a given type of mesh 
  will be able to resolve singularities in 
  the solution. This can be assessed by 
  creating a series of hierarchical meshes and 
  plotting for example the dissipation, with 
  respect to characteristic 
  mesh size. A given type of mesh will resolve 
  singularities in the solution and will not 
  be affected by pollution errors if the total 
  dissipation decreases uniformly and reaches 
  a plateau with a hierarchical mesh refinement.
\item The non-negative maximum vorticity and the 
  non-positive minimum vorticity under Darcy-Brinkman 
  equations with homogeneous isotropic medium properties 
  should occur on the boundary. 
\item Log plot of vorticity with respect to time, 
under transient Darcy models, is a line with slope 
$-\frac{\alpha}{\rho}$.

\end{enumerate}
The proposed \emph{a posteriori} techniques can be 
invaluable additions to the usual repertoire of 
methods for verification -- the method of exact 
solutions (MES) and the method of manufactured 
solutions (MMS).

%
\section*{APPENDIX:~UNIQUENESS OF SOLUTION}
For completeness and pedagogical value, we now show 
that the uniqueness of solution under Darcy-Brinkman 
equations is a direct consequence of the minimum 
dissipation inequality. One can construct a proof 
for the uniqueness of solution under Darcy equations 
on similar lines. 

\begin{theorem}[\textsf{Uniqueness theorem}]
  \label{Theorem:Uniqueness_theorem}
  The solution to Darcy-Brinkman equations 
  \eqref{Eqn:Brinkman_LM}--\eqref{Eqn:Brinkman_traction_BC} 
  is unique up to an arbitrary constant for the pressure 
  given $\mathbf{b}(\mathbf{x})$, 
  $\mathbf{v}^{\mathrm{p}}(\mathbf{x})$ and 
  $\mathbf{t}^{\mathrm{p}}(\mathbf{x})$. 
\end{theorem}
\begin{proof}
  On the contrary assume that $\left\{\mathbf{v}_1(\mathbf{x}),
  p_1(\mathbf{x})\right\}$ and $\left\{\mathbf{v}_2(\mathbf{x}),
  p_2(\mathbf{x})\right\}$ are two solutions to Darcy-Brinkman 
  equations for the prescribed data. Let us consider the 
  following quantity:
  \begin{align*}
    \mathcal{I} := \int_{\Omega} \alpha(\mathbf{x}) 
    \left(\mathbf{v}_1(\mathbf{x}) 
    - \mathbf{v}_2(\mathbf{x})\right) \cdot \left(\mathbf{v}_1 
    (\mathbf{x}) - \mathbf{v}_2(\mathbf{x})\right) \; \mathrm{d} 
    \Omega
    + \int_{\Omega} 2 \mu \left(\mathbf{D}_1(\mathbf{x}) 
    - \mathbf{D}_2(\mathbf{x}) \right) \cdot 
    \left(\mathbf{D}_1(\mathbf{x}) - \mathbf{D}_2(\mathbf{x}) 
    \right) \; \mathrm{d} \Omega
  \end{align*}
  Noting that $\mathrm{div}[\mathbf{v}_1] = 0$ 
  and $\mathrm{div}[\mathbf{v}_2] = 0$, the second 
  integral can be simplified as follows: 
  \begin{align*}
    \int_{\Omega} 2 \mu 
    \left(\mathbf{D}_1(\mathbf{x}) 
    - \mathbf{D}_2(\mathbf{x})\right) \cdot 
    \left(\mathbf{D}_1 (\mathbf{x}) - \mathbf{D}_2
    (\mathbf{x})\right) \; \mathrm{d} \Omega
    &= \int_{\Omega} \left(\mathbf{D}
    _1(\mathbf{x}) -\mathbf{D}_2(\mathbf{x}) \right) \cdot 
    \left(\mathbf{T}_1(\mathbf{x}) - \mathbf{T}_2(\mathbf{x}) 
    \right) \; \mathrm{d} \Omega \\
    &= \int_{\Omega} \mathrm{grad}\left[\mathbf{v}_1(\mathbf{x}) 
      - \mathbf{v}_2(\mathbf{x})\right] \cdot \left(\mathbf{T}_1 
    (\mathbf{x}) - \mathbf{T}_2(\mathbf{x})\right) \; \mathrm{d} 
    \Omega 
  \end{align*}
  In obtaining the above equation, we have used 
  the fact that $\mathbf{T}_1(\mathbf{x})$ and 
  $\mathbf{T}_2(\mathbf{x})$ are symmetric. Using 
  Green's identity, the above equation can be 
  written as follows: 
  \begin{align*}
    \int_{\Omega} 2 \mu 
    \left(\mathbf{D}_1(\mathbf{x}) 
    - \mathbf{D}_2(\mathbf{x})\right) \cdot 
    \left(\mathbf{D}_1 (\mathbf{x}) - \mathbf{D}_2
    (\mathbf{x})\right) \; \mathrm{d} \Omega
    &= \int_{\partial \Omega} (\mathbf{v}_1(\mathbf{x}) - 
    \mathbf{v}_2(\mathbf{x})) \cdot (\mathbf{T}_1(\mathbf{x}) 
    - \mathbf{T}_2(\mathbf{x})) \widehat{\mathbf{n}} 
    (\mathbf{x}) \; \mathrm{d} \Gamma \\
    &- \int_{\Omega} (\mathbf{v}_1(\mathbf{x}) - \mathbf{v}_2(\mathbf{x})) 
    \cdot (\mathrm{div}[\mathbf{T}_1] - \mathrm{div}[\mathbf{T}_2]) 
    \; \mathrm{d} \Omega
  \end{align*}
  Using boundary conditions and the balance 
  of linear momentum, we get the following:
  \begin{align*}
    \int_{\Omega} 2 \mu 
    \left(\mathbf{D}_1(\mathbf{x}) 
    - \mathbf{D}_2(\mathbf{x})\right) \cdot 
    \left(\mathbf{D}_1 (\mathbf{x}) - \mathbf{D}_2
    (\mathbf{x})\right) \; \mathrm{d} \Omega
    =- \int_{\Omega} \alpha(\mathbf{x}) \left(\mathbf{v}_1(\mathbf{x}) 
    - \mathbf{v}_2(\mathbf{x})\right) \cdot \left(\mathbf{v}_1 
    (\mathbf{x}) - \mathbf{v}_2(\mathbf{x})\right) \; \mathrm{d} 
    \Omega	
  \end{align*}
  This implies that $\mathcal{I} = 0 $. 
  Since $\alpha(\mathbf{x}) > 0 \; \forall 
  \mathbf{x} \in \Omega$ and $\mu > 0 \; 
  \mbox{in} \; \Omega$, one can conclude 
  that 
  \begin{subequations}
    \begin{align}
      &\mathbf{v}_1(\mathbf{x}) = \mathbf{v}_2(\mathbf{x}) \quad 
      \forall \mathbf{x} \in \Omega \\
      &\mathbf{D}_1(\mathbf{x}) = \mathbf{D}_2(\mathbf{x}) \quad 
      \forall \mathbf{x} \in \Omega
    \end{align}
  \end{subequations}
  That is, the velocity and symmetric part of 
  velocity vector field are unique. Using the 
  equation for the balance of linear momentum 
  \eqref{Eqn:Brinkman_LM} and the fact that 
  the velocity vector field is unique, one 
  can obtain the following equation: 
  \begin{align}
    \mathrm{grad}\left[p_1(\mathbf{x}) - p_2(\mathbf{x})\right] 
    = \mathbf{0} \quad \forall \mathbf{x} \in \Omega
  \end{align}
  This further implies that 
  \begin{align}
    p_1(\mathbf{x}) - p_2(\mathbf{x}) = p_0 
    \quad \forall \mathbf{x} \in \Omega
  \end{align}
  where $p_0$ is an arbitrary constant. 
  This completes the proof.
\end{proof}

\section*{ACKNOWLEDGMENTS}
The authors acknowledge the support from the Department 
of Energy through Subsurface Biogeochemical Research 
Program. Neither the United States Government nor any 
agency thereof, nor any of their employees, makes any 
warranty, express or implied, or assumes any legal 
liability or responsibility for the accuracy, 
completeness, or usefulness of any information. 
The opinions expressed in this paper are those of 
the authors and do not necessarily reflect that of 
the sponsor(s). 

\bibliographystyle{plainnat}
\bibliography{Master_References/Books,Master_References/Master_References}
\newpage
\clearpage
\begin{figure}
  \centering
  \subfigure[Q2P1 interpolation for nine-node 
    quadrilateral (Q9) element]
    {\includegraphics[scale=1.30]
    {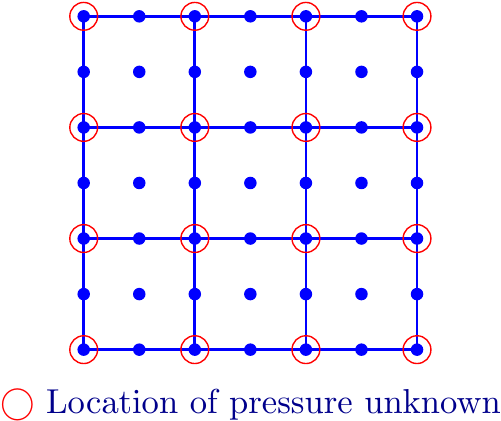}
    \label{Fig:structured_Q9_ele}}
  \qquad
  \subfigure[P2P1 interpolation for six-node triangular 
    (T6) element]{\includegraphics[scale=1.30]
    {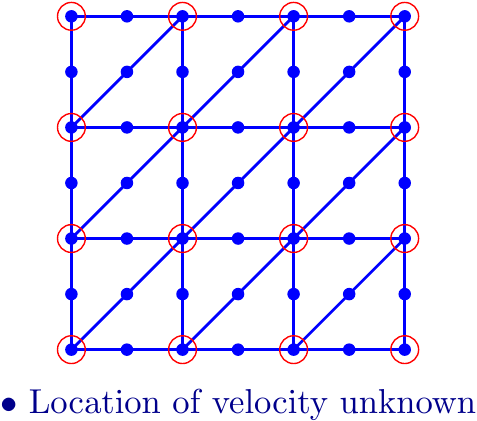}
    \label{Fig:structured_T6_ele}}
  \caption{This figure shows the typical structured 
    finite element meshes employed in the current 
    study. We use Q2P1 and P2P1 mixed interpolations 
    to approximate the unknowns (i.e., second-order 
    interpolation for the velocity field, and 
    first-order interpolation for the pressure 
    field). \label{Fig:quadrilateral_triangular_ele}}
\end{figure}

\let\thefootnote\relax\footnote{Results for 
\textsf{Body force problem}}

\begin{figure}
  \includegraphics[scale=0.35]
  {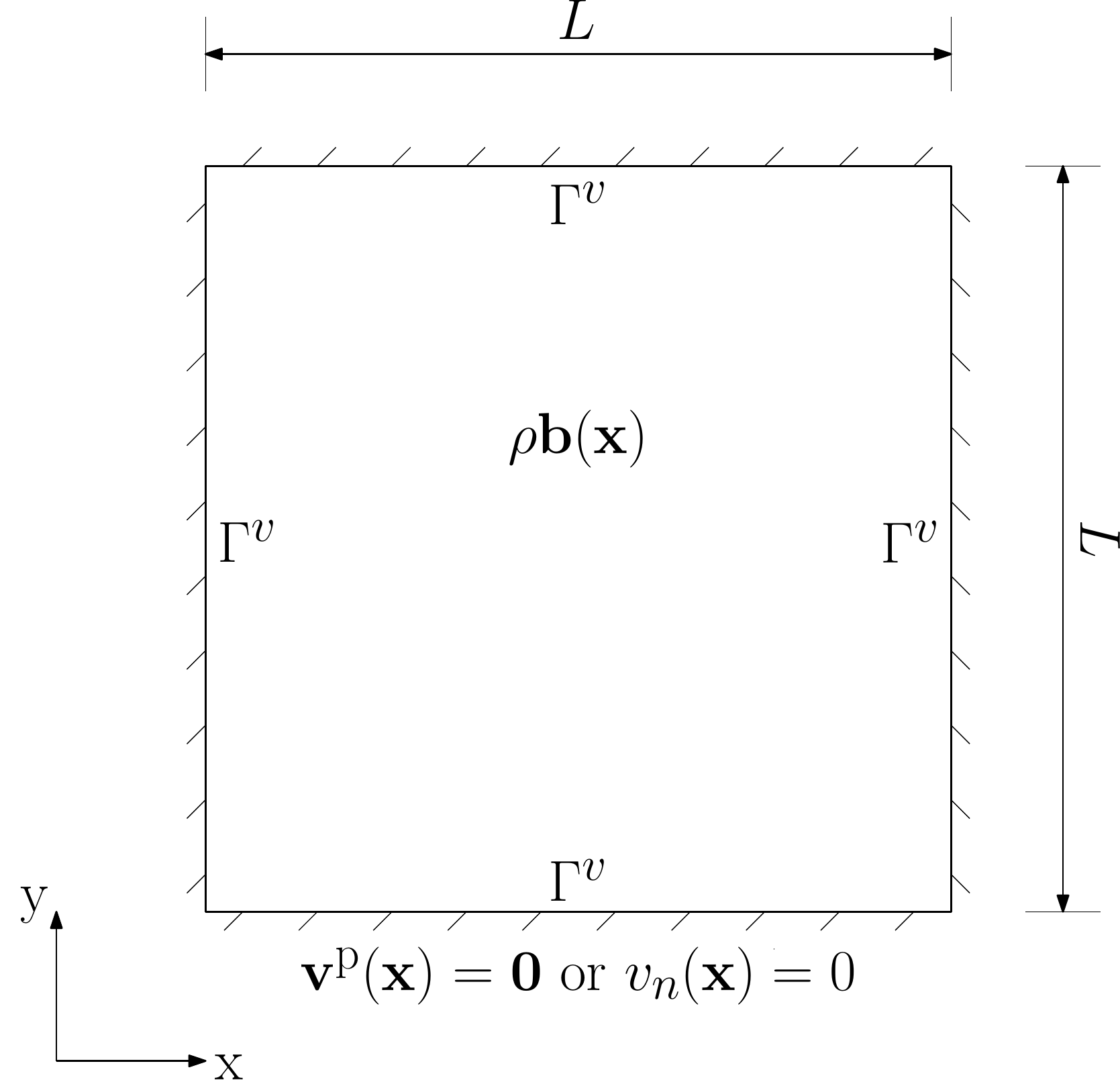}
  \caption{\textsf{Body force problem:}
  ~The computational domain is a square with $L=1$. 
    The prescribed conservative body force is 
    $\rho \mathbf{b}(\mathbf{x}) = 10\times[\mathrm{sin}
      (\pi x), \mathrm{cos}(\pi y)]$.
    Homogeneous velocity is enforced on the 
    entire boundary (i.e., $\Gamma^{v} = \partial 
    \Omega$). That is, $\mathbf{v}^{\mathrm{p}}
    (\mathbf{x}) = \mathbf{0}$ for Darcy-Brinkman 
    equations and $v_{n}(\mathbf{x}) = 0$ for Darcy 
    equations. 
    \label{Fig:Test_prob}}
\end{figure}

\begin{figure}  
  \subfigure[Darcy model with $\mu = 1$.]{
    \includegraphics[clip, scale=0.30]
    {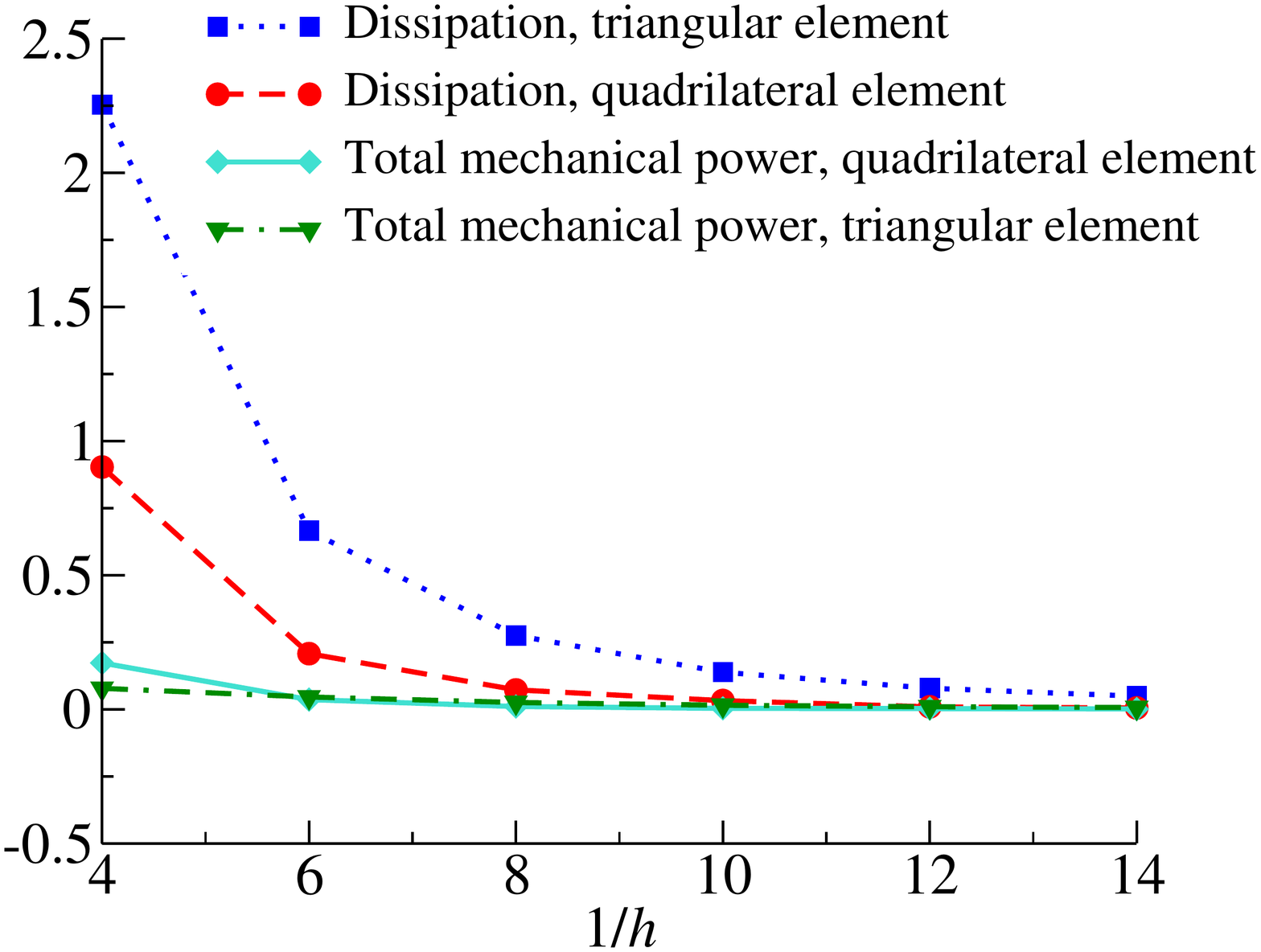} 
    \label{Fig:Dr_test_prob}}
  \subfigure[Darcy-Brinkman model with $\mu = 0.001$.]{
    \includegraphics[clip, scale=0.27]
    {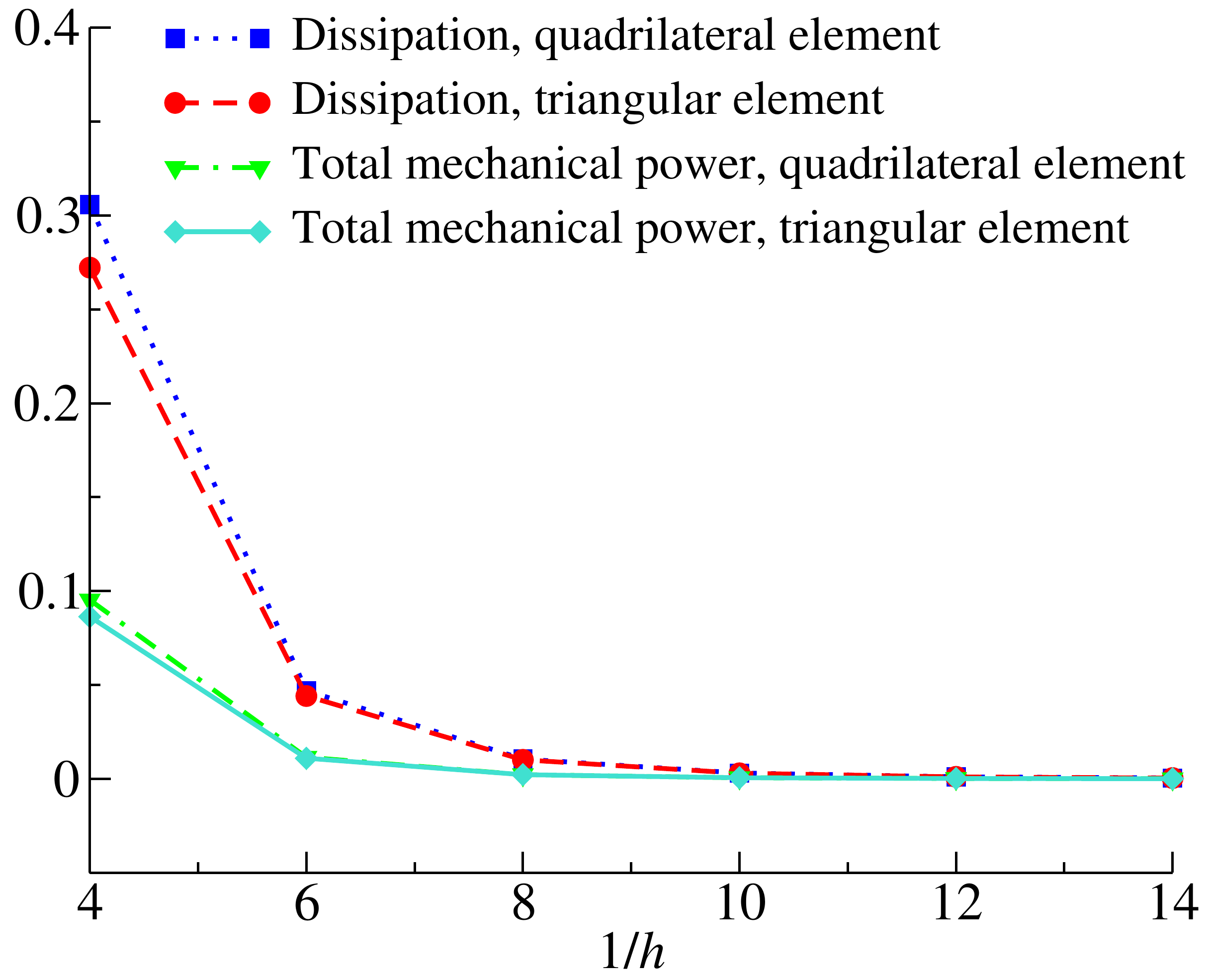} 
    \label{Fig:Br_test_prob}}
  \caption{\textsf{Body force problem:}~This figure 
    shows the variation of the dissipation and the 
    total mechanical power with mesh refinement 
    under the Darcy and Darcy-Brinkman models 
    using quadrilateral and triangular elements.
    The parameters used in this problem are provided 
    in Table \ref{Tab:Input_to_COMSOL_for_body_force}.
    The figure show that the total mechanical power 
    and the dissipation decrease uniformly with mesh 
    refinement, which are in accordance with Theorems 
    \ref{Theorem:Minimum_total_mechanical_power} and 
    \ref{Theorem:Minimum_dissipation}.
    \label{Fig:test_prob_diss}}
\end{figure}

\begin{figure}
  \subfigure[Darcy model with $\mu = 1$.]
  {\includegraphics[scale=0.28,clip]
    {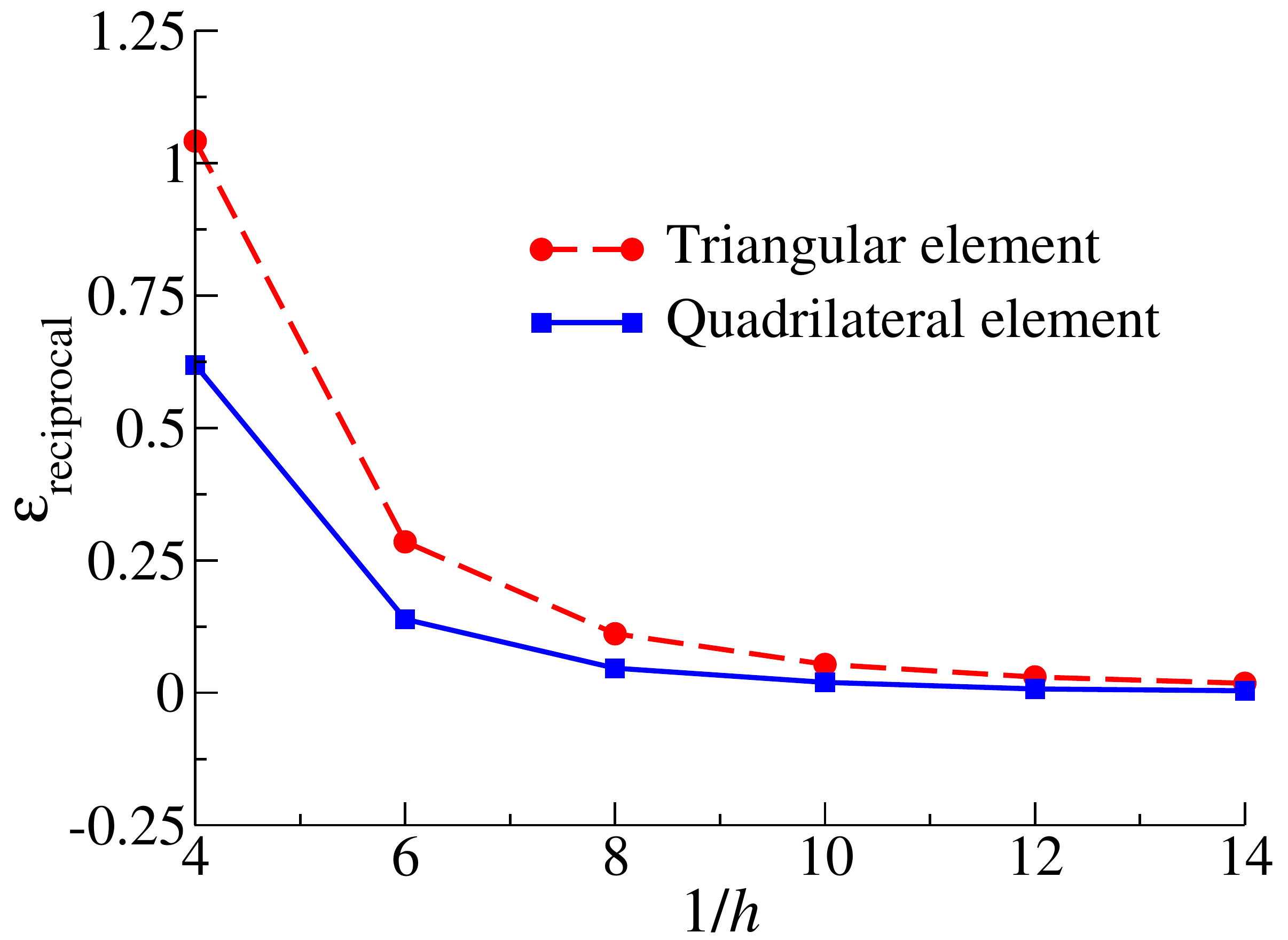}
    \label{Fig:Dr_test_prob_Betti}}
  \qquad
  \subfigure[Darcy-Brinkman model with $\mu = 0.001$.]
  {\includegraphics[scale=0.32,clip]
    {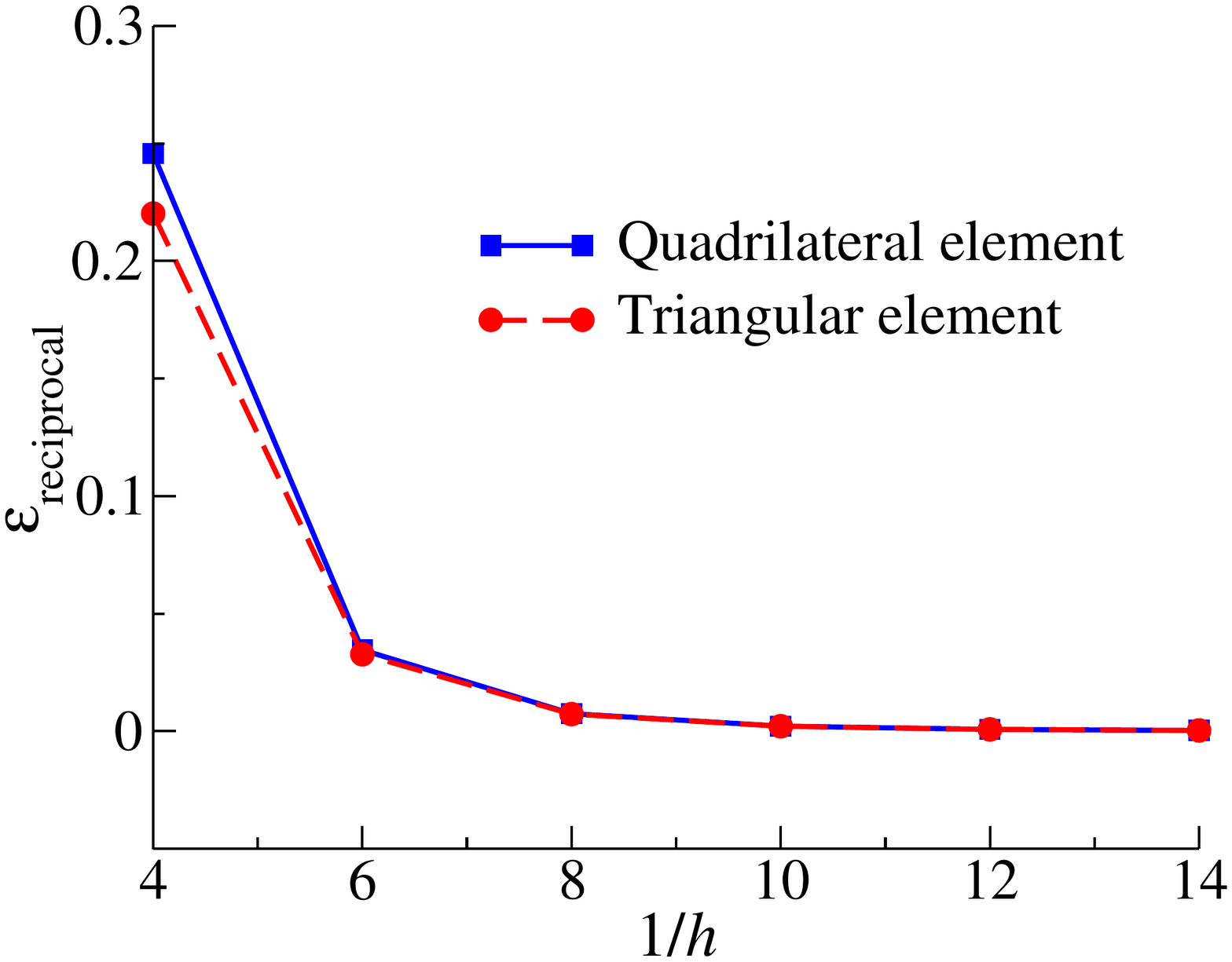}
    \label{Fig:Br_test_prob_Betti}}
  \caption{\textsf{Body force problem:}~This figure shows 
    the variation of $\varepsilon_{\mathrm{reciprocal}}$ with 
    mesh refinement for the Darcy and Darcy-Brinkman 
    models using quadrilateral and triangular elements.
    The parameters used in this problem is provided 
    in Table \ref{Tab:Input_to_COMSOL_for_body_force}.
    The figure show that the numerical error in the 
    reciprocal relation decrease uniformly to zero 
    with mesh refinement for this test problem. 
    \label{Fig:test_prob_Betti}}
\end{figure}

\clearpage 
\newpage 

\let\thefootnote\relax\footnote{Results for 
\textsf{Lid-driven cavity problem}}

\begin{figure}
  \includegraphics[scale=0.35]
  {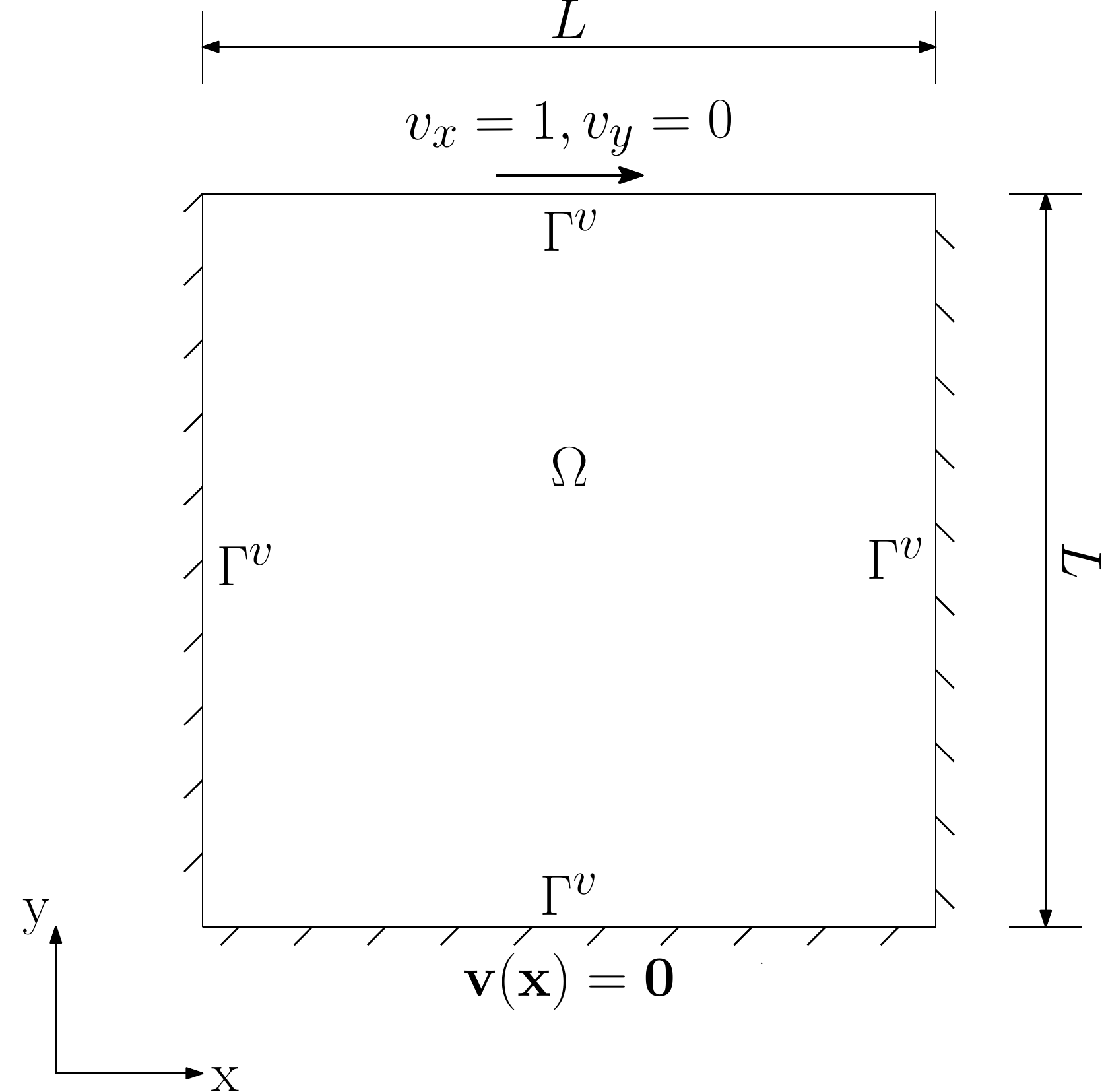}
  \caption{\textsf{Lid-driven cavity problem:}~The 
    computational domain is a unit square. Velocity 
    boundary condition 
    is prescribed on the entire boundary (i.e., 
    $\Gamma^{v} = \partial \Omega$). The prescribed 
    velocity on the top side is $v_{x}=1$ and $v_y 
    = 0$, and the prescribed velocity on the remaining 
    sides of the boundary is zero (i.e., 
    $\mathbf{v}^{\mathrm{p}}(\mathbf{x}) = \mathbf{0}$). 
    Note that this test problem is \emph{not} compatible 
    with Darcy equations. 
    \label{Fig:lid_driven_cavity_problem}}
\end{figure} 
%
\begin{figure}
  \centering
  \subfigure[Uniform structured mesh with $1/h=12$.]
  {\includegraphics[clip,scale=0.72]
    {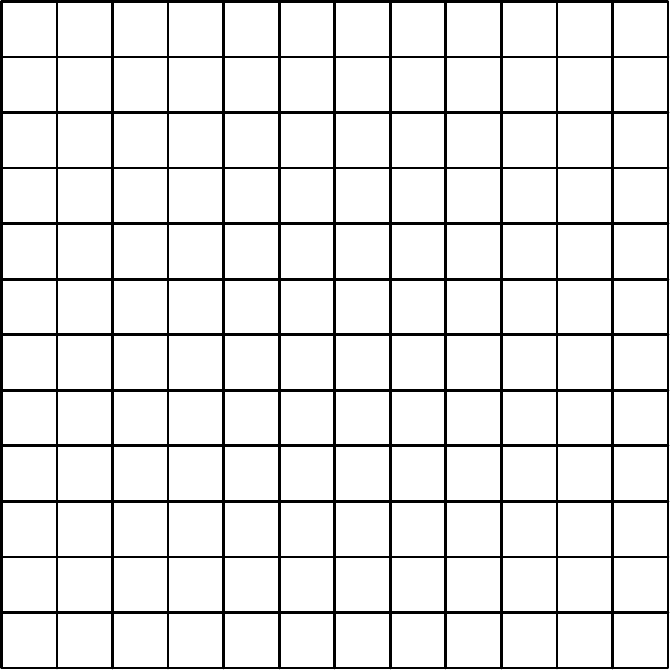}
    \label{Fig:lid_cavity_uni_Q9_4}} 
  \qquad \quad
  \subfigure[Adaptive mesh near the top left corner 
  with $1/h=12$ elsewhere.]
  {\includegraphics[clip,scale=0.86]
    {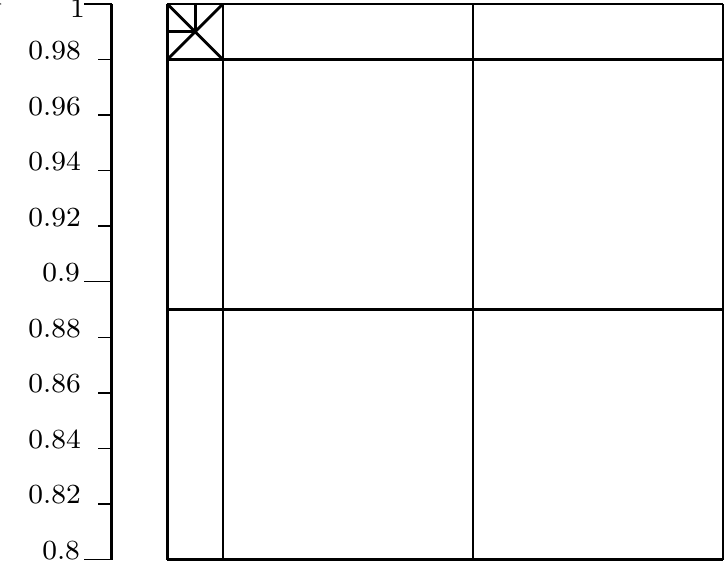}
    \label{Fig:lid_cavityadapt_Q9_4_top}}
  \subfigure[Dissipation vs. $h$ under a hierarchy 
    of uniform structured meshes.]
  {\includegraphics[clip,scale=0.26]
    {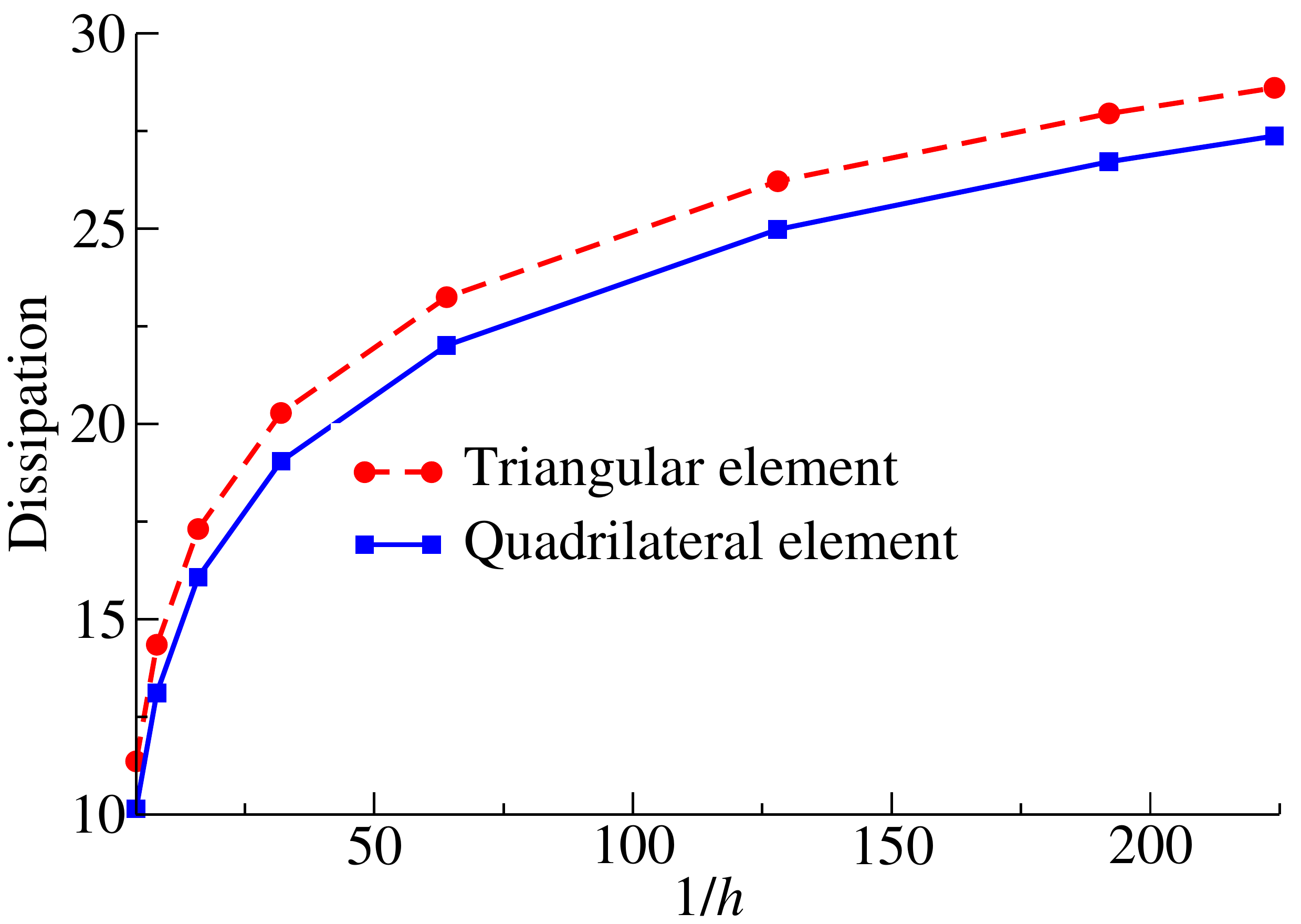}
    \label{Fig:Br_cavity_min_diss}}
  \qquad
  \subfigure[Dissipation vs. $h$ under a hierarchy 
  of adaptive meshes.]
  {\includegraphics[clip,scale=0.3]
    {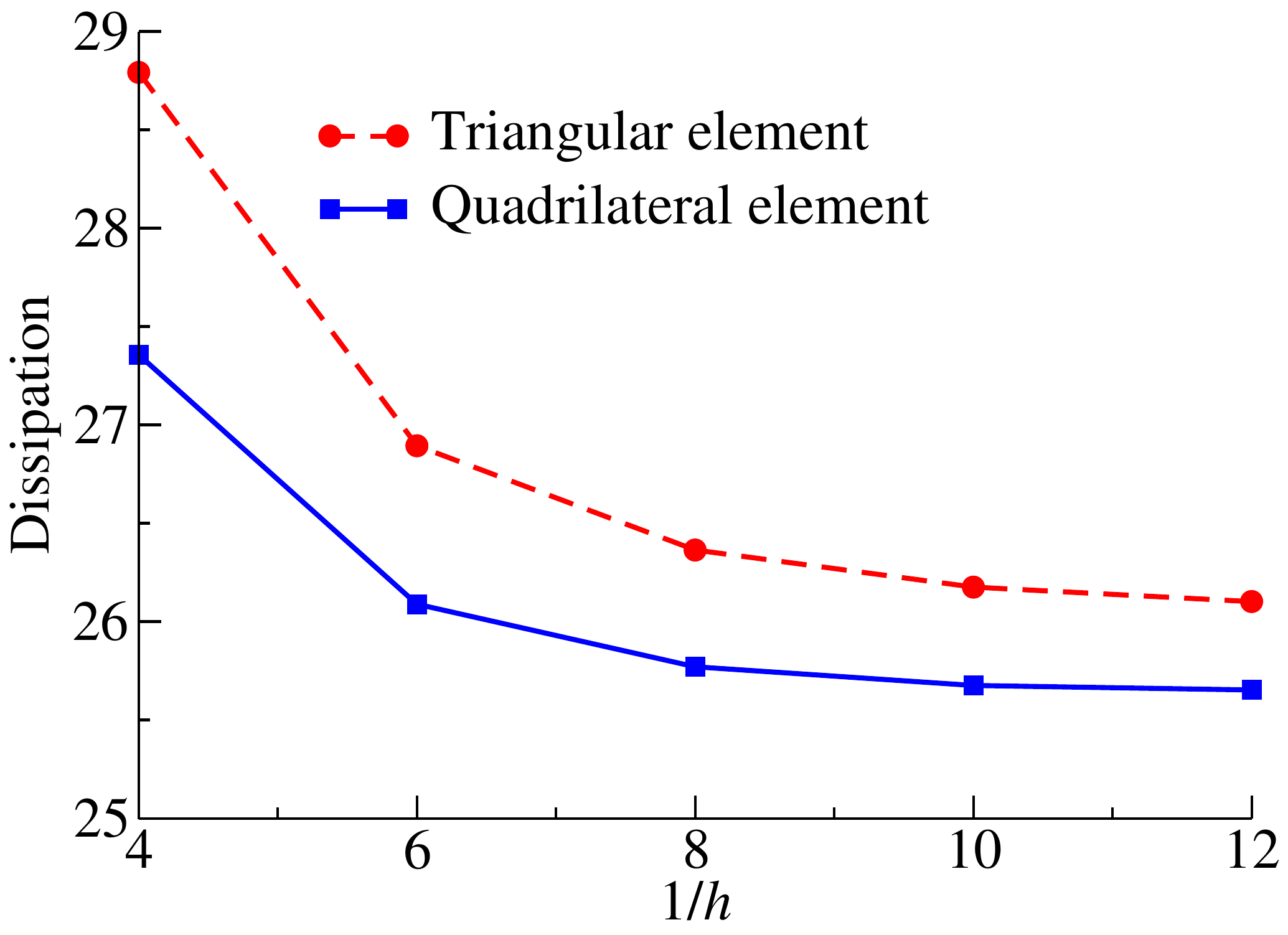}
    \label{Fig:Br_fine_top_cavity_min_diss}}
  \caption{\textsf{Lid-driven cavity problem:}~The 
    top figures show a uniform structured mesh, 
    and an adaptive mesh near the top corners.
    The bottom figures show the variation of total 
    dissipation with mesh refinement for quadrilateral 
    and triangular finite elements. 
    The parameters in 
    this problem are provided in Table 
    \ref{Tab:Input_to_COMSOL_for_lid_cavity}. 
    One should note that the solution exhibits 
    singularities at the top corners. It is 
    evident that the convergence is uniform 
    under both structured and adaptive meshes. 
    However, the dissipation under the structured 
    mesh increased with mesh refinements, whereas 
    the trend is reversed under the adaptive mesh. 
    This is because the adaptive mesh resolves the 
    singularities and removes the associated pollution 
    error. 
    \emph{This figure corroborates one of the 
      main conclusions of the paper that the minimum 
      dissipation theorem can be utilized to identify 
      pollution errors in numerical solutions.} 
      \label{Fig:cavity_min_diss}}
\end{figure}

\newpage 
\clearpage 

\let\thefootnote\relax\footnote{Results for 
\textsf{Pipe bend problem}}

\begin{figure}[h]
  \includegraphics[scale=0.35]
  {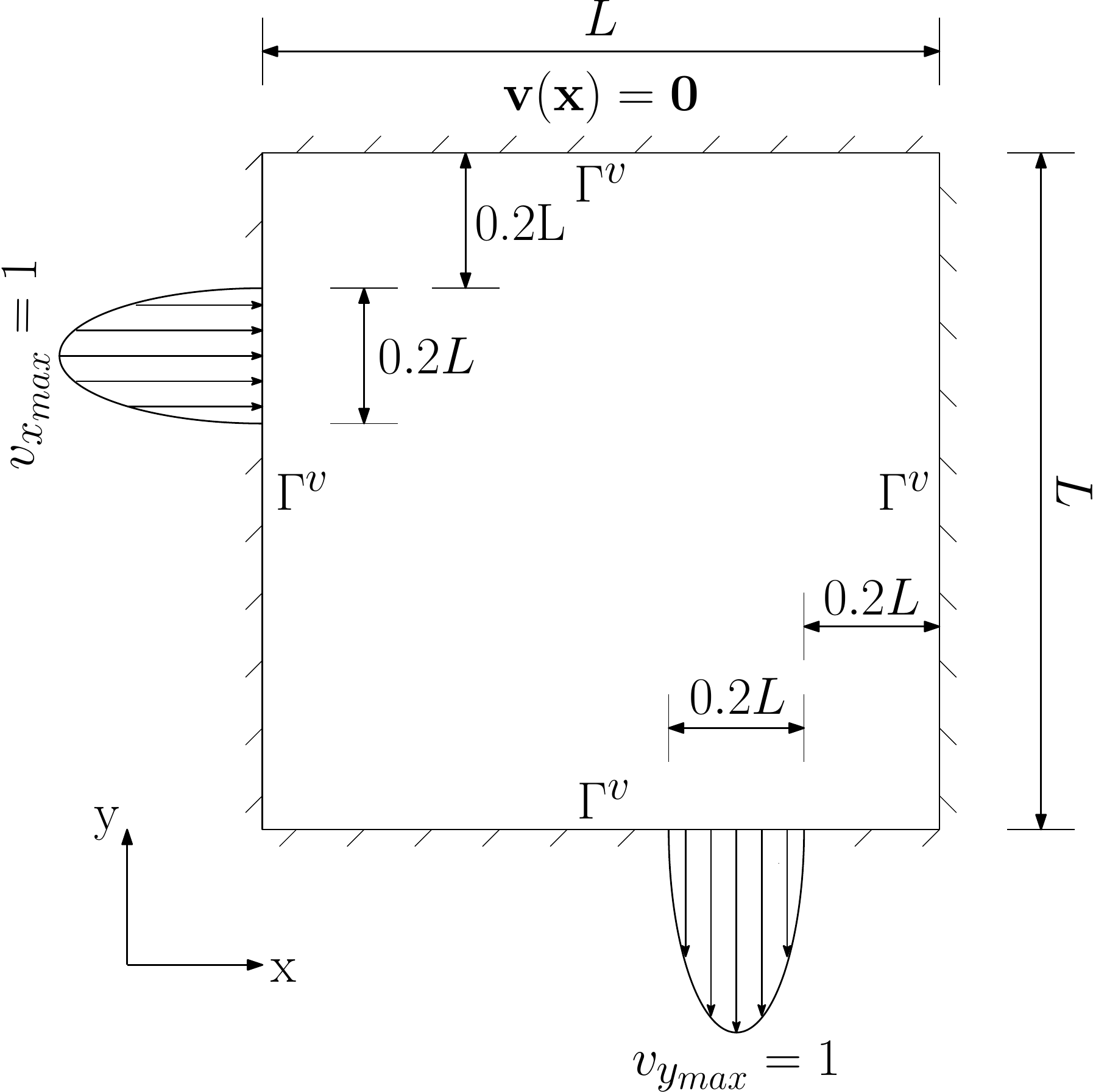}
  \caption{\textsf{Pipe bend problem (velocity 
      boundary condition):}~A pictorial description 
    of the problem. The computational domain $\Omega$ 
    is a unit square. 
    The velocity boundary condition is prescribed 
    on the entire boundary (i.e., $\Gamma^{v} = \partial 
    \Omega$). As indicated in the figure, a parabolic 
    velocity profile with $v_{x_{max}} = 1$ is prescribed 
    on a segment of the left side of the boundary. 
    Similarly, a parabolic velocity profile with 
    $v_{y_{max}} = 1$ is prescribed on a segment of 
    the bottom side of the boundary.  
    Homogeneous velocity is enforced on the 
    remaining parts of the boundary. 
    \label{Fig:pipe_bend}}
\end{figure}

\begin{figure}[h]
  \subfigure[Dissipation for $\alpha=10$.]
  {\includegraphics[scale=0.25]
    {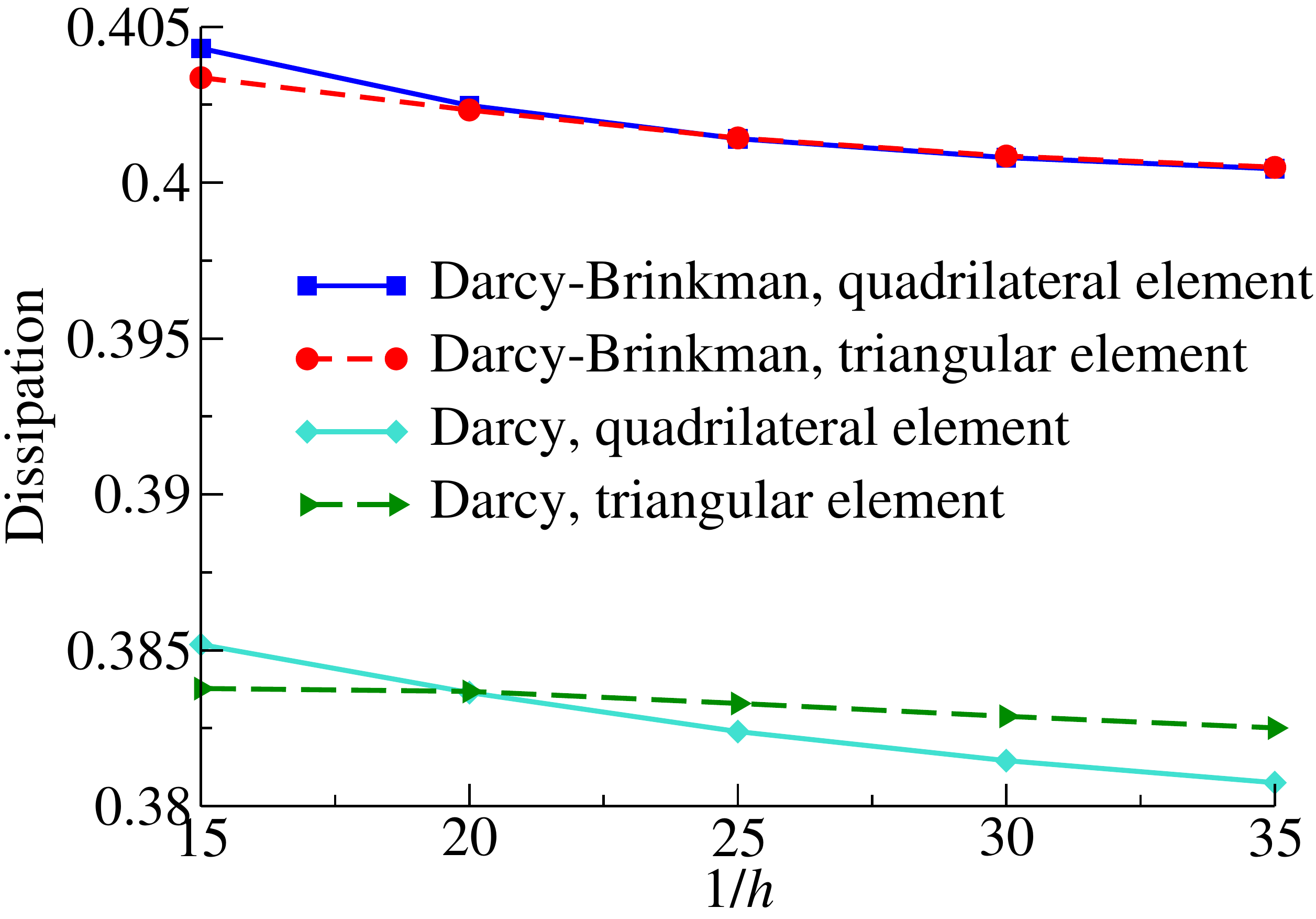}
    \label{Fig:Dr_vs_Br_pipe_min_diss}}
  \qquad
  \subfigure[Total mechanical power for $\alpha=1$.]
  {\includegraphics[scale=0.30]
    {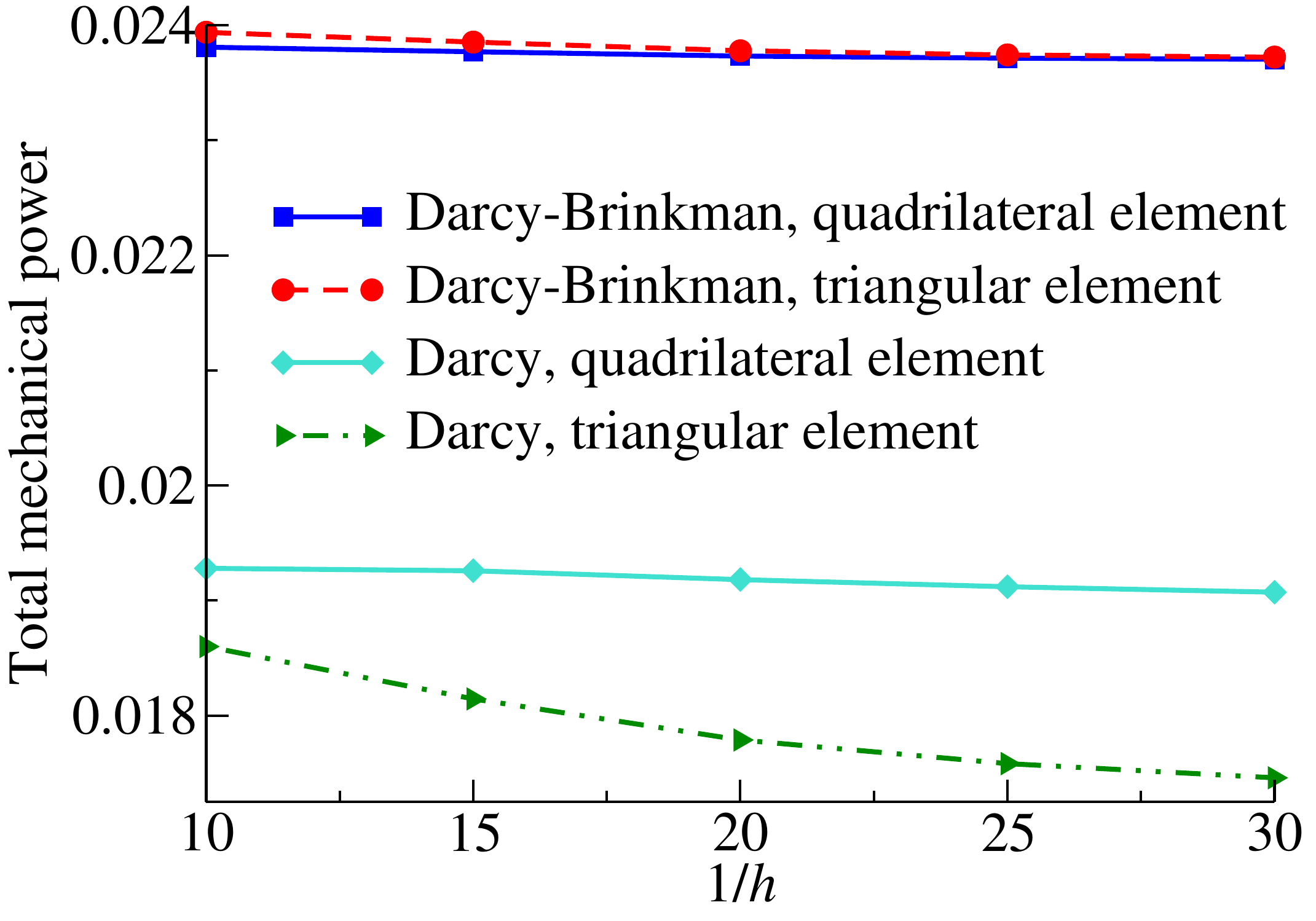}
    \label{Fig:Dr_vs_Br_pipe_min_mech_pw}}
  \caption{\textsf{Pipe bend problem (velocity 
      boundary condition):}~The figure shows that 
    the dissipation and total mechanical power 
    decrease uniformly with mesh refinement for the 
    Darcy and Darcy-Brinkman models. The results are 
    presented for both quadrilateral and triangular 
    elements. 
    The parameters used in this problem are provided in 
    Table \ref{Tab:Input_to_COMSOL_for_2D_pipe_bend} 
    ($\mu=1$ for Darcy and $\mu=0.001$ for Darcy-Brinkman models). 
    Note that the total mechanical power is defined 
    in equation \eqref{Eqn:Brinkman_TMP}. 
    \label{Fig:pipe_bend_diss}}
\end{figure}

\begin{figure}
  \includegraphics[scale=0.35]
  {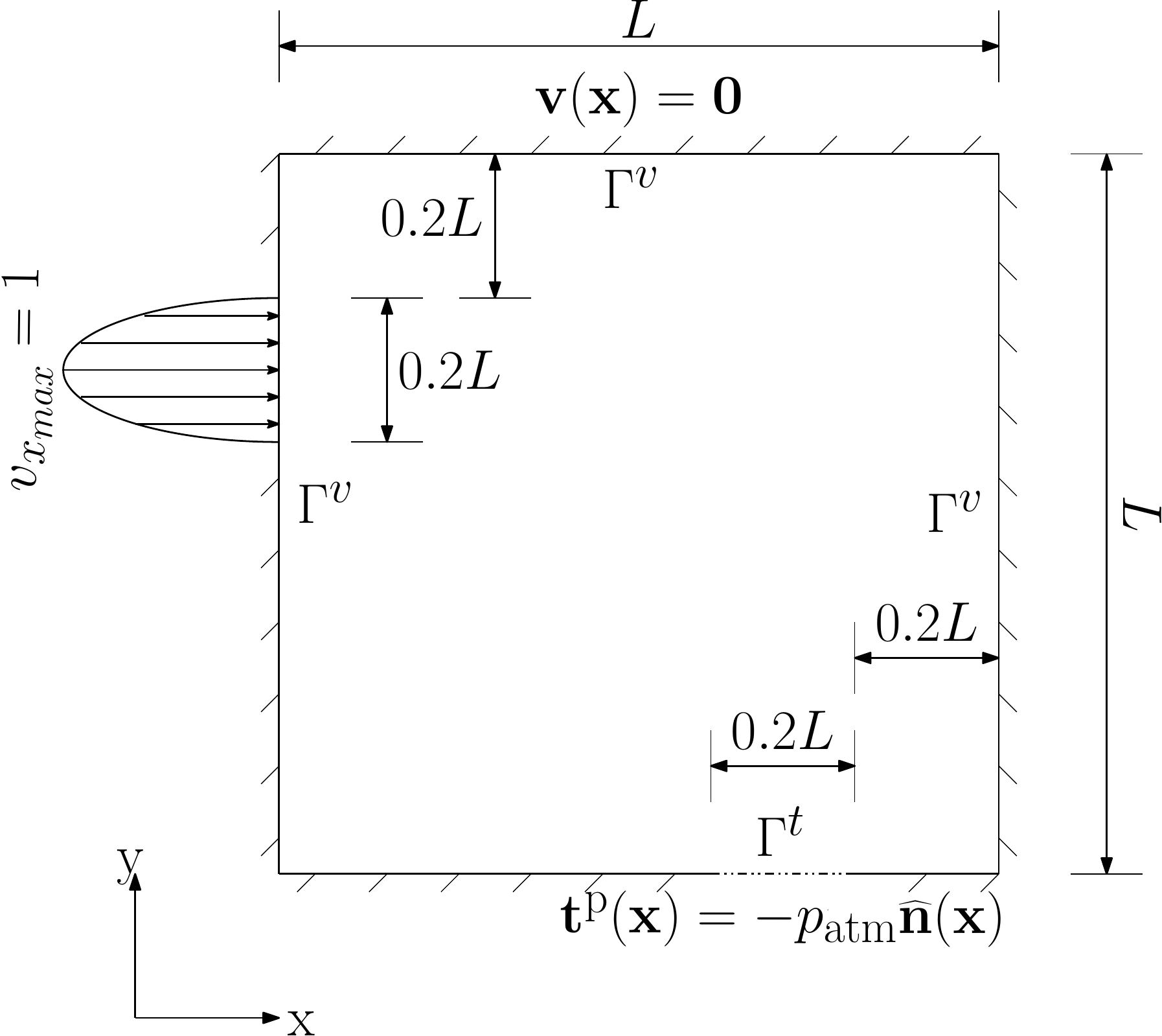} 
  \caption{\textsf{Pipe bend problem (velocity-pressure 
    boundary condition):}~This figure presents a pictorial 
    description of the problem. The computational 
    domain $\Omega$ is a square with $L=1$. 
    The traction boundary condition is $\mathbf{t}^\mathrm{p}
    (\mathbf{x}) = -p_{\mathrm{atm}} \widehat{\mathbf{n}}
    (\mathbf{x}) \; \mathrm{on} \; \Gamma^{t}$. 
    A parabolic velocity profile with ${v}_{{x}_{max}}=1$ 
    is prescribed on a segment of the left side of the 
    boundary, as indicated in the figure. Elsewhere, 
    the velocity is assumed to be zero (i.e., 
    $\mathbf{v}^{\mathrm{p}}(\mathbf{x}) = \mathbf{0}$ for 
    Darcy-Brinkman equations and $v_{n}(\mathbf{x}) = 0$ 
    for Darcy equations). Note that the corners at the 
    outlet are re-entrant corners. 
    \label{Fig:2D_press_pipe_bend}}
\end{figure}

\begin{figure}
  \centering
  \subfigure[Uniform structured mesh with $1/h=15$.]
  {\includegraphics[clip,scale=0.465]
  {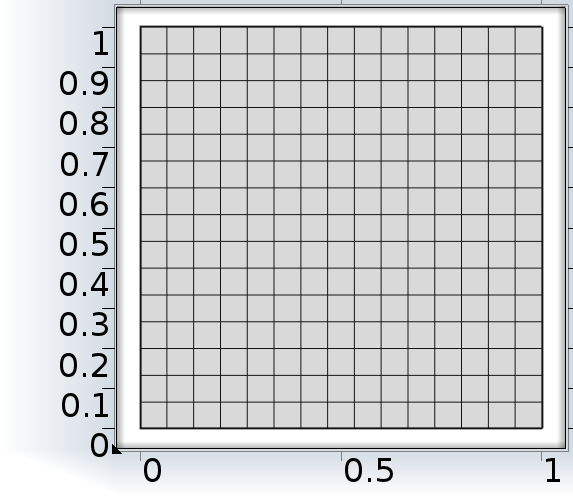}
    \label{Fig:press_pipe_uni_Q9_15}}
  \hspace{1.75cm}
  \subfigure[Adaptive mesh near the outlet with $1/h=15$ 
    elsewhere.]{\includegraphics[clip,scale=0.37]
    {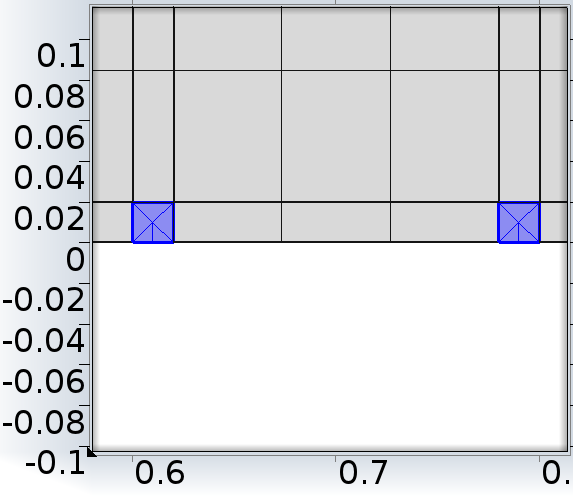}
    \label{Fig:press_pipe_adapt_Q9_15}}
  %
    %
  %
  \subfigure[Darcy-Brinkman model under the uniform structured mesh, 
  $\mu=0.001$.]
  {\includegraphics[scale=0.265]
  {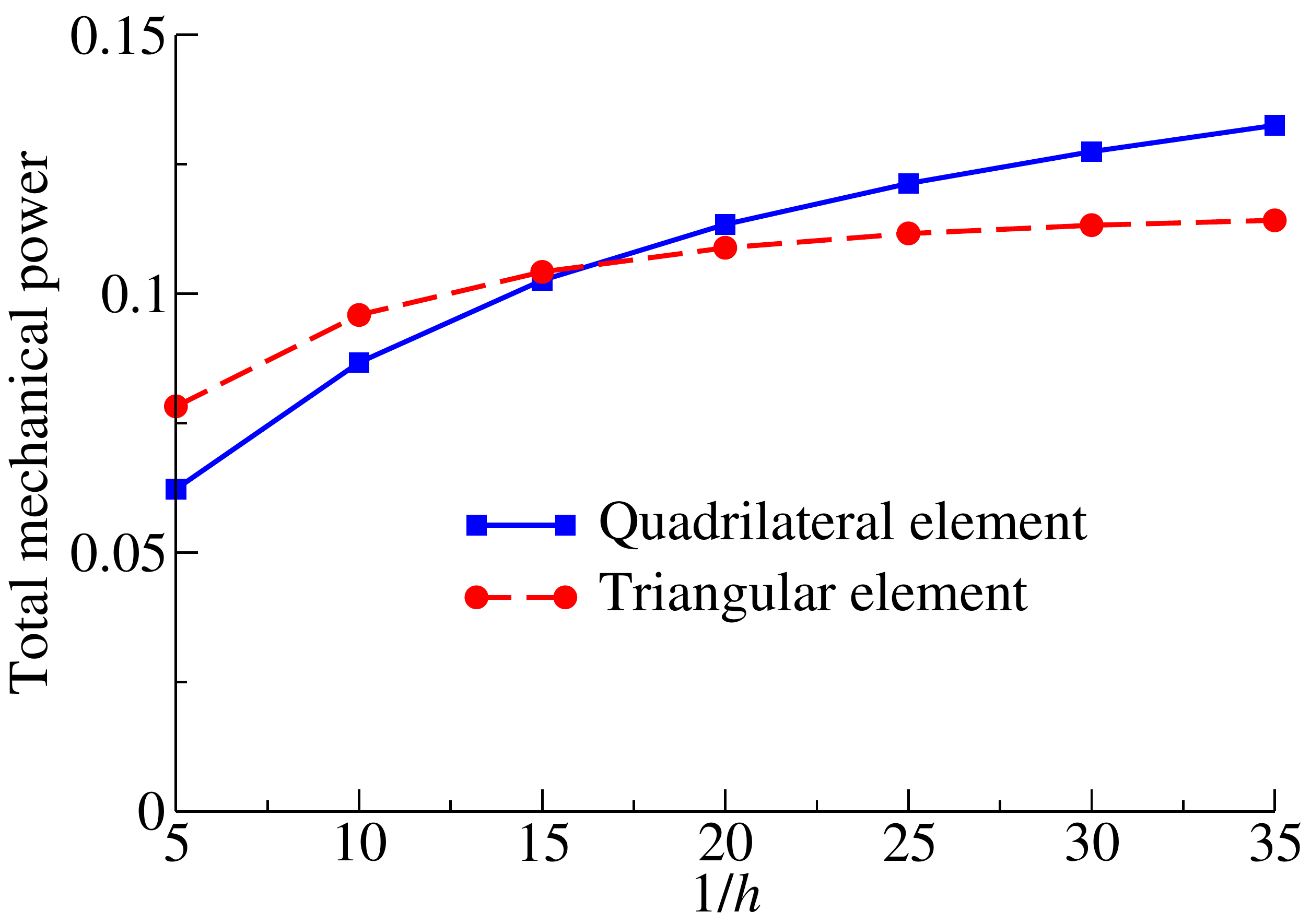} 
    \label{Fig:Br_press_pipe_min_mech_pw}}
  \hspace{.25cm}
  \subfigure[Darcy-Brinkman model under the adaptive mesh, 
  $\mu=0.001$.]
  {\includegraphics[scale=0.265]
  {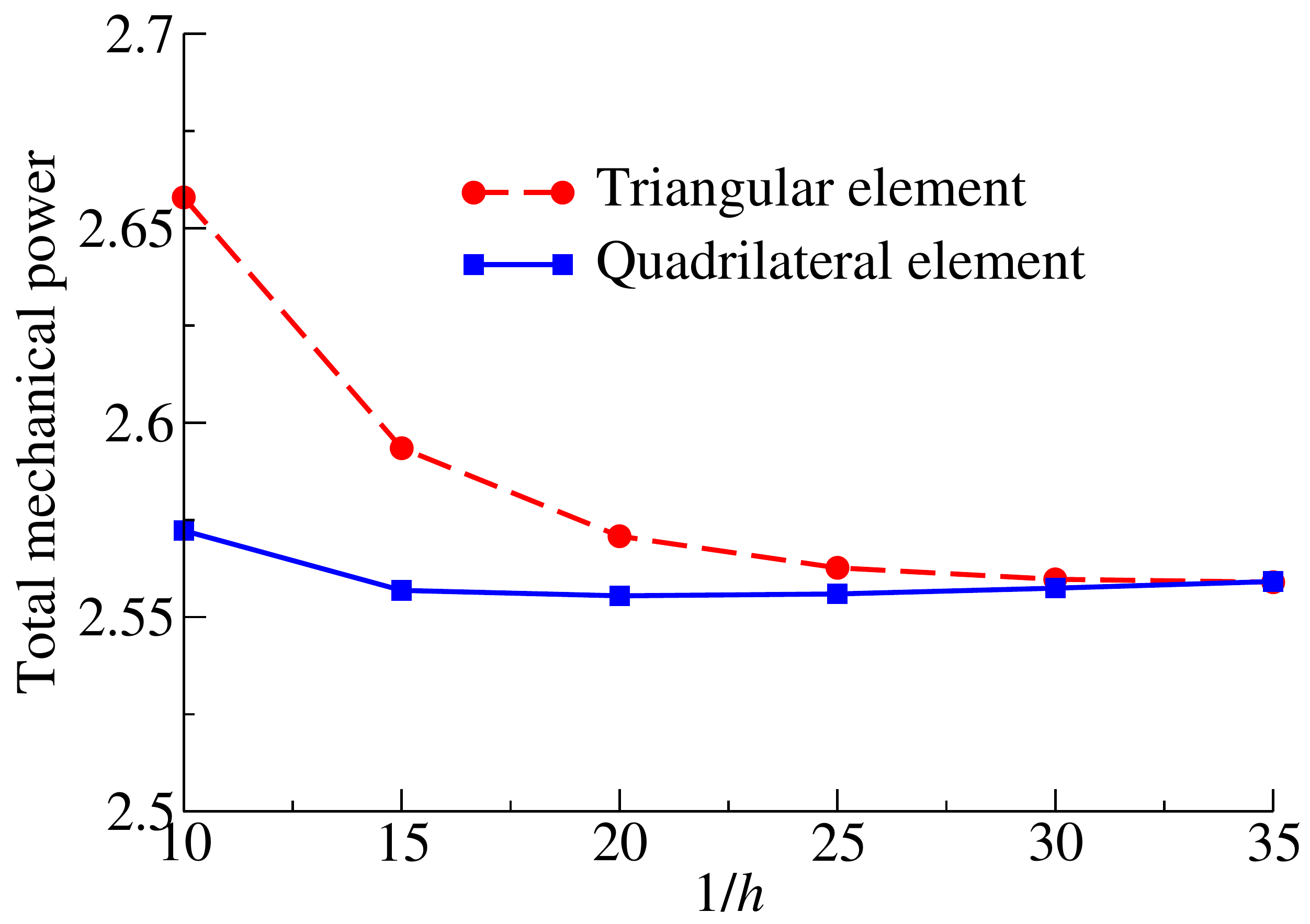} 
    \label{Fig:Br_press_pipe_Adapt_mesh_mech_pw}}
  \caption{\textsf{Pipe bend problem (velocity-pressure 
      boundary condition):}~The top figures show 
    the uniform structured mesh and the adaptive 
    mesh near the corners of the outlet (i.e., 
    $\Gamma^{t}$).
    The bottom figures show the variation of the 
    total mechanical power with mesh refinement 
    under quadrilateral and triangular elements.
    The parameters used in this problem are provided in 
    Table \ref{Tab:Input_to_COMSOL_for_2D_pipe_bend}. 
    The total mechanical power increased uniformly with 
    mesh refinement under the structured mesh, while the 
    total mechanical power decreased uniformly and reached 
    a plateau with mesh refinement under the adaptive mesh.
    The numerical results clearly demonstrate that the 
    total mechanical theorem can be utilized to identify 
    pollution errors due to singularities 
    and to assess whether a particular type of 
    computational mesh is suitable for a given 
    problem. 
    \label{Fig:press_pipe_min_mech_pw}}
\end{figure}

\let\thefootnote\relax\footnote{Results for 
\textsf{Pressure slab problem}}

\begin{figure}
  \includegraphics[scale=0.375]
  {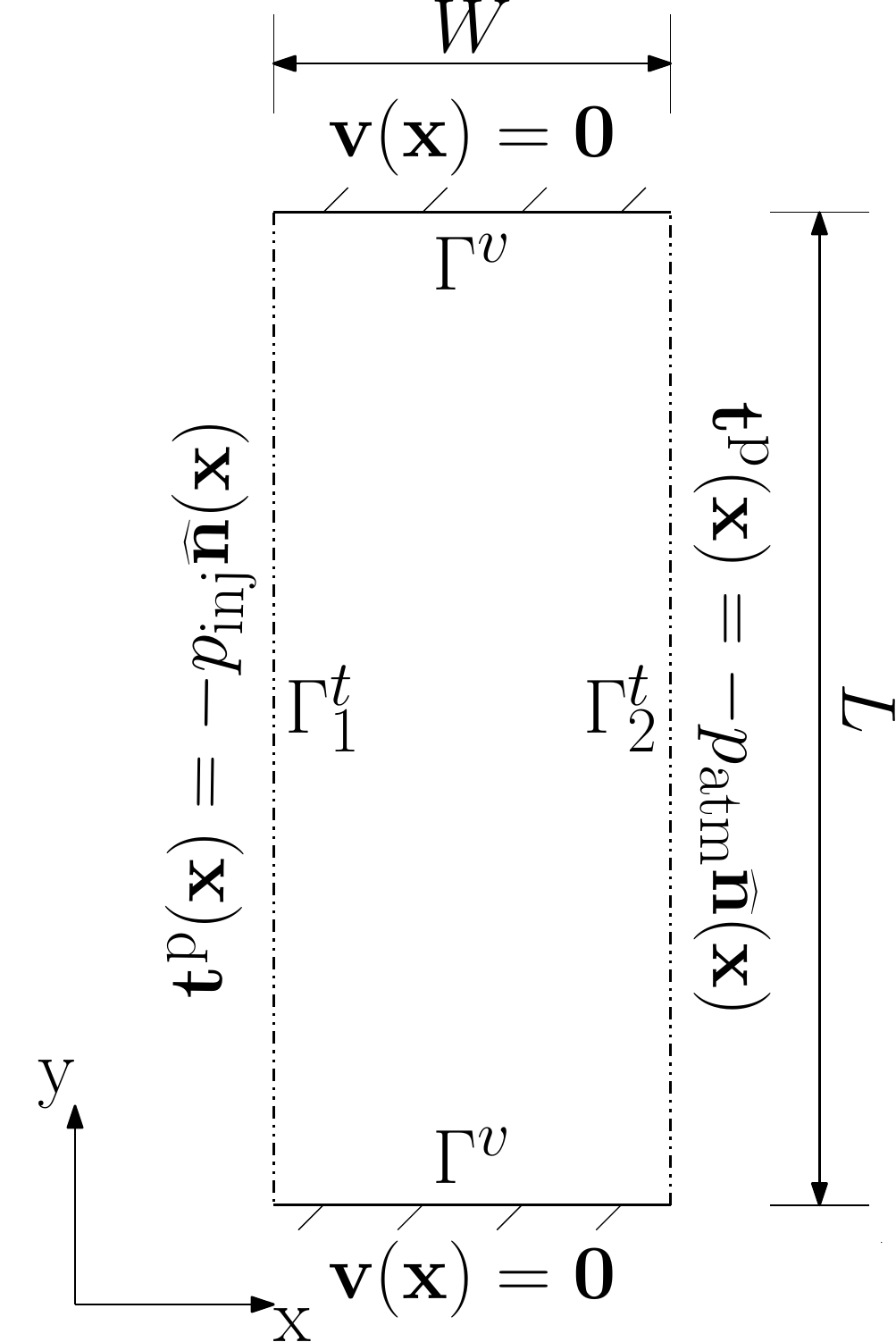}
  \caption{\textsf{Pressure slab problem:}~A pictorial 
    description of the problem. The computational domain 
    is a $W \times L$ rectangle. 
    The traction is prescribed on the left side of the 
    boundary (i.e., $\mathbf{t}^\mathrm{p}(\mathbf{x}) = 
    -p_{\mathrm{inj}} \widehat{\mathbf{n}}(\mathbf{x})$ 
    on $\Gamma^t_1$) and on the right side (i.e., 
    $\mathbf{t}^\mathrm{p}(\mathbf{x}) = -p_{\mathrm{atm}} 
    \widehat{\mathbf{n}}(\mathbf{x})$ on $\Gamma^t_2$).
    Elsewhere, homogeneous velocity is enforced.        
    \label{Fig:press_slab}}
\end{figure}

\begin{figure}
  \includegraphics[scale=0.375]
  {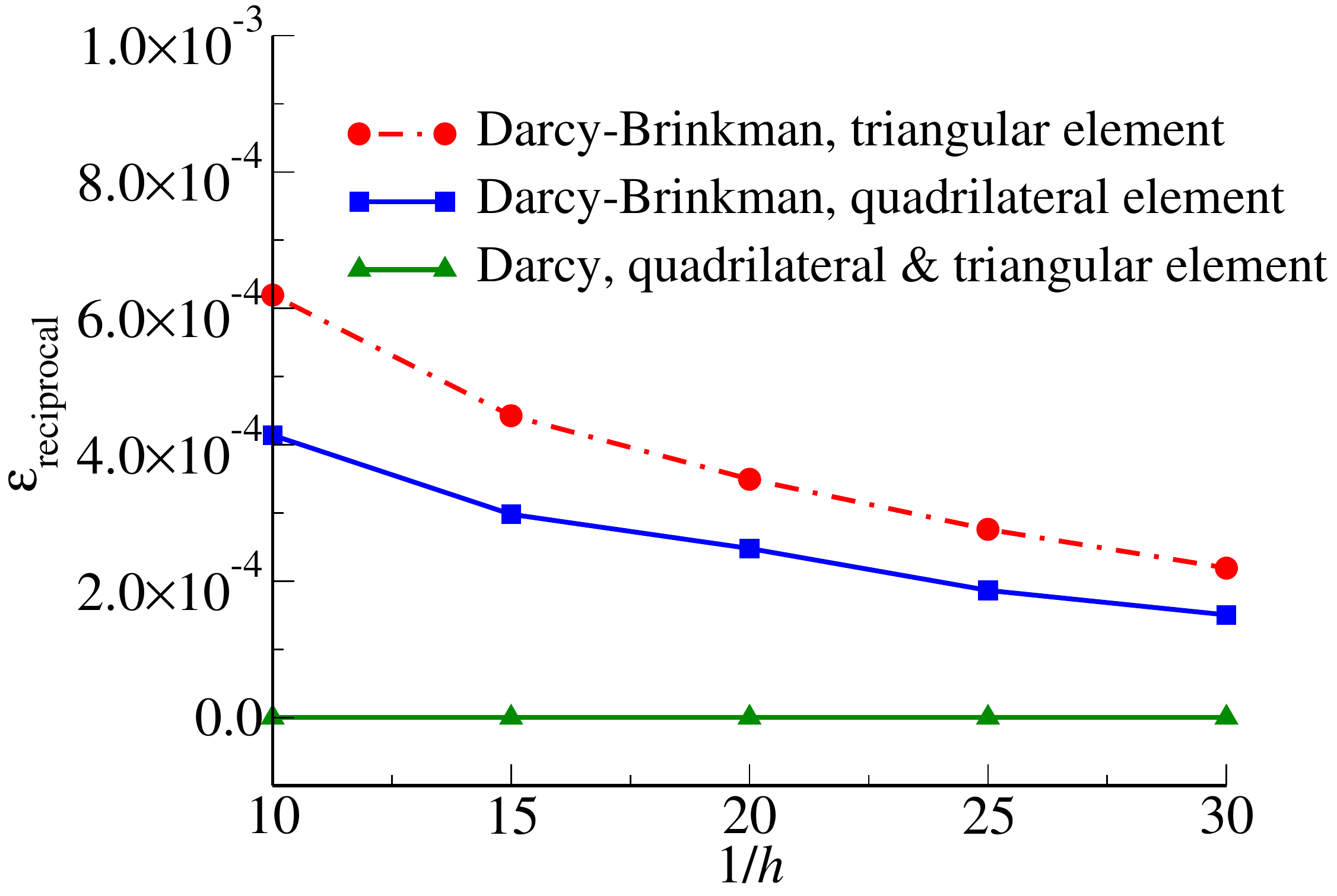}
  \caption{\textsf{Pressure slab problem:}~The figure 
    shows the variation of $\varepsilon_{\mathrm{reciprocal}}$ 
    with mesh refinement for Darcy and Darcy-Brinkman 
    equations using quadrilateral and triangular grids. 
    The parameters in this problem are provided in Table 
    \ref{Tab:Input_to_COMSOL_for_slab}. 
    One can see the error in the numerical results with 
    respect to the reciprocal relation for Darcy equations 
    is very close to zero for all meshes. The corresponding 
    error under Darcy-Brinkman equations decreases uniformly 
    with mesh refinement. 
    \label{Fig:Br_Betti_slab_no_bodyf}}
\end{figure}

\clearpage 
\newpage 

\let\thefootnote\relax\footnote{Results for 
\textsf{Pressure-driven problem}}

\begin{figure}
  \includegraphics[scale=0.35]
  {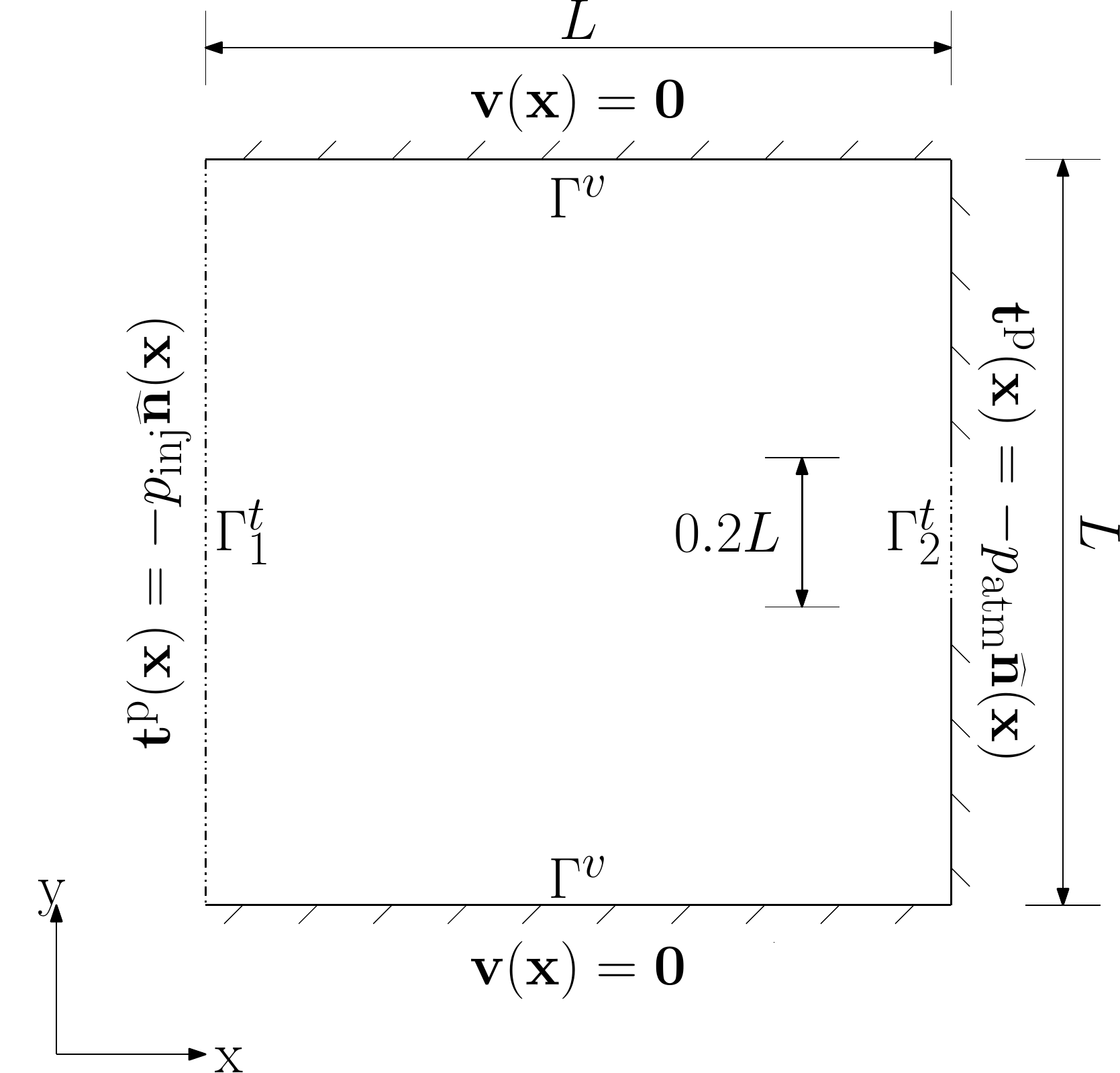}
  \caption{\textsf{Pressure-driven problem:}
  ~A pictorial description of the problem. The 
    computational domain $\Omega$ is a square 
    with $L=1$.
    The traction is prescribed on the left side of the 
    boundary (i.e., $\mathbf{t}^\mathrm{p}(\mathbf{x}) 
    = -p_{\mathrm{inj}} \widehat{\mathbf{n}}(\mathbf{x})$ 
    on $\Gamma^t_1$) and on the middle of right side 
    (i.e., $\mathbf{t}^\mathrm{p}(\mathbf{x}) = 
    -p_{\mathrm{atm}} \widehat{\mathbf{n}}(\mathbf{x})$
    on $\Gamma^t_2$).
    Elsewhere, homogeneous velocity is enforced.
    \label{Fig:press_driven}}
\end{figure}
%
\begin{figure}
	\includegraphics[scale=0.3]
	{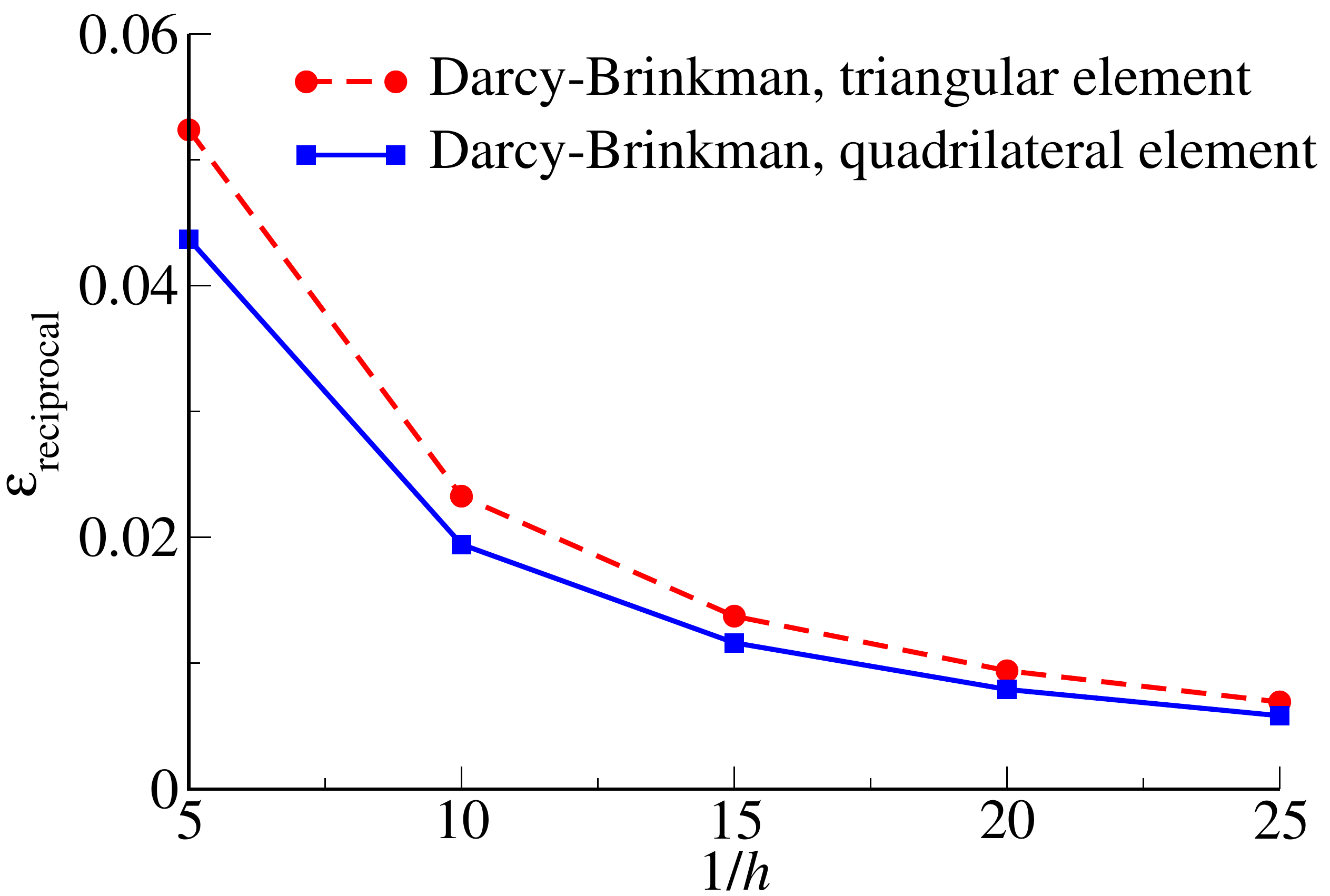}
    \caption{\textsf{Pressure-driven problem:}
    ~Figure shows the variation of $\varepsilon_
    {\mathrm{reciprocal}}$ under the Darcy-Brinkman 
    models using quadrilateral and triangular 
    finite elements. 
    The parameters used in this problem are 
    provided in Table  
    \ref{Tab:Input_to_COMSOL_for_press_driven}. 
    Clearly, the reciprocal theorem can be 
    utilized to obtain information on the 
    numerical performance of the problem that 
    does not possess an analytical solution. 
    \label{Fig:Betti_press_driven}}
\end{figure}
\begin{figure}
  \subfigure[Lid-driven cavity, regular mesh.
  $\omega_{max}=35.641$ and $\omega_{min}=0$.]{
    \includegraphics[scale=0.3,clip]{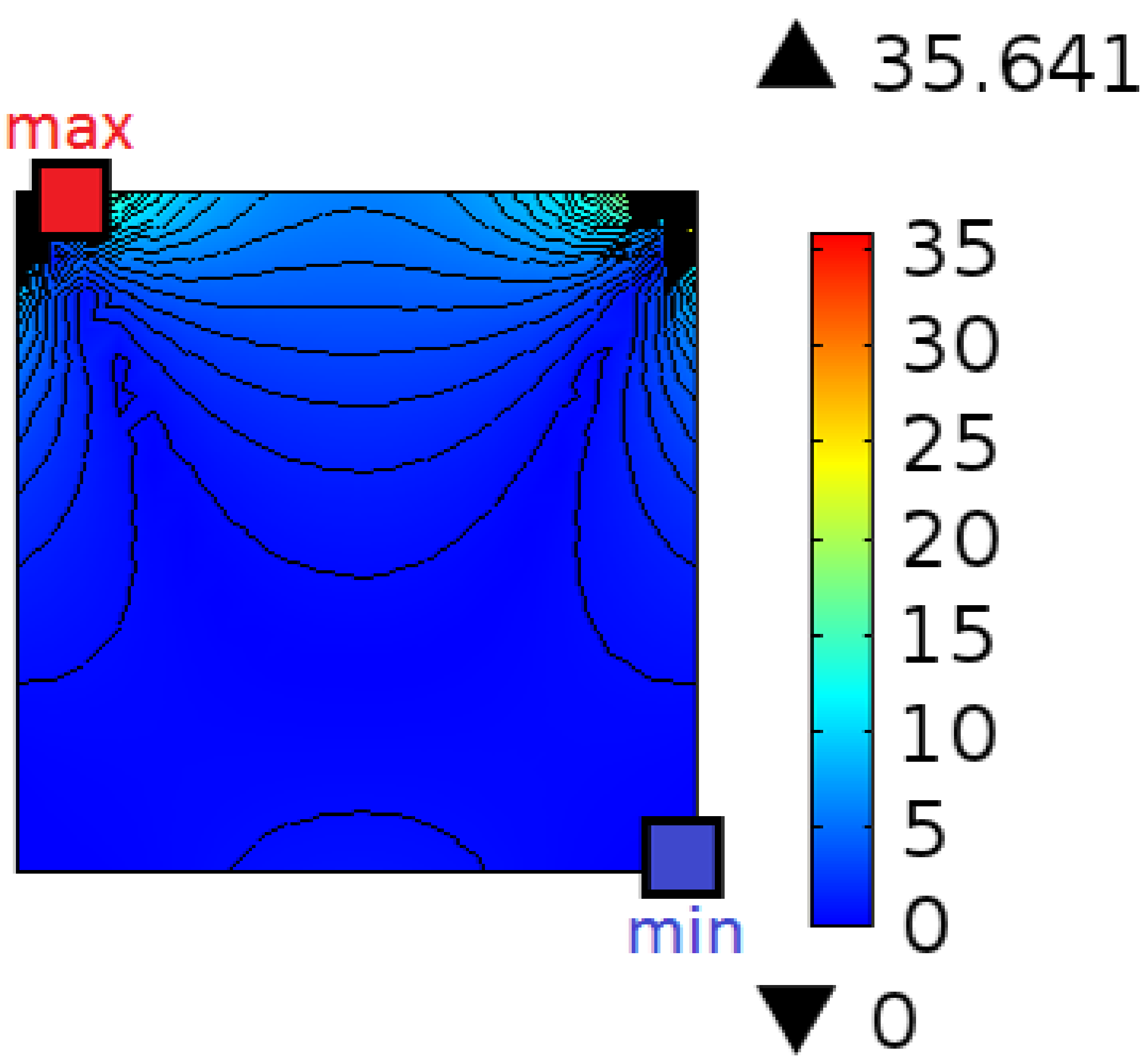}
    \label{Fig:vort_lid_cavity_regular_mesh}}
  \hspace{1.25 cm}
  \subfigure[Lid-driven cavity, adaptive mesh.
  $\omega_{max}=778$ and $\omega_{min}=0$.]{
    \includegraphics[scale=0.3,clip]{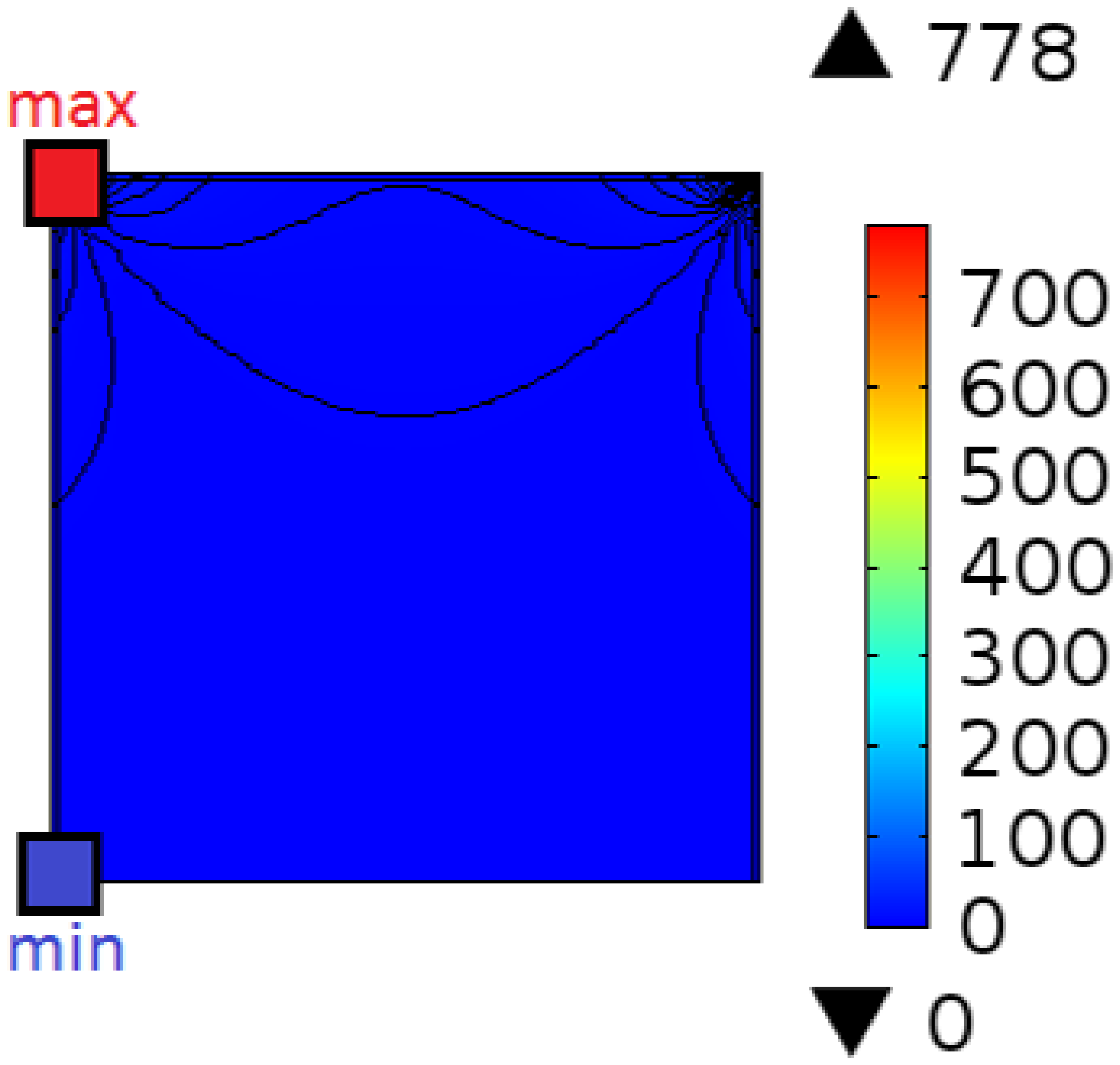}}
  \subfigure[Body force problem.
  $\omega_{max}=0.1923$ and $\omega_{min}=4.8073\times10^{-9}$.]{
    \includegraphics[scale=0.3,clip]{Figures/Vorticity/20x20_Vort_body_f.pdf}}
  \hspace{1. cm}
  \subfigure[Slab problem.
  $\omega_{max}=561.52$ and $\omega_{min}=3.1957\times10^{-5}$.]{
    \includegraphics[scale=0.292,clip]{Figures/Vorticity/20x4_Vort_slab.pdf}
    \label{Fig:vort_slab_prob}}
  \caption{The figure verifies the maximum principle 
    for the vorticity for various two-dimensional 
    problems under the Darcy-Brinkman model. We 
    employed quadrilateral elements with size 
    $1/h=20$. The numerical results corroborate 
    the theoretical predictions given in Theorem 
    \ref{Theorem:Vorticity} for all the test 
    problems we have considered. 
    \label{Fig:Metrics_vorticity_Brinkman}}
\end{figure}

\clearpage
\newpage 

\let\thefootnote\relax\footnote{Results for \textsf{Synthetic reservoir problem}}

\begin{figure}
  \includegraphics[scale=0.5,clip]{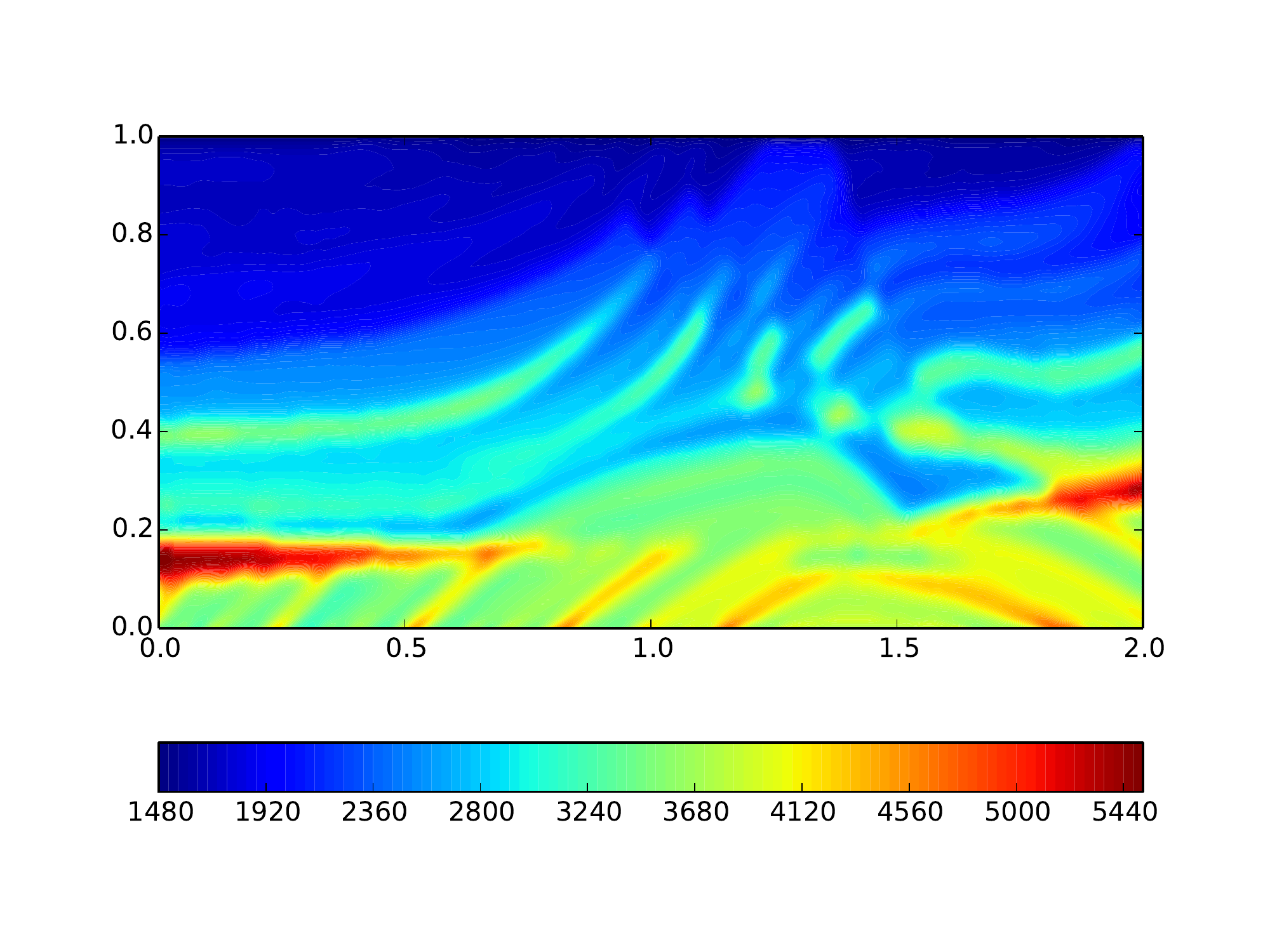}
  \caption{\textsf{Synthetic reservoir problem:}~This 
    figure shows the contours of (smooth) Marmousi 
    dataset \citep{Marmousi_Rice}. The dataset 
    provides values on a $384 \times 122$ grid 
    which we have scaled to our rectangular 
    computational domain of $L = 2$ and $H 
    = 1$. This spatially heterogeneous dataset 
    is widely used as a benchmark dataset in 
    reservoir modeling. \label{Fig:Brinkman_Marmousi}}
\end{figure}
%
\begin{figure}
  \includegraphics[scale=.75]{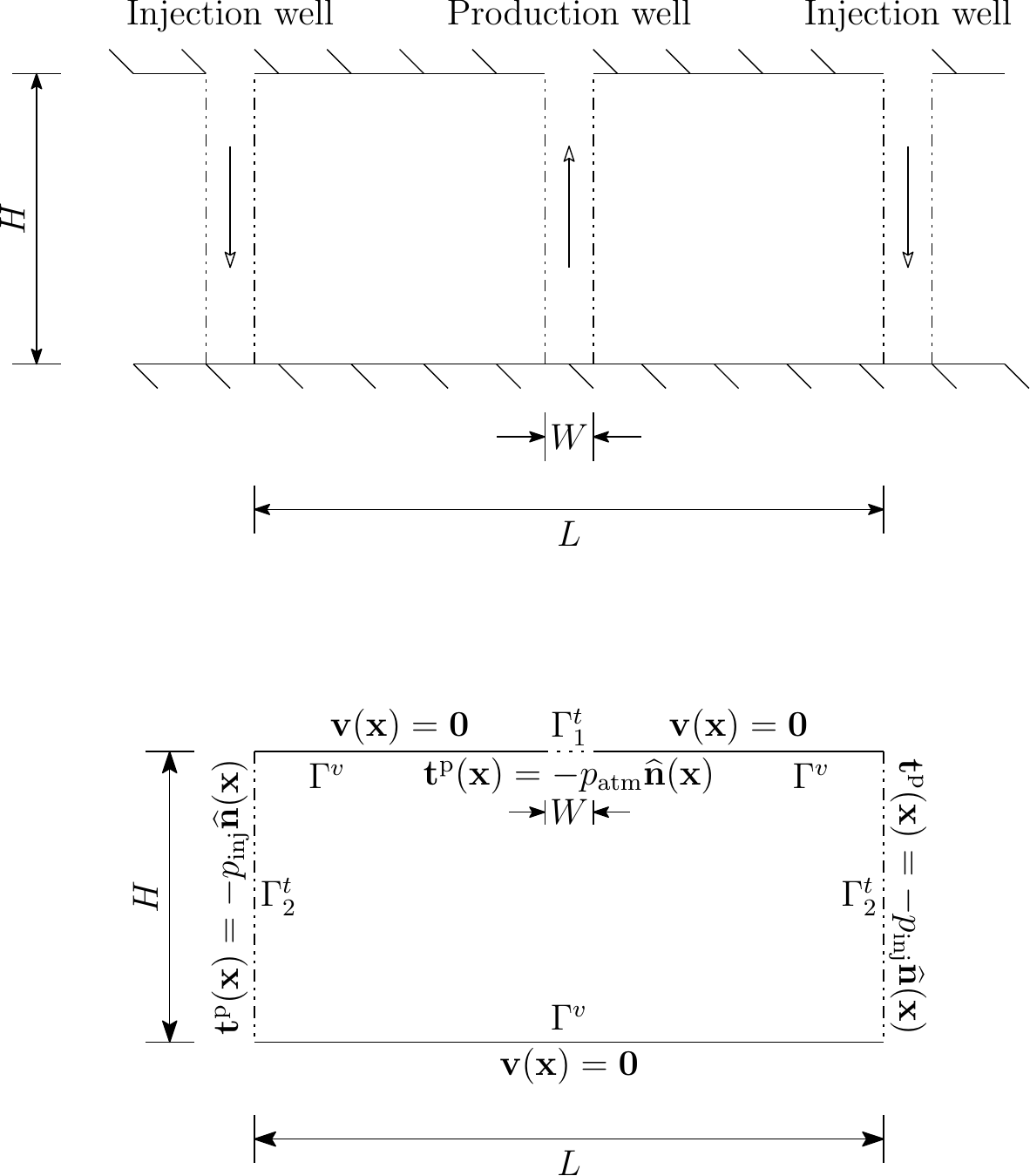}
  \caption{\textsf{Synthetic reservoir problem:}~This 
    figure provides a pictorial description of the test 
    problem. The domain of the problem is a rectangle 
    of size $H \times L$. The injection pressure is 
    prescribed on the left and right boundaries (i.e., 
    $\mathbf{t}^\mathrm{p}(\mathbf{x}) = -p_{\mathrm{inj}} 
    \widehat{\mathbf{n}}(\mathbf{x}) \; \mathrm{on} 
    \; \Gamma_2^{t}$), and the atmosphere pressure is 
    prescribed on the middle of top side (i.e., 
    $\mathbf{t}^\mathrm{p}(\mathbf{x}) = -p_{\mathrm{atm}} 
    \widehat{\mathbf{n}}(\mathbf{x}) \; \mathrm{on} 
    \; \Gamma_{1}^{t}$). The prescribed velocity on 
    the remaining parts of the boundary is zero (i.e., 
    $\mathbf{v}^{\mathrm{p}}(\mathbf{x}) = \mathbf{0}$ for 
    the Darcy-Brinkman model and $v_{n}(\mathbf{x}) = 0$ 
    for the Darcy model).\label{Fig:reservoir}}
\end{figure}
%
\begin{figure}
  \subfigure[Uniform structured mesh with $1/h=20$.]
            {\includegraphics[clip,scale=0.375]
              {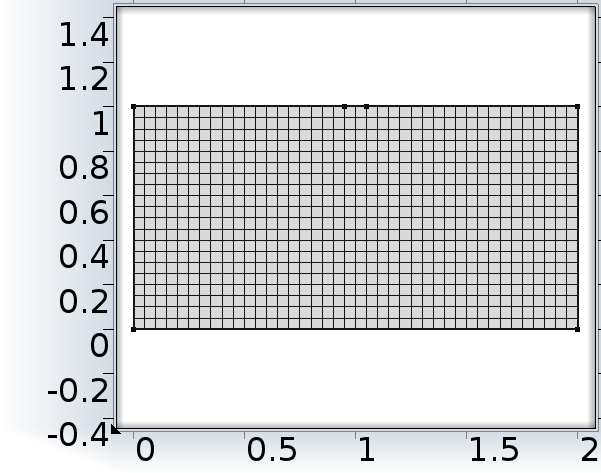}
              \label{Fig:reservoir_uni_Q9_40x20}}
            \hspace{2.cm}
            \subfigure[Adaptive mesh near the production well with $1/h=20$ 
              elsewhere.]{\includegraphics[clip,scale=0.375]
              {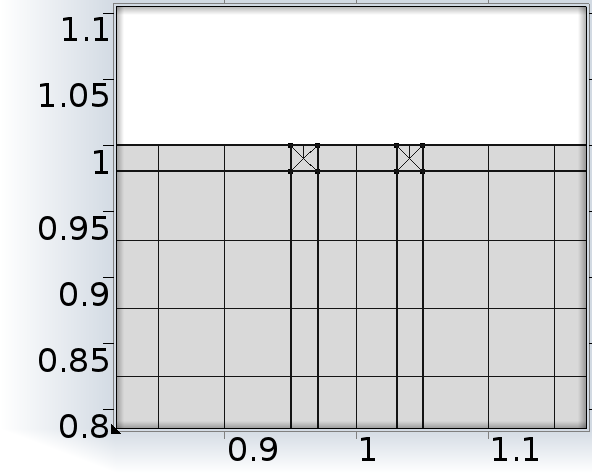}
              \label{Fig:reservoir_adapt_Q9_40x20}}
            \subfigure[Variation of $\varepsilon_{\mathrm{reciprocal}}$ under 
              a hierarchy of uniform structured meshes.]
                      {\includegraphics[scale=0.35]
                        {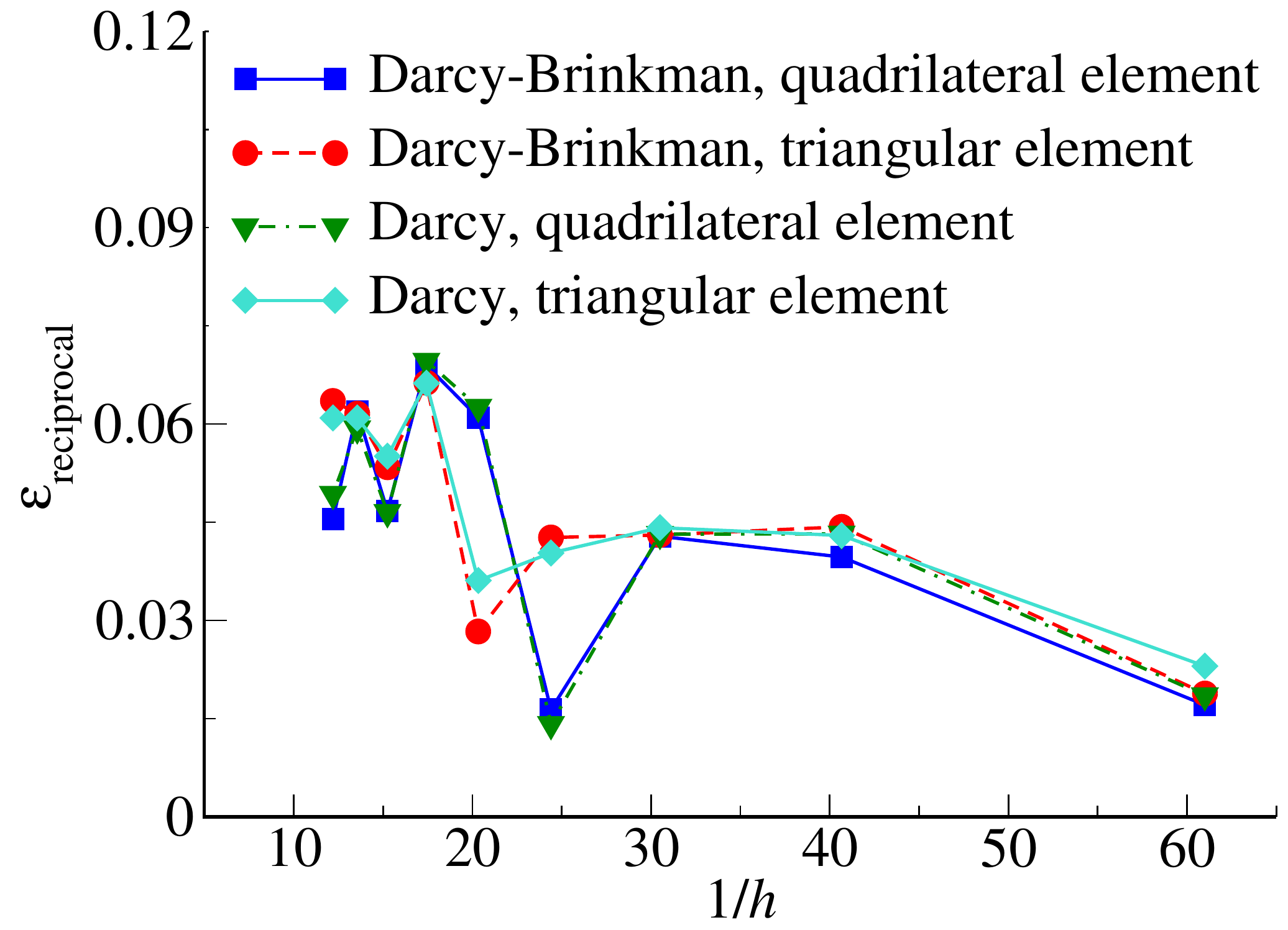}
                        \label{Fig:Betti_regular_mesh_Marm_smooth_reservoir}}
                      \subfigure[Variation of $\varepsilon_{\mathrm{reciprocal}}$ 
                        under a hierarchy of adaptive meshes.]
                                {\includegraphics[scale=0.43]
                                  {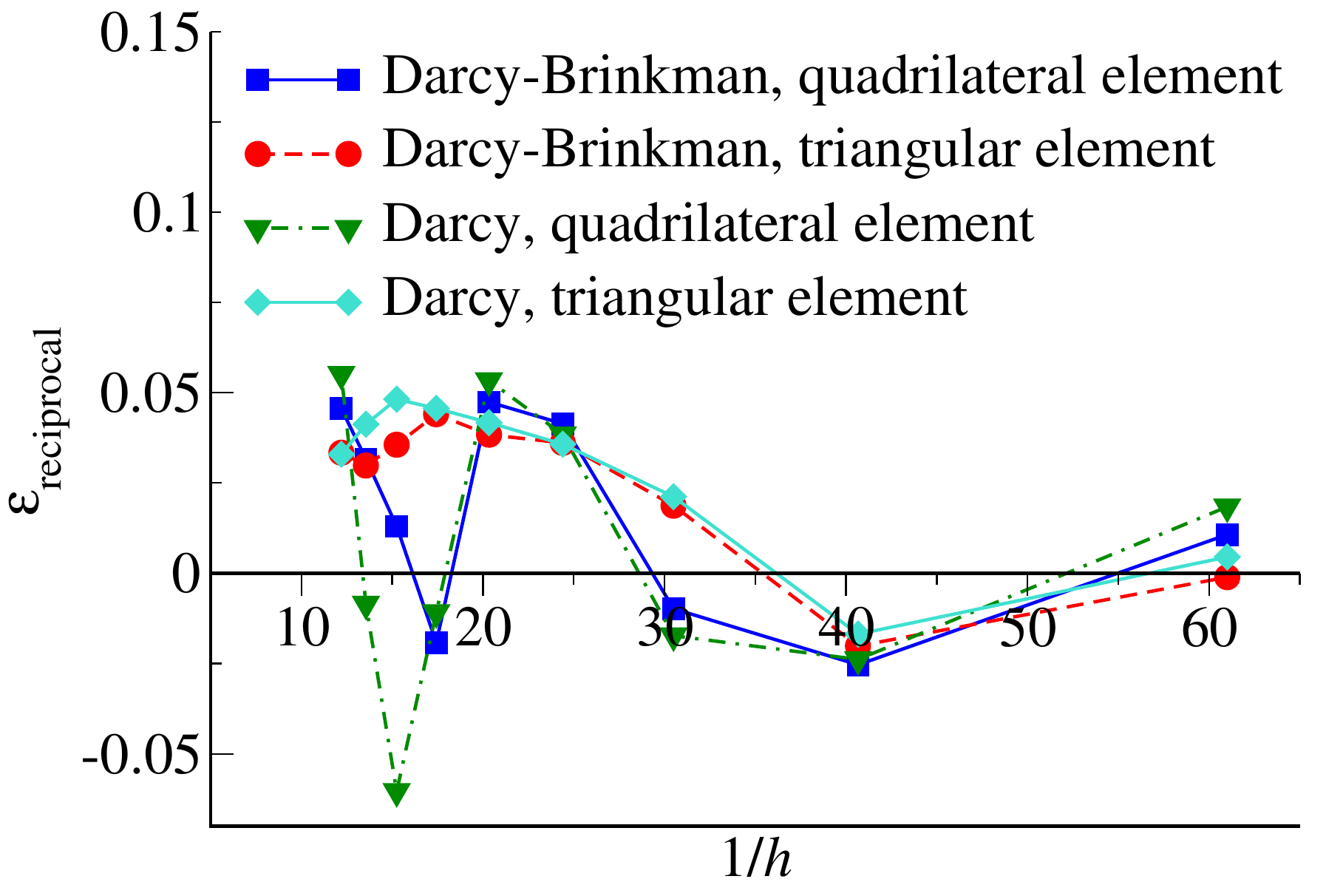}
                                  \label{Fig:Betti_adapt_mesh_Marm_smooth_reservoir}}
  \caption{\textsf{Synthetic reservoir problem:} 	
    The figures show the variation of $\varepsilon_
    {\mathrm{reciprocal}}$ with $h$ for the Darcy and 
    Darcy-Brinkman models using structured and 
    adaptive meshes.
    The parameters in this problem are provided in 
    Table \ref{Tab:Input_to_COMSOL_for_Marmousi}.
    The figures clearly show that the errors in the 
    reciprocal relation are larger under the uniform 
    structured meshes. This can be attributed to 
    pollution errors due to singularities at the 
    corners of the production wells. 
    \label{Fig:Betti_Marm_smooth_reservoir}}
\end{figure}

\begin{figure}
  {\includegraphics[scale=0.4]
    {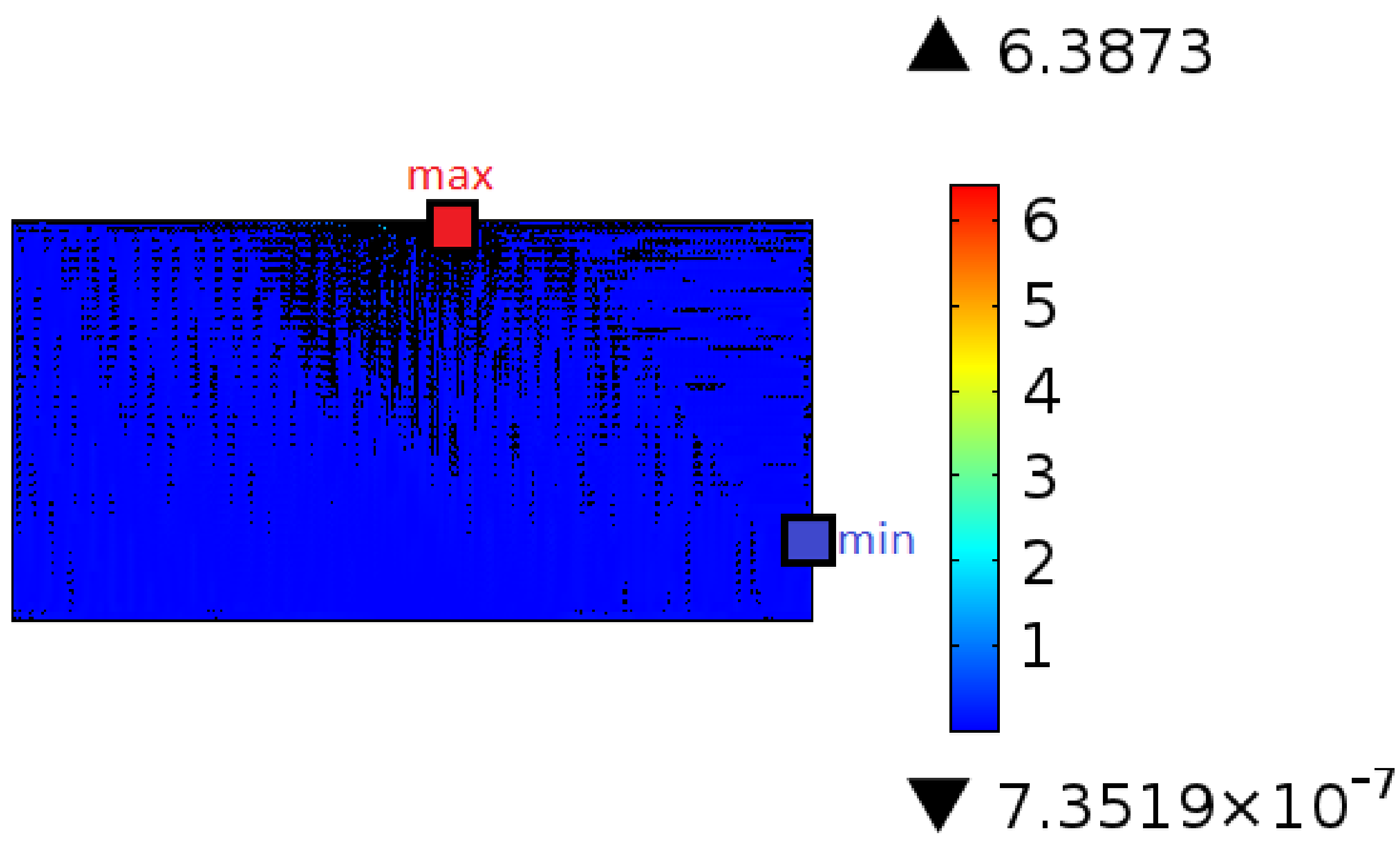}}
  \caption{\textsf{Synthetic reservoir problem:}~The 
    figure shows the magnitude of vorticity for 
    two-dimensional Darcy-Brinkman equations using 
    quadrilateral elements. The parameters used in this problem 
    are provided in Table 
    \ref{Tab:Input_to_COMSOL_for_Marmousi}.
    This figure verifies the maximum principle 
    for the vorticity for the synthetic reservoir 
    problem using Marmousi dataset.     
    \label{Fig:Vort_Marm_smooth_reservoir}}
\end{figure}

\clearpage
\newpage


\begin{figure}
  \subfigure[Quadrilateral element size of $1/h=20$ for $dt=0.01$.]{
    \includegraphics[scale=0.375,clip]{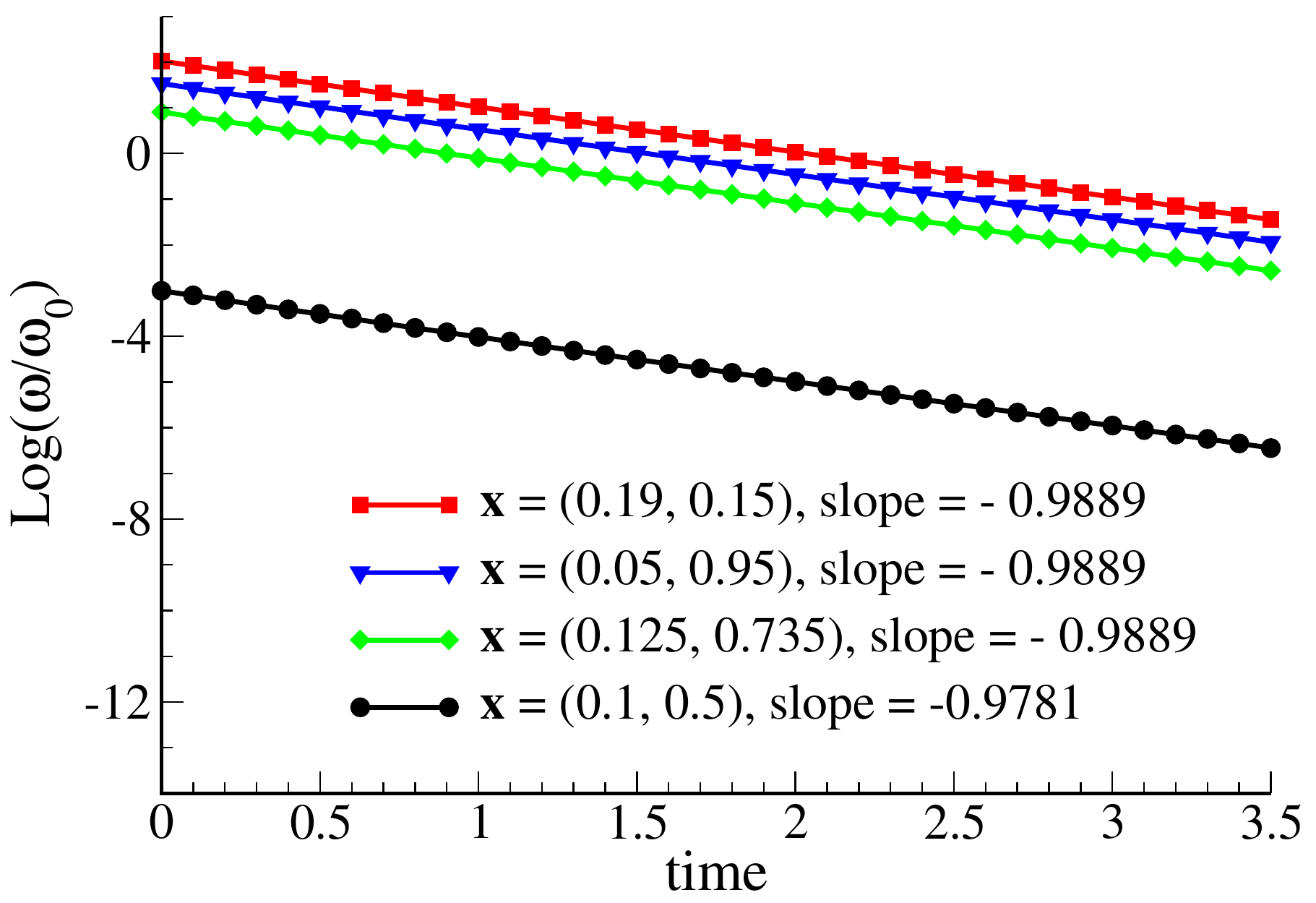}
    \label{Fig:vort_press_slab}}
  \hspace{2.0cm}
  \subfigure[Initial vorticity for quadrilateral element size of $1/h=20$.]{
    \includegraphics[scale=0.25,clip]{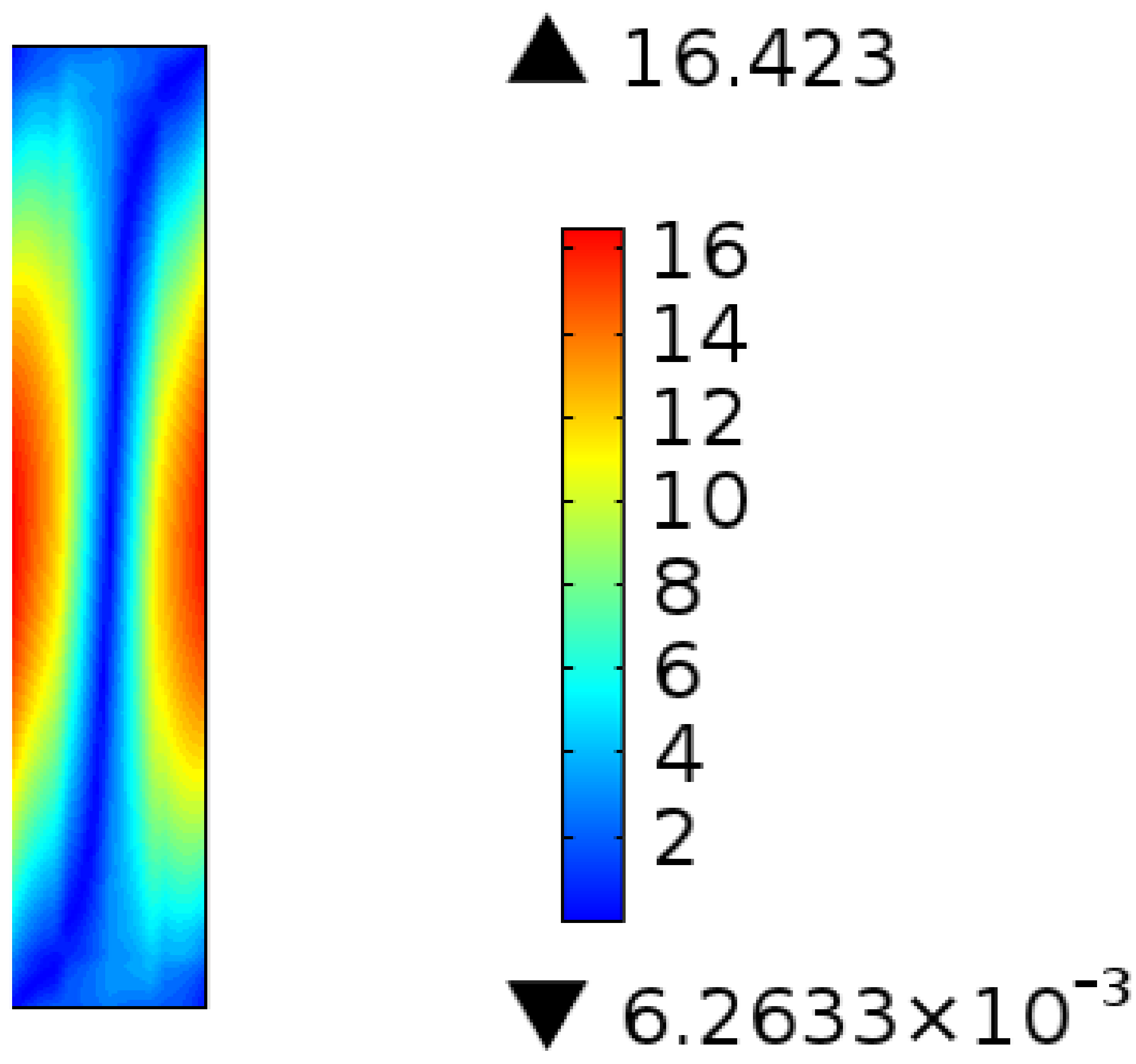}
    \label{Fig:vort_Body_F}}
  \subfigure[Mesh refinement for quadrilateral elements, at 
  $\mathbf{x}=(0.05, 0.95)$ and $dt=0.01$.]{
    \includegraphics[scale=0.325,clip]{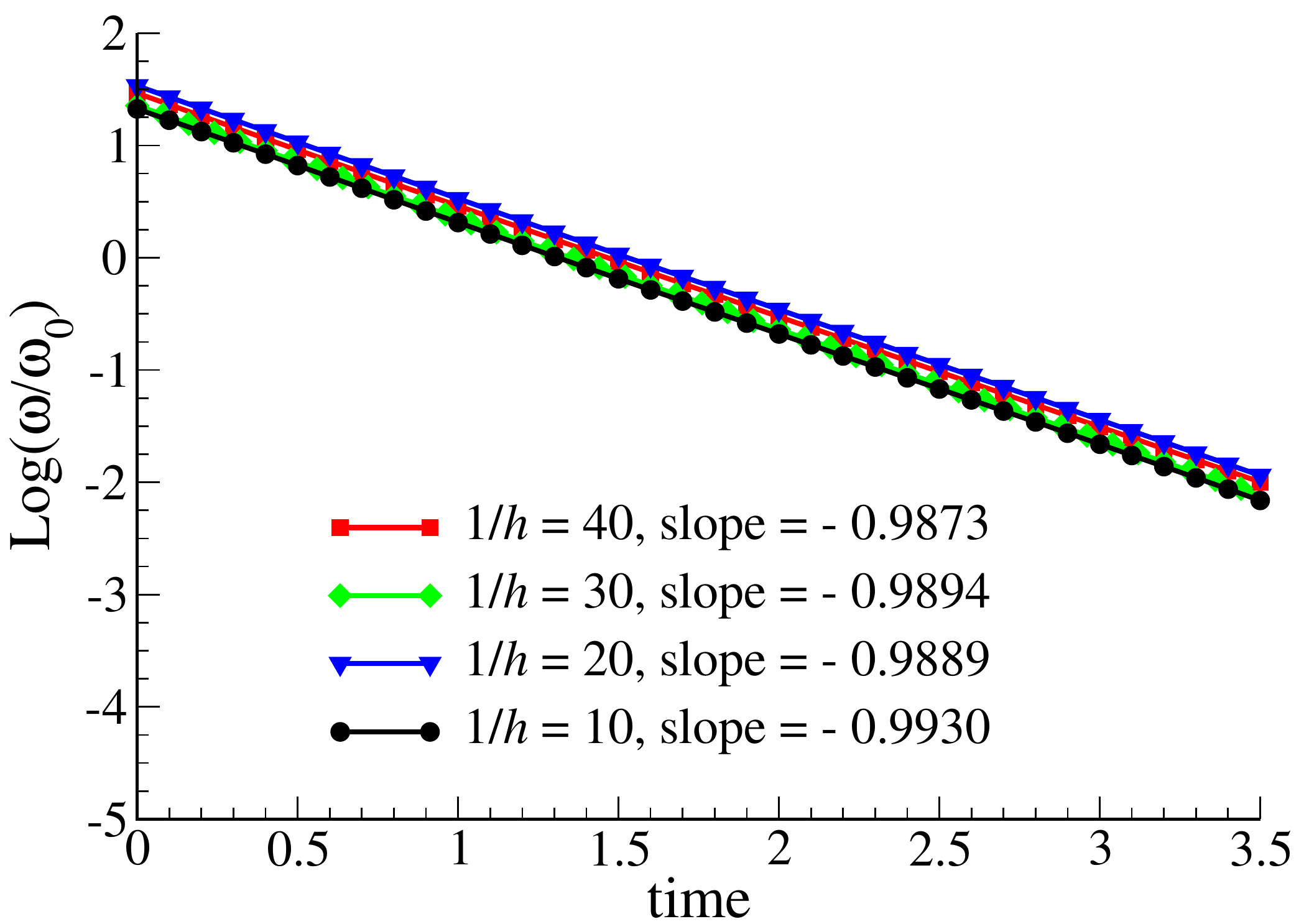}
    \label{Fig:vort_press_slab_mesh_refinement}}
    \quad
    \subfigure[Time refinement for $1/h=20$ at $\mathbf{x}=(0.15, 0.35)$.]{
    \includegraphics[scale=0.325,clip]{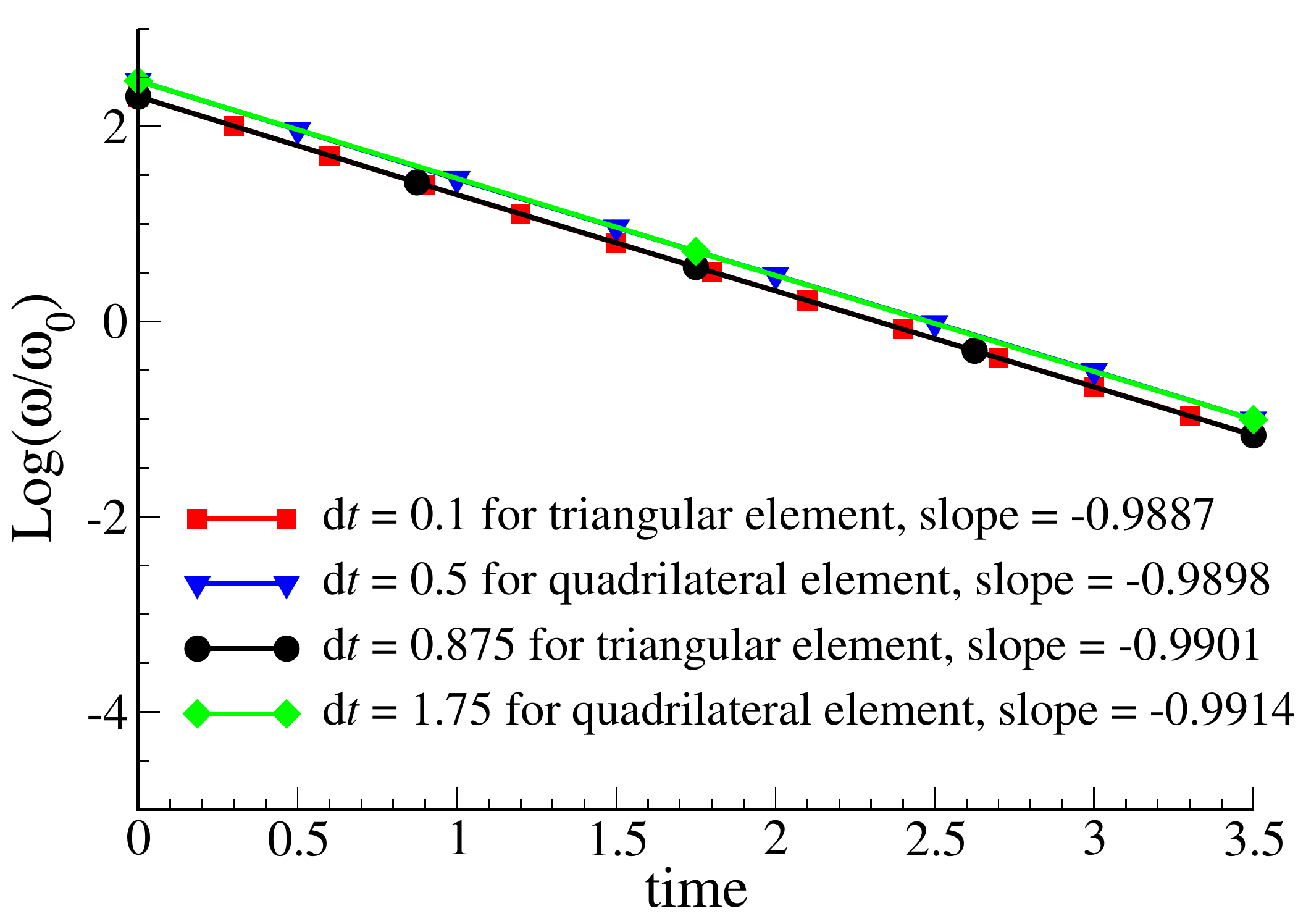}
    \label{Fig:vort_press_slab_time_step}}  
    \caption{\textsf{Pressure slab problem:}~The figure 
      verifies the theoretical results for the vorticity 
      under transient Darcy equations. The results show 
      that the slope of $\log(\frac{\omega}{\omega_0})$ 
      for various spatial points in the computational 
      domain are close to $-\frac{\alpha}{\rho}=-1$.
      The figure also shows the slope is follows 
      the theoretical prediction for various mesh 
      refinements $h$ and time-steps $dt$. The 
      non-dimensional parameters used in this 
      numerical experiment are provided in 
      Tables \ref{Tab:Input_to_COMSOL_for_slab} 
      and \ref{Tab:IC_for_slab}.
      \label{Fig:Vorticity_trans_Darcy_slab}}
\end{figure}

\begin{figure}
	\subfigure[Lid-driven cavity, structured mesh.]{
   \includegraphics[scale=0.3,clip]{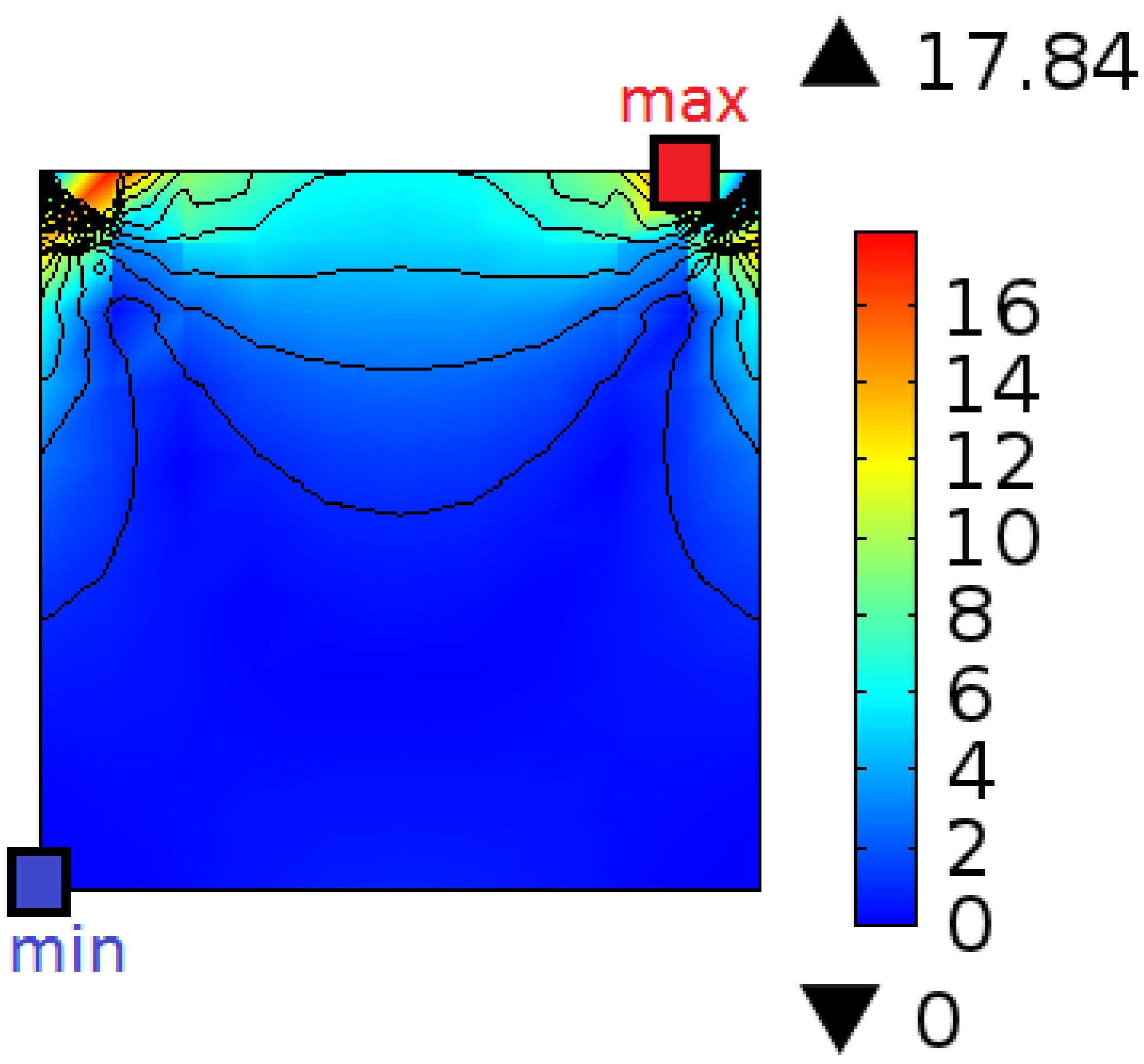}}
  \hspace{.5cm}
  \subfigure[Lid-driven cavity, adaptive mesh.]{
   \includegraphics[scale=0.3,clip]{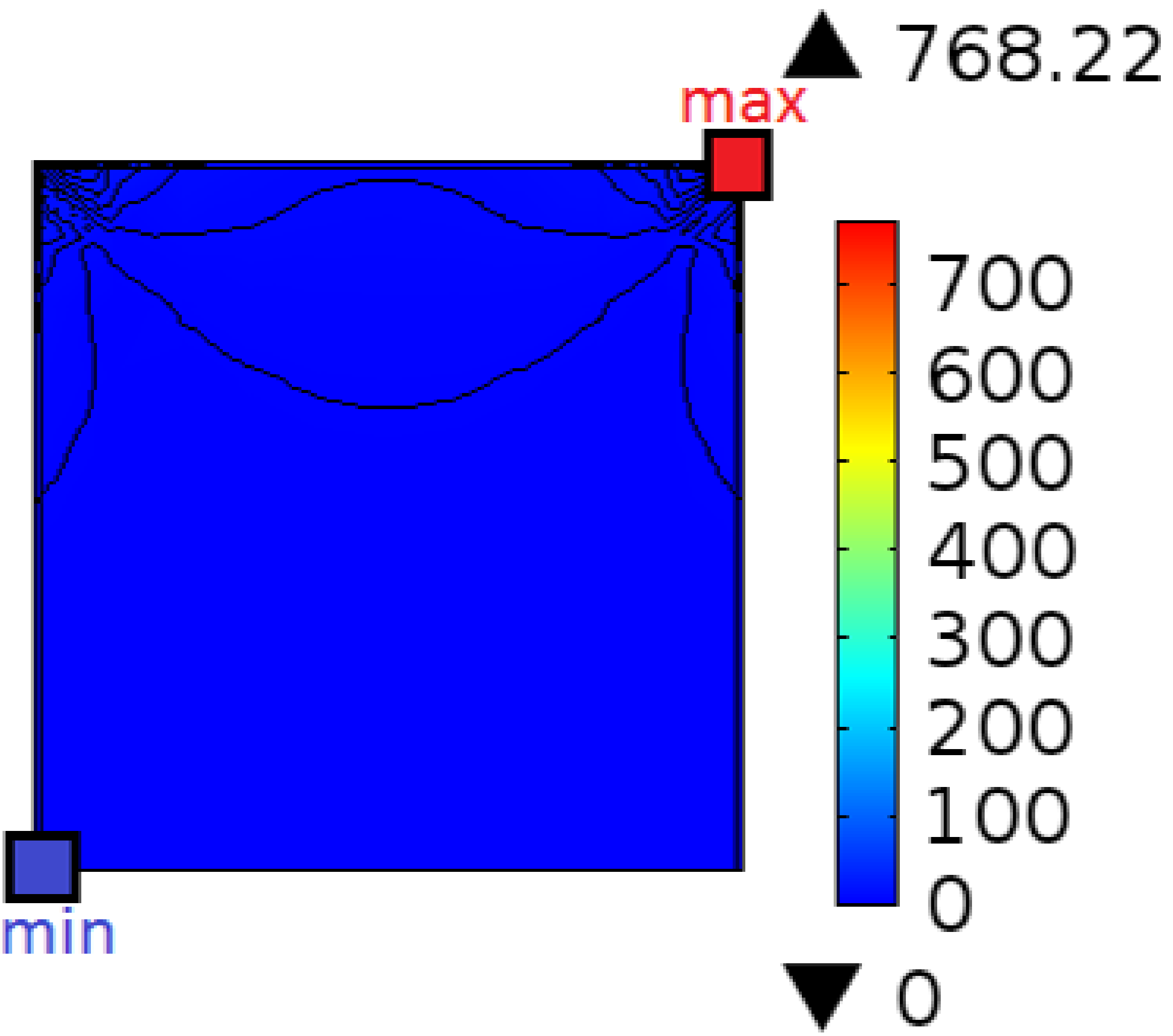}}
  \subfigure[Body force problem.]{
    \includegraphics[scale=0.3,clip]{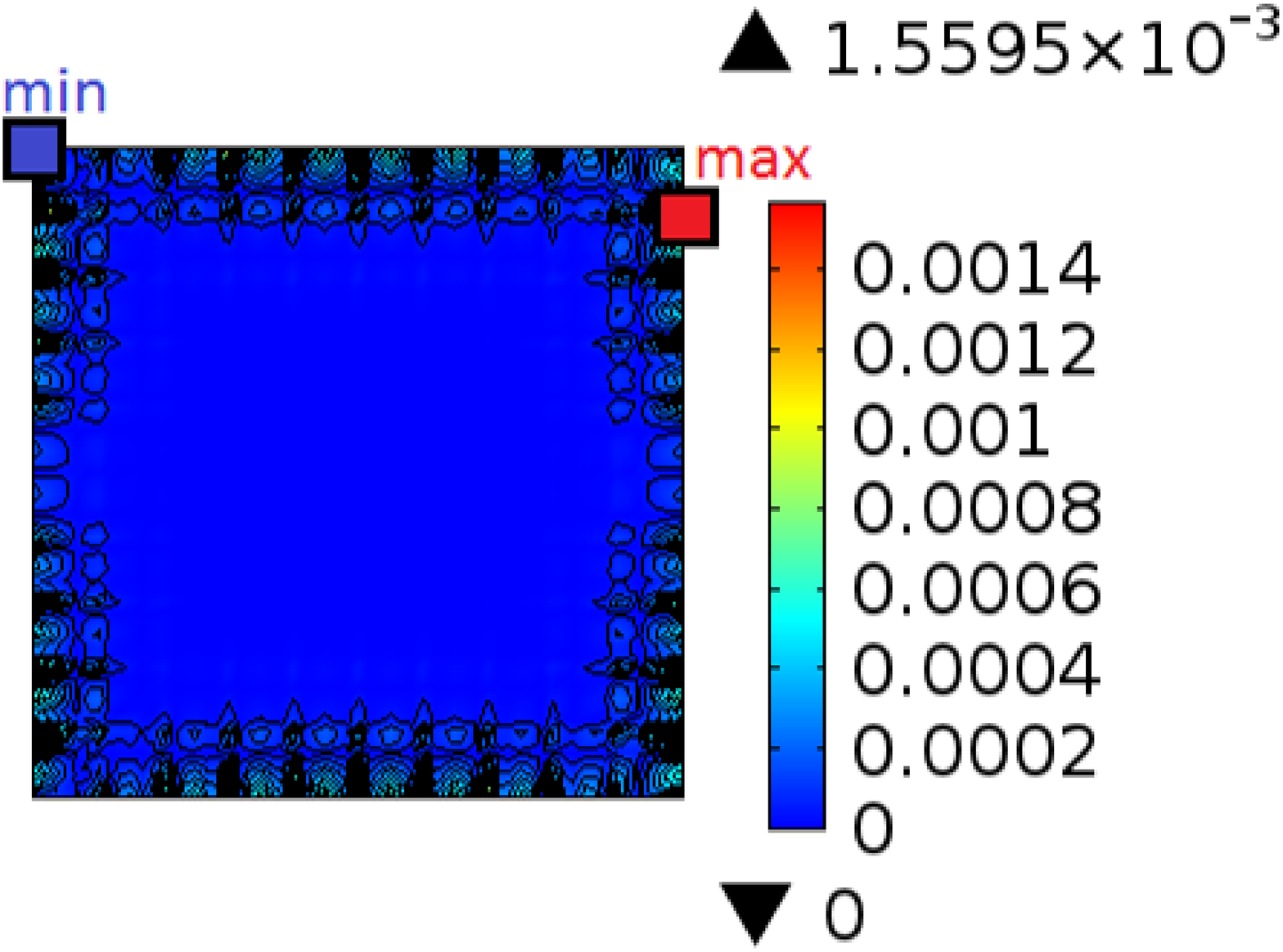}}
  \hspace{1.65cm}
  \subfigure[Pressure slab problem.]{
    \includegraphics[scale=0.24,clip]{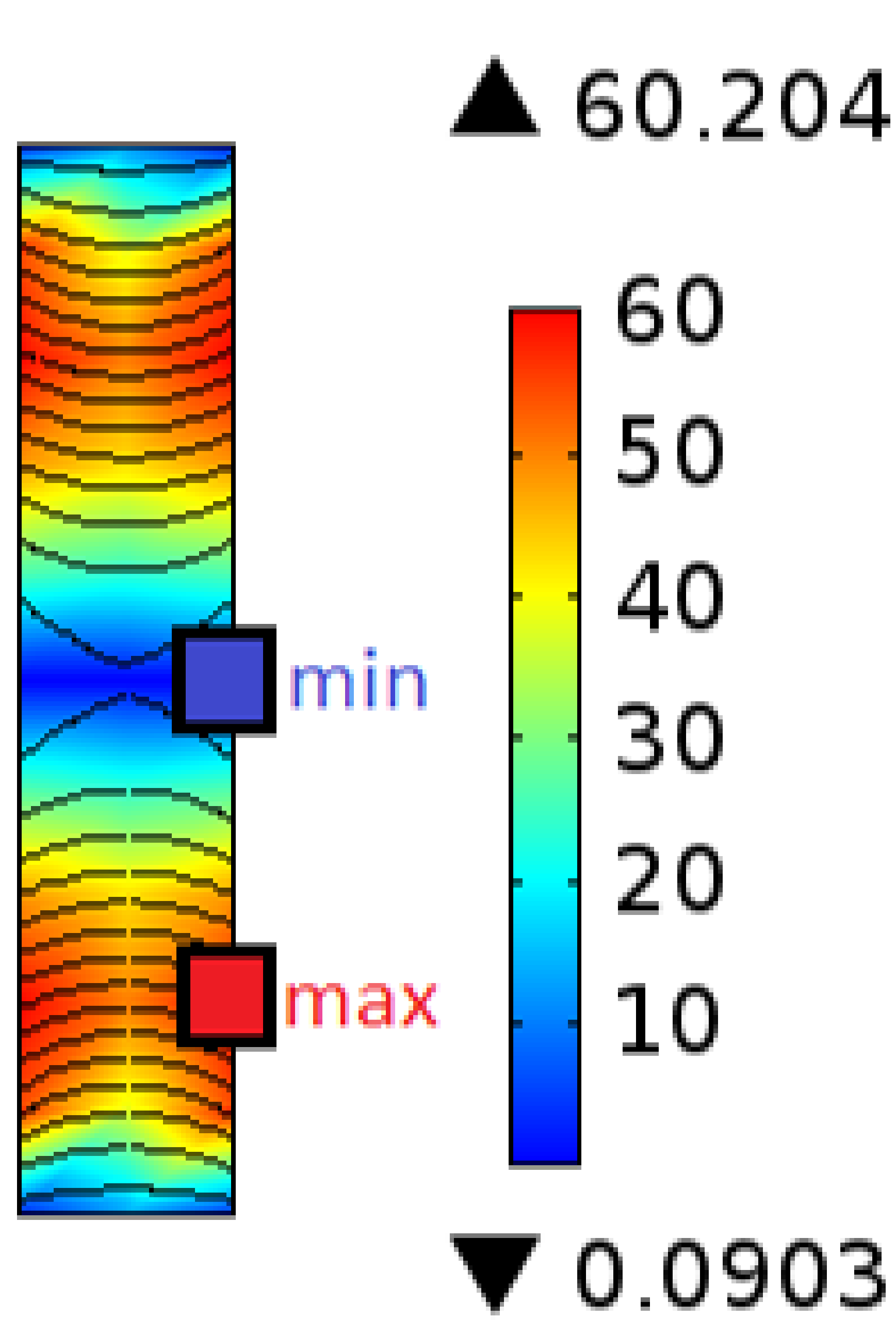}}
  \caption{The figure verifies the maximum principle 
    for the vorticity under transient Darcy-Brinkman 
    equations for various two-dimensional problems. 
    The numerical results satisfy the maximum principle 
    for all the test problems. The results are reported 
    at $t = 1$ and for time-step $dt=0.5$ (which is 
    chosen arbitrarily). The results are presented 
    for the adaptive and structured meshes based on 
    quadrilateral elements with size $1/h = 10$. We 
    employed the backward Euler time-stepping scheme 
    in the numerical experiment. 
    \label{Fig:Vorticity_trans_Brinkman}}
\end{figure}

\end{document}